\documentclass[12pt,a4paper,reqno]{amsart}
\usepackage{geometry}
\usepackage{tabls}
\usepackage{CJKutf8}
\usepackage{url}
\usepackage[utf8]{inputenc}
\usepackage{amsfonts}
\usepackage{amssymb}
\usepackage{mathrsfs}
\usepackage{amsmath}
\usepackage{amsthm}
\usepackage{shuffle}
\usepackage{amscd}
\usepackage{fancyhdr}
\usepackage{enumerate}
\usepackage{paralist}
\usepackage{graphicx}
\usepackage{cite}
\usepackage{rotating}
\usepackage{tikz}
\usetikzlibrary{matrix}

\usepackage{footmisc}
 \allowdisplaybreaks

\numberwithin{equation}{section}

\theoremstyle{plain}
\newtheorem{thm}{Theorem}[section]
\newtheorem{thm-defn}[thm]{Theorem-Definition}
\newtheorem{lem}[thm]{Lemma}

\newtheorem{cor}[thm]{Corollary}
\newtheorem{prop}[thm]{Proposition}

\theoremstyle{definition}
\newtheorem{defn}[thm]{Definition}
\newtheorem{rem}[thm]{Remark}
\newtheorem{eg}[thm]{Example}
\newtheorem{prob}[thm]{Problem}

\DeclareMathOperator{\Li}{{Li}}
\DeclareMathOperator{\bs}{{bs}}

\DeclareMathOperator{\Ree}{{Re}}

\newcommand{\R}{\mathbb{R}}

\newcommand{\N}{\mathbb{N}}
\newcommand{\Z}{\mathbb{Z}}
\newcommand{\Q}{\mathbb{Q}}
\newcommand{\CC}{\mathbb{C}}

\newcommand{\bfv}{{\boldsymbol{v}}}
\newcommand{\bfk}{{\boldsymbol{k}}}
\newcommand{\bfm}{{\boldsymbol{m}}}

\newcommand{\tbfm}{{\widetilde{\bfm}}}
\newcommand{\calZ}{{\mathcal{Z}}}
\newcommand{\calD}{{\mathcal{D}}}

\newcommand{\myrho}{{\tau}}
\newcommand{\st}{{\mathrm{st}}}

\newcommand{\tm}{{\widetilde{m}}}
\newcommand{\om}{{\omega}}
\newcommand{\must}{{\overset{\mu}{*}}}
\newcommand{\gs}{{\sigma}}
\newcommand{\gt}{{\tau}}

\begin{document}

\title{Regularized double shuffle relations of \\
$\mu$-multiple Hurwitz zeta values}

\author{Jia Li}
\author{Jianqiang zhao}


\email{1901110015@pku.edu.cn}
\address{Jia Li \\ School of Mathematical Sciences, Peking University, Beijing, China}

\email{zhaoj@ihes.fr}
\address{Jianqiang zhao \\ Department of Mathematics, The Bishop's School, La Jolla, CA 92037, USA}

\begin{abstract}
In this paper we consider a family of multiple Hurwitz zeta values with bi-indices parameterized by $\mu$ with $\Ree(\mu)>0$. These values are equipped with both the $\mu$-stuffle product from their series definition and the shuffle product from their integral expressions. We will give a detailed analysis of the two different products and discuss their regularization when the bi-indices are non-admissible. Our main goal is to prove the comparison theorem relating the two ways of regularization, which is an analog of Ihara, Kaneko, and Zagier's celebrated result concerning the regularization of multiple zeta values. We will see that the comparison formula is $\mu$-invariant, that is, it is independent of the parameter $\mu$. As applications, we will first provide a few interesting closed formulas for some multi-fold infinite sums that are closely related to the multiple zeta values of level $N$ studied by Yuan and the second author. We will also derive two sum formulas for $\mu$-double Hurwitz zeta values which generalize the sum formulas for double zeta values and double zeta star values, respectively, and a weighted sum formula which generalizes the weighted sum formulas for both double zeta values and double $T$-values simultaneously. At the end of the paper, we will propose a problem extending that for multiple zeta values first conjectured by Ihara et al.
\end{abstract}

\maketitle



\tableofcontents

\section{Introduction}
\subsection{Multiple zeta values and some variants}
 Let $\Z$ denote the integer set, $\N$ the positive integer set, $\Q$ the rational number set, and $\R$ the real
number set. For any $r\in\N$ and any index $\bfk=(k_1,\cdots,k_r)\in\N^r$, the multiple zeta values (MZVs)
and the multiple zeta star values (MZSVs) are defined by
\begin{align}\label{equ:defnMZVs}
\zeta(\bfk):= \sum_{0<n_1< \cdots<n_r}\frac{1}{n_1^{k_1} \cdots n_r^{k_r}},\quad
\zeta^{\star}(\bfk):=\sum_{1\leqslant n_1 \leqslant\cdots\leqslant n_r}\frac{1}{n_1^{k_1} \cdots n_r^{k_r}},
\end{align}
respectively. These series converge if and only if $k_r>1$ when we say $\bfk$ is admissible.
It is easy to convert these values from one type to the other and both of them have
played important and sometimes unexpected roles in the studies of many subjects in mathematics
and physics alike, such as the knot theory, motives, modular forms, and the Feynman integrals,
to name just a few. See, for example, the two books \cite{bj,ZhaoBook} and the references therein.

Recently, by restricting the summation indices $n_j$'s in \eqref{equ:defnMZVs} to even or odd numbers,
Hoffman \cite{Ho}, Kaneko and Tsumura \cite{KanekoTs2020}, and Xu and the second author \cite{XuZhao2020a}
have defined a few different variations of MZVs. For example, Hoffman's multiple $t$-values and
multiple $t$-star values (MtSVs) are defined as follows:
\begin{align*}
t(\bfk):= \sum_{ 0<n_1< \cdots<n_r,\, 2\nmid n_i}\frac{2^r}{n_1^{k_1} \cdots n_r^{k_r}},\quad
t^{\star}(\bfk):=\sum_{ 1\leqslant n_1\leqslant \cdots\leqslant n_r,\, 2\nmid n_i}\frac{2^r}{n_1^{k_1} \cdots n_r^{k_r}},
\end{align*}
where $\bfk=(k_1,\cdots,k_r)\in\N^r$ and $k_r>1$. Note that there exist different conventions to define such variations
(e.g., see \cite{KanekoTs2020,Zhao2023Nov}) where the 2-power might be omitted. In these previous works, many
interesting results and insightful relations have been discovered and some conjectures have been proposed.
It is not hard to see that all these values can be expressed as $\Q$-linear combinations of Euler sums
defined as follows. For $s_1,\dots,s_r\in\N$ and $\varepsilon_1,\dots,\varepsilon_r=\pm 1$, we define the \emph{Euler sums}
\begin{equation}\label{equ:EulerSums}
\zeta\binom{k_1,\cdots,k_r}{\varepsilon_1,\cdots,\varepsilon_r}
   := \sum_{0<n_1<\cdots<n_r}\;\prod_{j=1}^r  \frac{\varepsilon_j^{n_j}}{n_j^{s_j}}.
\end{equation}
To save space, if $\varepsilon_j=-1$ then $\overline s_j$ will be used and if a substring $S$ repeats $n$
times in the list then $\{S\}_n$ will be used. For example, $\zeta(\{3\}_n)=8^n\zeta(\{1,\bar2\}_n)$ for all $n\in\N$
(see \cite{Zhao2010a}).

\subsection{$\mu$-multiple Hurwitz zeta values}
In this article, we will consider a further generalization of MZVs as follows.
Fix a parameter $\mu$ with $\Ree(\mu)>0$. For any admissible $\bfk=(k_1,\cdots,k_r)\in\N^r$ and $\bfm=(m_1,\cdots,m_r)\in(\Q_{>0})^r$,
we define the \emph{$\mu$-multiple Hurwitz zeta value} ($\mu$-MHZV) by
$$
\zeta^{\mu}(\bfk;\bfm):=\sum_{0\leqslant n_1\leqslant\cdots\leqslant n_r}\frac{\mu^r}{(n_1\mu+m_1)^{k_1}\cdots(n_r\mu+m_r)^{k_r}}.
$$
It is easy to see that for all admissible $\bfk\in\N^r$
\begin{alignat*}{4}
\zeta(\bfk)=&\, \zeta^{\mu=1}(\bfk;1,2,\cdots,r),\quad  &
\zeta^{\star}(\bfk)=  &\, \zeta^{\mu=1}(\bfk;1,1,\cdots,1),\\
t(\bfk)=&\, \zeta^{\mu=2}(\bfk;1,3,\cdots,2r-1),\quad   &
t^{\star}(\bfk)= &\, \zeta^{\mu=2}(\bfk;1,1,\cdots,1).
\end{alignat*}
Then we have a canonical embedding from MZSVs (resp.\ MtSVs) to 1-MHZVs (resp.\ 2-MHZVs) which agrees with the algebraic
structure, respectively. We use the name $\mu$-MHZVs because they can be regarded as special cases of the more general
multiple Hurwitz zeta values defined by
\begin{equation*}
\zeta^H(k_1,\cdots,k_r;a_1,\cdots,a_r):=\sum_{0<n_1<n_2<\cdots<n_r}\frac{1}{(n_1+a_1)^{k_1}\cdots (n_r+a_r)^{k_r}},
\end{equation*}
where $k_j\in\N$ and $a_j\in\CC$ with $\Ree(a_j)>0$ for all $j$. Many analytic properties and some arithmetic
properties of multiple Hurwitz zeta values have been investigated, see
\cite{KelliherMa2008,Essouabri1997,EssouabriMaTs2011,EssouabriMa2020}.
However, very little has been done on their algebraic structures generalizing that for the MZVs.
Our primary goal of
this paper is to initiate such a study, at least when they have the special form of $\mu$-MHZVs.

\begin{defn}
A \emph{bi-index}
$$
(\bfk;\bfm)=(k_1,\cdots,k_r;m_1,\cdots,m_r)\in\N^r\times\Q^r
$$
is called \emph{positive} if $m_1,\cdots,m_r>0$, and is called \emph{admissible} if it is positive and, in addition,
satisfies $k_r>1$. We sometimes also say $\bfk$ is \emph{admissible} if $k_r>1$ and $\bfm$ is positive.
The number $r$ is called the \emph{length} of $\bfk$, denoted by $\ell(\bfk)=r.$
Furthermore, the sum $|\bfk|:=k_1+\cdots+k_r$ is called the \emph{weight} of $\bfk$.
By convention, the empty bi-index (when $r=0$) will also be considered admissible and we set $\zeta^{\mu}(\emptyset):=1$.
\end{defn}

By the above notation, the $\mu$-MHZV $\zeta^{\mu}(\bfk;\bfm)$ converges
if  $(\bfk;\bfm)$ is admissible. Furthermore, let
\begin{align*}
(\bfk;\bfm)&\, =(k_1,\cdots,k_{r-1},1;m_1,\cdots,m_{r-1},m_r),\\
(\bfk;\tbfm)&\, =(k_1,\cdots,k_{r-1},1;m_1,\cdots,m_{r-1},\tm_r)
\end{align*}
be a pair of non-admissible positive bi-indices with $m_r\ne\tm_r$, in which case
we call $(\bfk;\bfm;\tbfm)$ a \emph{special triple}. Then the value
\begin{align*}
&\calD^{\mu}(\bfk;\bfm;\tbfm)\\
&:=(\tm_r-m_r)\sum_{0\leqslant n_1\leqslant\cdots\leqslant n_r}\frac{\mu^r}{(n_1\mu+m_1)^{k_1}\cdots(n_{r-1}\mu+m_{r-1})^{k_{r-1}}(n_r\mu+m_r)(n_r\mu+\tm_r)}
\end{align*}
is called a \emph{quasi-$\mu$-multiple Hurwitz zeta value} (quasi-$\mu$-MHZV),
which clearly always converges. Let $\calZ^{\mu}$ denote the $\Q[\mu]$-span of all (quasi-)$\mu$-MHZVs, that is,
$$\calZ^{\mu}:=
\left\langle \left.
\aligned
& 1,\zeta^{\mu}(\bfk;\bfm),\\
& \calD^{\mu}(\bfk_1;\bfm_1;\tbfm_1)
\endaligned \right|
\aligned
& (\bfk;\bfm): \text{ admissible}, \\
& (\bfk_1;\bfm_1;\tbfm_1): \text{ special triple}
\endaligned \right\rangle_{\Q[\mu]}.$$
One of our main goals is to explore a couple of useful algebraic structures on $\calZ^{\mu}$.

\subsection{$\mu$-stuffle and shuffle products and regularization}
It is well-known that MZVs are equipped with a stuffle product due to their series representations. It is
natural to expect that this product structure can be extended to the $\mu$-stuffle product on $\calZ^{\mu}$ --
the algebra of $\mu$-MHZVs -- using the series definition of $\mu$-MHZVs, see Theorem~\ref{thm:ddog} for the details.

Further, for every classical MZV, we have an integral representation (first discovered by Kontsevich) as follows:
$$
\zeta(\bfk)=\zeta(k_1,\cdots,k_r)=\int_{\Delta^{|\bfk|}}\omega_{\bfk}
=\int_0^1 \left(\frac{dt}{t}\right)^{k_1-1}\frac{dt}{1-t} \cdots \left(\frac{dt}{t}\right)^{k_r-1}\frac{dt}{1-t}.
$$
See \cite[Ch.~1]{bj} for the simplex notation and \cite[Ch.~2]{ZhaoBook} for the iterated integral notation
(see \S\ref{sec:intRep} for more details).
Using this representation, we can obtain another product on $\calZ$ -- the algebra of MZVs -- called the shuffle product. For the $\mu$-MHZVs, we have similar representations given below.

\begin{thm}\label{cmzv}(Integral representation)
Let $(\bfk;\bfm)$ be an admissible bi-index. Then we have the following integral representation
for $\mu$-MHZVs:
$$\zeta^{\mu}(\bfk;\bfm)=\int_{\Delta^{|\bfk|}}\omega^{\mu}_{(\bfk;\bfm)}.$$
For quasi-$\mu$-MHZVs, we have for any special triple $(\bfk;\bfm;\tbfm)$
$$\calD^{\mu}(\bfk;\bfm;\tbfm)=\int_{\Delta^{|\bfk|}}\left(\omega^{\mu}_{(\bfk;\bfm)}-\omega^{\mu}_{(\bfk;\tbfm)}\right).$$
\end{thm}

We know that the multiplication rules on the MZV space
obtained by the stuffle product and the shuffle product are generally different, which yield the
famous double shuffle relations which in turn provide one of the main tools to generate linear dependencies among MZVs.
In fact, to conjecturally generate all possible $\Q$-linear relations, it is necessary to
regularize these values in two different ways according to the product structures to use,
as done by Ihara, Kaneko and Zagier. Moreover, they proved a comparison theorem relating
these two ways of regularization \cite[Theorem~2]{ikz}. We now extend this to the $\mu$-MHZVs.

\begin{thm}\label{tseh}
\emph{(The $\mu$-stuffle regularization homomorphism)}
Let $Y=\{y_{k,m}|k\in\N,m\in\Q_{>0}\}$ be an infinite alphabet and $\Q[\mu]\langle Y\rangle$ the non-commutative polynomial
algebra on words over $Y$. For any given $m\in\Q_{>0}$, there exists a unique $\Q[\mu]$-linear map
$$
\zeta_{*,m}^{\mu,T}:(\Q[\mu]\langle Y\rangle,\must)\longrightarrow(\calZ^{\mu},\cdot)[T]
$$
satisfying
$$\begin{cases}
\zeta_{*,m}^{\mu,T}(w)=\zeta_{*}^{\mu}(w)\qquad&\forall w\in\Q[\mu]\langle Y\rangle^0;\\
\zeta_{*,m}^{\mu,T}(y_{1,m})=T;\\
\zeta_{*,m}^{\mu,T}(v\stackrel{\mu}{*}w)=\zeta_{*,m}^{\mu,T}(v)\cdot\zeta_{*,m}^{\mu,T}(w)&\forall v,w\in\Q[\mu]\langle Y\rangle,
\end{cases}$$
where $\Q[\mu]\langle Y\rangle^0$ is the sub-algebra of $\Q[\mu]\langle Y\rangle$ generated by (quasi-)admissible words
(see Prop.~\ref{prop:Y0subalgebra}).
\end{thm}
Similarly, we also have

\begin{thm}\label{mseh}
\emph{(The shuffle regularization homomorphism)}
Let $X^{\mu}=\{x,y^{\mu}_m|m\in\Q\}$ be an infinite alphabet and $\Q[\mu]\langle X^{\mu}\rangle$ the non-commutative
polynomial $\Q[\mu]$-algebra on words over $X^{\mu}$.
For any $m\in\Q_{>0}$, there exists a unique $\Q[\mu]$-linear map
$$
\zeta_{\shuffle,m}^{\mu,T}:(\Q[\mu]\langle X^{\mu}\rangle^1,\shuffle)\longrightarrow(\calZ^{\mu},\cdot)[T]
$$
satisfying
$$\begin{cases}
\zeta_{\shuffle,m}^{\mu,T}(w)=\zeta_{\shuffle}^{\mu}(w)\qquad&\forall w\in\Q[\mu]\langle X^{\mu}\rangle^0;\\
\zeta_{\shuffle,m}^{\mu,T}(y^{\mu}_{m})=T;\\
\zeta_{\shuffle,m}^{\mu,T}(v\shuffle w)=\zeta_{\shuffle,m}^{\mu,T}(v)\cdot\zeta_{\shuffle,m}^{\mu,T}(w)&\forall v,w\in\Q[\mu]\langle X^{\mu}\rangle^1,
\end{cases}$$
where $\Q[\mu]\langle X\rangle^0$ is the sub-algebra of $\Q[\mu]\langle X\rangle$ generated by (quasi-)admissible words
(see Prop.~\ref{prop:X0subalgebra}).
\end{thm}

\subsection{Comparison theorems}
We have seen by Theorem~\ref{tseh} and Theorem~\ref{mseh} that there are two ways to regularize divergent $\mu$-MHZVs. However,  there are close relations between them as stated by the following comparison theorem
which plays the key role in this paper.

\begin{thm}\label{thm:ComparisonMaps} (The Comparison Theorem)
Fix $\mu\in\CC$ with $\Ree(\mu)>0$. Let $H(x):=\int_0^1\frac{1-t^x}{1-t}dt$.
Then the following three statements hold.

$(1).$ For any given $m,m'\in\Q_{>0}$, their exists a unique $\R$-linear map
\begin{align*}
\myrho^{\mu,*}_{m,m'}:\R[T]&\, \longrightarrow\R[T],\\
T^s&\, \longmapsto(T+H(m'/\mu-1)-H(m/\mu-1))^s \quad\forall s\in\Z_{\geqslant0},
\end{align*}
called the \emph{$\mu$-stuffle comparison map}, that satisfies
$$
\zeta_{*,m'}^{\mu,T}(w)=\myrho^{\mu,*}_{m,m'}(\zeta_{*,m}^{\mu,T}(w))\qquad\forall w\in \Q[\mu]\langle Y\rangle.
$$
That is, the following diagram commutes:
\begin{center}
\begin{tikzpicture}[scale=0.9]
\node (A) at (0,2) {$(\Q[\mu]\langle Y\rangle,\must)$};
\node (B) at (-2,0) {$(\calZ^{\mu},\cdot)[T]$};
\node (C) at (2,0) {$(\calZ^{\mu},\cdot)[T]$};
\draw[->] (A) to node[pos=.8,above] {$\zeta_{*,m}^{\mu,T}\ \ $}(B);
\draw[->] (A) to node[pos=.8,above] {$\ \ \zeta_{*,m'}^{\mu,T}$}(C);
\draw[->] (B) to (C) node[pos=.5,below] {$\myrho^{\mu,*}_{m,m'}$};
\end{tikzpicture}
\end{center}

$(2).$ For any $m,m'\in\Q_{>0}$, their exists a unique $\R$-linear map
\begin{align*}
\myrho^{\mu,\shuffle}_{m,m'}:\R[T]&\, \longrightarrow\R[T],\\
T^s&\, \longmapsto(T+H(m'/\mu-1)-H(m/\mu-1))^s\quad\forall s\in\Z_{\geqslant0},
\end{align*}
called the \emph{shuffle comparison map}, that satisfies
$$
\zeta_{\shuffle,m'}^{\mu,T}(w)=\myrho^{\mu,\shuffle}_{m,m'}(\zeta_{\shuffle,m}^{\mu,T}(w))\qquad\forall w\in \Q[\mu]\langle X\rangle^1.
$$
That is, the following diagram commutes:
\begin{center}
\begin{tikzpicture}[scale=0.9]
\node (A) at (0,2) {$(\Q[\mu]\langle X^{\mu}\rangle^1,\shuffle)$};
\node (B) at (-2,0) {$(\calZ^{\mu},\cdot)[T]$};
\node (C) at (2,0) {$(\calZ^{\mu},\cdot)[T]$};
\draw[->] (A) to node[pos=.8,above] {$\zeta_{\shuffle,m}^{\mu,T}\ \ $}(B);
\draw[->] (A) to node[pos=.8,above] {$\ \ \zeta_{\shuffle,m'}^{\mu,T}$}(C);
\draw[->] (B) to (C) node[pos=.5,below] {$\myrho^{\mu,\shuffle}_{m,m'}$};
\end{tikzpicture}
\end{center}

$(3).$ For any $m,m'\in\Q_{>0}$, their exists a unique $\R$-linear map
$$\rho_{m,m'}^{\mu}:\R[T]\longrightarrow\R[T],$$
called the \emph{mixed comparison map}, that satisfies
$$
\zeta_{\shuffle,m'}^{\mu,T}(\varphi_{\mu}(w))=\rho_{m,m'}^{\mu}(\zeta_{*,m}^{\mu,T}(w))\qquad\forall w\in \Q[\mu]\langle Y\rangle
$$
where $\varphi_{\mu}$ is defined by \eqref{equ:bijection-varphi}. That is, the following diagram commutes:
$$\begin{CD}
(\Q[\mu]\langle Y\rangle,\must)@>{\varphi_{\mu}}>>(\Q[\mu]\langle X^{\mu}\rangle^1,\shuffle)\\
@V\zeta_{*,m}^{\mu,T}VV@VV\zeta_{\shuffle,m'}^{\mu,T}V\\
(\calZ^{\mu},\cdot)[T]@>{\rho_{m,m'}^{\mu}}>>(\calZ^{\mu},\cdot)[T]\end{CD}$$
\end{thm}

Theorem~\ref{thm:ComparisonMaps} yields the following result.

\begin{thm}\label{kfhy}
We have the following commutative diagram
$$\begin{CD}
(\calZ^{\mu},\cdot)[T]@>\myrho^{\mu,*}_{m_1,m_1'}>>(\calZ^{\mu},\cdot)[T]\\
@V{\rho_{m_1,m_2}^{\mu}}VV@VV{\rho_{m_1',m_2'}^{\mu}}V\\
(\calZ^{\mu},\cdot)[T]@>\myrho^{\mu,\shuffle}_{m_2,m_2'} >>(\calZ^{\mu},\cdot)[T]\end{CD}
$$
Hence, we have a network composed of comparison maps shown by the following diagram:
$$
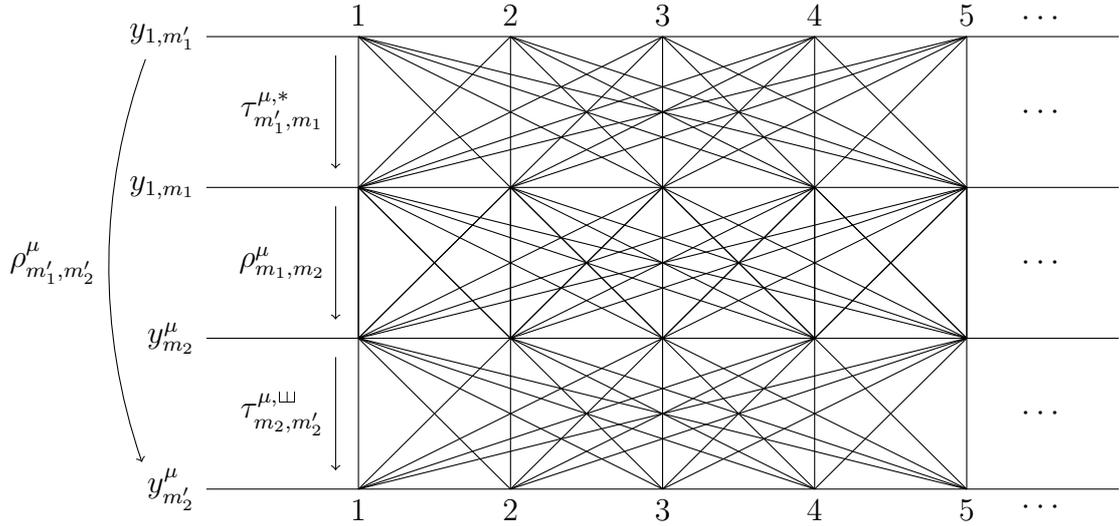
\begin{tikzpicture}
\draw (-6,3) -- (6,3);
\draw (-6,-3) -- (6,-3);
\draw (-6,1) -- (6,1);
\draw (-6,-1) -- (6,-1);
\draw (4,-3) -- (4,3);
\draw (2,-3) -- (2,3);
\draw (0,-3) -- (0,3);
\draw (-2,-3) -- (-2,3);
\draw (-4,-3) -- (-4,3);
\draw (4,-1) -- (2,-3);
\draw (4,1) -- (0,-3);
\draw (4,3) -- (-2,-3);
\draw (2,3) -- (-4,-3);
\draw (0,3) -- (-4,-1);
\draw (-2,3) -- (-4,1);
\draw (2,3) -- (4,1);
\draw (0,3) -- (4,-1);
\draw (-2,3) -- (4,-3);
\draw (-4,3) -- (2,-3);
\draw (-4,1) -- (0,-3);
\draw (-4,-1) -- (-2,-3);
\draw (-4,3) -- (0,1);
\draw (-4,3) -- (2,1);
\draw (-4,3) -- (4,1);
\draw (-2,3) -- (2,1);
\draw (-2,3) -- (4,1);
\draw (0,3) -- (-4,1);
\draw (0,3) -- (4,1);
\draw (2,3) -- (-4,1);
\draw (2,3) -- (-2,1);
\draw (4,3) -- (-4,1);
\draw (4,3) -- (-2,1);
\draw (4,3) -- (0,1);
\draw (4,-1) -- (4,1);
\draw (4,-1) -- (2,1);
\draw (4,-1) -- (0,1);
\draw (4,-1) -- (-2,1);
\draw (4,-1) -- (-4,1);
\draw (2,-1) -- (4,1);
\draw (2,-1) -- (2,1);
\draw (2,-1) -- (0,1);
\draw (2,-1) -- (-2,1);
\draw (2,-1) -- (-4,1);
\draw (0,-1) -- (4,1);
\draw (0,-1) -- (2,1);
\draw (0,-1) -- (0,1);
\draw (0,-1) -- (-2,1);
\draw (0,-1) -- (-4,1);
\draw (-2,-1) -- (4,1);
\draw (-2,-1) -- (2,1);
\draw (-2,-1) -- (0,1);
\draw (-2,-1) -- (-2,1);
\draw (-2,-1) -- (-4,1);
\draw (-4,-1) -- (4,1);
\draw (-4,-1) -- (2,1);
\draw (-4,-1) -- (0,1);
\draw (-4,-1) -- (-2,1);
\draw (-4,-1) -- (-4,1);
\draw (-4,-3) -- (0,-1);
\draw (-4,-3) -- (2,-1);
\draw (-4,-3) -- (4,-1);
\draw (-2,-3) -- (2,-1);
\draw (-2,-3) -- (4,-1);
\draw (0,-3) -- (-4,-1);
\draw (0,-3) -- (4,-1);
\draw (2,-3) -- (-4,-1);
\draw (2,-3) -- (-2,-1);
\draw (4,-3) -- (-4,-1);
\draw (4,-3) -- (-2,-1);
\draw (4,-3) -- (0,-1);

\node [above] at (-4,3) {$1$};
\node [above] at (-2,3) {$2$};
\node [above] at (-0,3) {$3$};
\node [above] at (2,3) {$4$};
\node [above] at (4,3) {$5$};
\node [below] at (-4,-3) {$1$};
\node [below] at (-2,-3) {$2$};
\node [below] at (-0,-3) {$3$};
\node [below] at (2,-3) {$4$};
\node [below] at (4,-3) {$5$};
\node [above] at (5,3) {$\cdots$};
\node at (5,0) {$\cdots$};
\node at (5,2) {$\cdots$};
\node at (5,-2) {$\cdots$};
\node [below] at (5,-3) {$\cdots$};
\node [left] at (-6,1) {$y_{1,m_1}$};
\node [left] at (-6,3) {$y_{1,m_1'}$};
\node [left] at (-6,-1) {$y^{\mu}_{m_2}$};
\node [left] at (-6,-3) {$y^{\mu}_{m_2'}$};
\node at (-5,2) {$\myrho^{\mu,*}_{m_1',m_1}$};
\node at (-5,0) {$\rho_{m_1,m_2}^{\mu}$};
\node at (-5,-2) {$\myrho^{\mu,\shuffle}_{m_2,m_2'}$};
\draw [->] (-4.3,2.75) -- (-4.3,1.25);
\draw [->] (-4.3,0.75) -- (-4.3,-0.75);
\draw [->] (-4.3,-1.25) -- (-4.3,-2.75);
\draw [->] (-6.8,2.7) arc [radius=8, start angle=160, end angle= 199];
\node [left] at (-7.3,0) {$\rho_{m_1',m_2'}^{\mu}$};
\end{tikzpicture}$$
\end{thm}

By Theorem~\ref{thm:ComparisonMaps} and Theorem~\ref{kfhy}, if we want to give a detailed description of the map $\rho_{m_1,m_2}^{\mu}$
we only need to consider the map $\rho_{1,1}^{\mu}$ whose property is described by the next theorem.

\begin{thm}\label{oovn}
We have
\begin{align*}
\rho_{1,1}^{\mu}:\R[T]&\, \longrightarrow\R[T],\\
T^n&\, \longmapsto n!\sum_{k=0}^n\frac{\gamma_k}{(n-k)!}\cdot T^{n-k}
\end{align*}
where
$$
\sum_{k=0}^{\infty}\gamma_ku^k=e^{\gamma u}\Gamma(1+u).
$$
Furthermore, for any $m,m'\in\Q_{>0}$
\begin{align*}
\rho_{m,m'}^{\mu}:\R[T]&\, \longrightarrow\R[T],\\
T^n&\, \longmapsto\sum_{k=0}^n\binom{n}{k}c_{m,1}^{n-k}(\mu)\left(k!\sum_{l=0}^k\frac{\gamma_l}{(k-l)!}(T+c_{1,m'}(\mu))^{k-l}\right),
\end{align*}
where $c_{m,m'}(\mu)=H(m'/\mu-1)-H(m/\mu-1)$.
\end{thm}

\begin{rem}
From Theorem~\ref{oovn}, we see that the map $\rho_{1,1}^{\mu}$ is independent of the parameter $\mu$.
Hence, the results in these comparison theorems are all $\mu$-invariant.
\end{rem}

After proving the above main structural results on $\mu$-MHZVs in sections \ref{sec:algSetup} and \ref{sec:RegHomo},
as an application, we will present some identities (see Example~\ref{eg:ClauCata}) such as
\begin{align*}
\sum_{0\leqslant n\leqslant m}\frac{1}{(4n+3)(4m+3)(4m+6)}&\, =\frac{7\pi^2}{576}-\frac{G}{12},
\end{align*}
where $G=\sum\limits_{n=0}^{\infty}\frac{(-1)^n}{(2n+1)^2}$ is Catalan's constant.
These identities are closely related to the multiple zeta values of level $N$ studied by Yuan and the second author \cite{YuanZhao2015b}.

In section \ref{sec:sumFormula} we will derive two sum formulas for $\mu$-double Hurwitz zeta values which generalize the sum formulas for double zeta values and double zeta star values, respectively, and a weighted sum formula which generalizes the weighted sum formulas for both double zeta values and double $T$-values simultaneously.

At the end of the paper, we will propose a problem extending \cite[Conjecture 1]{ikz} by Ihara et al.
for multiple zeta values to the $\mu$-MHZV setting.

\medskip
\noindent

\bigskip
\noindent{\bf Acknowledgments.} The authors would like to thank Dr.\ S.\ Charlton for his detailed comments of
the first version of the paper. Jia Li is supported by the School of Mathematical Sciences, Peking University.
J. Zhao is supported by the Jacobs Prize from The Bishop's School. 

\section{Algebraic structures on $\calZ^{\mu}$} \label{sec:stANDsh}

In this section, using their series definition, we will prove that the product of two $\mu$-MHZVs is a
$\Q[\mu]$-linear combination of $\mu$-MHZVs. We call this the $\mu$-stuffle product. Then we will describe
the integral representation of $\mu$-MHZVs similar to Kontsevich's integral representation for MZVs. These two
kinds of product of $\mu$-MHZVs play the key roles in our theory, leading to many $\Q[\mu]$-linear
relations in a systematic way.

\subsection{The $\mu$-stuffle product on $\calZ^{\mu}$}

We first introduce some preliminary results.

\begin{lem}\label{lem:parFrac}
For all $m, n\in\Q_{>0}$, $m\ne n$, and $j,k\in \N$ we have the partial fraction decomposition
\begin{multline*}
\frac{1}{(x+m)^j (x+n)^k}
= \frac{(-1)^{j-1}(j+k-2)!}{(j-1)!(k-1)!(n-m)^{j+k-1}} \left(\frac{1}{x+m}-\frac{1}{x+n}\right ) \\
+ \sum_{i=0}^{k-2} \binom{i+j-1}{i} \frac{(-1)^{j}}{(n-m)^{i+j}}\frac{1}{(x+n)^{k-i}}
+\sum_{i=0}^{j-2} \binom{i+k-1}{i} \frac{(-1)^{i} }{(n-m)^{i+k}} \frac{1}{(x+m)^{j-i}}.
\end{multline*}
\end{lem}
\begin{proof}
We have
\begin{align*}
&\, \frac{(-1)^{j-1}(j-1)!}{(x+m)^j (x+n)}
=\frac{\partial^{j-1}}{\partial m^{j-1}} \frac{1}{(x+m)(x+n)} \\
=&\, \frac{\partial^{j-1}}{\partial m^{j-1}}\left[ \frac{1}{n-m} \left(\frac{1}{x+m}-\frac{1}{x+n}\right)\right] \\
=&\, \frac{(j-1)!}{(n-m)^j} \left(\frac{1}{x+m}-\frac{1}{x+n}\right )
+\sum_{i=0}^{j-2} \binom{j-1}{i} \frac{ i!}{(n-m)^{i+1}} \frac{(-1)^{j-i-1}(j-i-1)!}{(x+m)^{j-i}}.
\end{align*}
Thus
\begin{align*}
&\, \frac{(-1)^{j-1}(j-1)!(-1)^{k-1}(k-1)!}{(x+m)^j (x+n)^k}
=\frac{\partial^{k-1}}{\partial n^{k-1}}\frac{\partial^{j-1}}{\partial m^{j-1}} \frac{1}{(x+m)(x+n)} \\
=&\, \frac{(-1)^{k-1}(j+k-2)!}{(n-m)^{j+k-1}} \left(\frac{1}{x+m}-\frac{1}{x+n}\right ) \\
-&\, \sum_{i=0}^{k-2} \binom{k-1}{i} \frac{(-1)^{i}(i+j-1)!}{(n-m)^{i+j}}\frac{(-1)^{k-i-1}(k-i-1)!}{(x+n)^{k-i}} \\
+&\,\sum_{i=0}^{j-2} \binom{j-1}{i} \frac{(-1)^{k-1} (i+k-1)!}{(n-m)^{i+k}} \frac{(-1)^{j-i-1}(j-i-1)!}{(x+m)^{j-i}}.
\end{align*}
The lemma follows immediately.
\end{proof}

Let $M\in\R_{\geqslant0}$, and $(\bfk;\bfm)=(k_1,\cdots,k_r;m_1,\cdots,m_r)$ be a positive bi-index. Define
$$
\zeta_M^{\mu}(\bfk;\bfm) :=\sum_{0\leqslant n_1\leqslant\cdots\leqslant n_r\leqslant M}\frac{\mu^r}{(n_1\mu+m_1)^{k_1}\cdots(n_r\mu+m_r)^{k_r}}.
$$
Set $\zeta^{\mu}(\emptyset):=1$ by convention.
If $(\bfk;\bfm)$ is a non-empty admissible bi-index, then we have
$$
\zeta^{\mu}(\bfk;\bfm)=\lim_{M\to+\infty}\zeta^{\mu}_M(\bfk;\bfm).
$$
Furthermore, if $(\bfk;\bfm;\tbfm)$ is a special triple then
$$
\calD^{\mu}(\bfk;\bfm;\tbfm):=\lim_{M\to+\infty} \zeta_M^{\mu}(\bfk;\bfm;\tbfm)
$$
where
$$
\zeta_M^{\mu}(\bfk;\bfm;\tbfm)=\zeta_M^{\mu}(\bfk;\bfm)-\zeta_M^{\mu}(\bfk;\tbfm).
$$

\begin{lem}\label{somc}
Let $(\bfk;\bfm)=(k_1,\cdots,k_r;m_1,\cdots,m_r)\in\N^r\times(\Q_{>0})^r$. Then we have
$$\zeta^{\mu}_M(\bfk;\bfm)=
\sum_{0\leqslant n_r\leqslant M}\frac{\mu}{(n_r\mu+m_r)^{k_r}}\zeta^{\mu}_{n_r}(k_1,\cdots,k_{r-1};m_1,\cdots,m_{r-1}).
$$
\end{lem}

\begin{proof} By the definition,
\begin{align*}
&\zeta^{\mu}_M(k_1,\cdots,k_r;m_1,\cdots,m_r)\\
&=\sum_{0\leqslant n_1\leqslant\cdots\leqslant n_r\leqslant M}\frac{\mu^r}{(n_1\mu+m_1)^{k_1}\cdots(n_r\mu+m_r)^{k_r}}\\
&=\sum_{0\leqslant n_r\leqslant M}\frac{\mu}{(n_r\mu+m_r)^{k_r}}\sum_{0\leqslant n_1\leqslant\cdots\leqslant n_{r-1}\leqslant n_r}\frac{\mu^{r-1}}{(n_1\mu+m_1)^{k_1}\cdots(n_{r-1}\mu+m_{r-1})^{k_{r-1}}}\\
&=\sum_{0\leqslant n_r\leqslant M}\frac{\mu}{(n_r\mu+m_r)^{k_r}}\zeta^{\mu}_{n_r}(k_1,\cdots,k_{r-1};m_1,\cdots,m_{r-1}),
\end{align*}
as desired.
\end{proof}

\begin{thm}\label{thm:ddog}
We have

\medskip
$(1).$ For any positive bi-indices $(\bfk_1;\bfm_1)$ and $(\bfk_2;\bfm_2)$,
\begin{align} \notag
& \zeta_M^{\mu}(\bfk_1;\bfm_1)\cdot\zeta_M^{\mu}(\bfk_2;\bfm_2) \\
= & \sum_{(\bfk;\bfm)} \st_{1;1}^{\mu}\left[\binom{\bfk_1}{\bfm_1},\binom{\bfk_2}{\bfm_2};\binom{\bfk}{\bfm}\right]\zeta^{\mu}_M(\bfk;\bfm) \label{equ:Mstfl11}\\
+& \sum_{(\bfk;\bfm;\tbfm)} \st_{1;2}^{\mu}\left[\binom{\bfk_1}{\bfm_1}, \binom{\bfk_2}{\bfm_2};
 \binom{\bfk}{\bfm,\tbfm} \right]
\calD^{\mu}_M(\bfk;\bfm;\tbfm) \label{equ:Mstfl12}
\end{align}
where the coefficients $\st_{1;1}^{\mu},\st_{1;2}^{\mu}\in\Q[\mu]$. Here, in the first
sum $(\bfk;\bfm)$ runs through all positive bi-indices such that $|\bfk|\leqslant |\bfk_1|+|\bfk_2|$ and $\ell(\bfk)=\ell(\bfm)\leqslant \ell(\bfk_1)+\ell(\bfk_2)$, and in the second sum $(\bfk;\bfm;\tbfm)$ runs through all positive bi-indices such that $|\bfk|\leqslant |\bfk_1|+|\bfk_2|$ and $\ell(\bfk)=\ell(\bfm)\leqslant \ell(\bfk_1)+\ell(\bfk_2)$ such that $(\bfk;\bfm;\tbfm)$ are special triples.
Moreover,
$\st_{1;1}^{\mu}\big(-;\binom{\bfk}{\bfm}\big)=0$ for all non-admissible $\bfk$ if $\bfk_1$ and $\bfk_2$ are both admissible.

\medskip
$(2).$ For any positive bi-index $(\bfk_1;\bfm_1)$ and any special triple $(\bfk_2;\bfm_2;\tbfm_2)$,
\begin{align}
&\zeta_M^{\mu}(\bfk_1;\bfm_1)\cdot\calD_M^{\mu}(\bfk_2;\bfm_2;\tbfm_2) \notag\\
=&\sum_{(\bfk;\bfm)} \st_{2;1}^{\mu}\left[\binom{\bfk_1}{\bfm_1}, \binom{\bfk_2}{\bfm_2,\bfm_2'} ;
 \binom{\bfk}{\bfm} \right]\zeta^{\mu}_M(\bfk;\bfm) \notag\\
+ &\sum_{(\bfk;\bfm;\tbfm)} \st_{2;2}^{\mu}\left[\binom{\bfk_1}{\bfm_1}, \binom{\bfk_2}{\bfm_2,\bfm_2'} ;
 \binom{\bfk}{\bfm,\tbfm} \right] \calD^{\mu}_M(\bfk;\bfm;\tbfm)\label{equ:MSpecialPstf2}
\end{align}
where the coefficients $\st_{2;1}^{\mu},\st_{2;2}^{\mu}\in\Q[\mu]$. Here the indices in the sums have the
same restrictions as in $(1)$. Moreover, $\st_{2;1}^{\mu}\big(-;\binom{\bfk}{\bfm}\big)=0$ for all
non-admissible $\bfk$ if $\bfk_1$ is admissible.

\medskip
$(3).$ For any two special triples $(\bfk_1;\bfm_1;\tbfm_1)$ and $(\bfk_2;\bfm_2;\tbfm_2)$,
\begin{align}
&\calD_M^{\mu}(\bfk_1;\bfm_1;\tbfm_1) \cdot\calD_M^{\mu}(\bfk_2;\bfm_2;\tbfm_2) \notag\\
=&\sum_{(\bfk;\bfm)} \st_{3;1}^{\mu}\left[\binom{\bfk_1}{\bfm_1,\tbfm_1}, \binom{\bfk_2}{\bfm_2,\tbfm_2} ;
 \binom{\bfk}{\bfm} \right]\zeta^{\mu}_M(\bfk;\bfm) \notag\\
+ &\sum_{(\bfk;\bfm;\tbfm)} \st_{3;2}^{\mu}\left[\binom{\bfk_1}{\bfm_1,\tbfm_1}, \binom{\bfk_2}{\bfm_2,\tbfm_2} ;
 \binom{\bfk}{\bfm,\tbfm} \right] \calD^{\mu}_M(\bfk;\bfm;\tbfm) \label{equ:MSpecialPstf3}
\end{align}
where the coefficients $\st_{3;1}^{\mu},\st_{3;2}^{\mu}\in\Q[\mu]$. Here the indices in the sums have the
same restrictions as in $(1)$ except that in the first sum $\bfk$ can only be admissible.
\end{thm}

\begin{proof} We will prove (1) and (2) simultaneously by induction on $\ell(\bfk_1)+\ell(\bfk_2)$.
When $\ell(\bfk_1)+\ell(\bfk_2)=1$ in (1) or (2) then the claims are trivially true since
$\zeta_M^{\mu}(\emptyset)=1$ by definition.
Suppose (1) and (2) hold if $\ell(\bfk_1)+\ell(\bfk_2)<l$ and assume now $\ell(\bfk_1)+\ell(\bfk_2)=l$.

(1). For convenience, we denote by $\bfv'$ the index obtained from $\bfv$ by removing its last component. Set
\begin{align*}
 (\bfk_1;\bfm_1)&\, =(k_1,\cdots,k_r;m_1,\cdots,m_r), \quad (\bfk_2;\bfm_2) =(k'_1,\cdots,k'_s;m'_1,\cdots,m'_s).
\end{align*}
Then
\begin{align}
&\zeta_M^{\mu}(\bfk_1;\bfm_1)\cdot\zeta_M^{\mu}(\bfk_2;\bfm_2) \notag\\
=&\sum_{0\leqslant n_r\leqslant M}\frac{\mu}{(n_r\mu+m_r)^{k_r}}\zeta^{\mu}_{n_r}(\bfk_1';\bfm_1')\cdot\zeta^{\mu}_{n_r}(\bfk_2;\bfm_2) \label{equ:stufflePf1}\\
 +&\sum_{0\leqslant n_s'\leqslant M}\frac{\mu}{(n_s'\mu+m_s')^{k_s'}}\zeta^{\mu}_{n_s'}(\bfk_1;\bfm_1)\cdot\zeta^{\mu}_{n_s'}(\bfk_2';\bfm_2')\label{equ:stufflePf2}\\
-& \mu\sum_{0\leqslant n\leqslant M}\frac{\mu}{(n\mu+m_r)^{k_r}(n\mu+m_s')^{k_s'}}\zeta^{\mu}_{n}(\bfk_1';\bfm_1')\cdot\zeta^{\mu}_{n}(\bfk_2';\bfm_2')\label{equ:stufflePf3}.
\end{align}
By induction, we can rewrite \eqref{equ:stufflePf1} as
\begin{align*}
\eqref{equ:stufflePf1}=&\, \sum_{(\bfk;\bfm)} \st_{1;1}^{\mu}\left[\binom{\bfk_1'}{\bfm_1'},\binom{\bfk_2}{\bfm_2};\binom{\bfk}{\bfm}\right]\sum_{0\leqslant n_r\leqslant M}\frac{\mu}{(n_r\mu+m_r)^{k_r}}\zeta^{\mu}_{n_r}(\bfk;\bfm) \\
 +&\,\sum_{(\bfk;\bfm;\tbfm)} \st_{1;2}^{\mu}\left[\binom{\bfk_1'}{\bfm_1}, \binom{\bfk_2}{\bfm_2};
 \binom{\bfk}{\bfm,\tbfm} \right] \sum_{0\leqslant n_r\leqslant M} \frac{\mu}{(n_r\mu+m_r)^{k_r}}\calD^{\mu}_{n_r}(\bfk;\bfm;\tbfm)\\
=&\,\sum_{(\bfk;\bfm)} \st_{1;1}^{\mu}\left[\binom{\bfk_1'}{\bfm_1'},\binom{\bfk_2}{\bfm_2};\binom{\bfk}{\bfm}\right] \zeta^{\mu}_M(\bfk,k_r;\bfm,m_r) \\
 +&\, \sum_{(\bfk;\bfm;\tbfm)} \st_{1;2}^{\mu}\left[\binom{\bfk_1'}{\bfm_1}, \binom{\bfk_2}{\bfm_2};
 \binom{\bfk}{\bfm,\tbfm} \right]
\Big(\zeta^{\mu}_M(\bfk,k_r;\bfm,m_r)-\zeta^{\mu}_M(\bfk,k_r;\bfm',m_r)\Big).
\end{align*}
Therefore, \eqref{equ:stufflePf1} is expressed in the form as stated in the theorem. In particular, all terms are
admissible if $\bfk_1$ is admissible. The sum \eqref{equ:stufflePf2} can be dealt with similarly and all terms are
admissible if $\bfk_2$ is admissible.

For \eqref{equ:stufflePf3}, by Lemma~\ref{lem:parFrac}
\begin{equation}\label{equ:ProductTwoPowers}
\frac{\mu}{(n\mu+m_r)^{k_r}(n\mu+m_s')^{k_s'}} =
\left\{
 \begin{array}{ll}
\displaystyle A+\sum_{i=2}^{k_r}\frac{a_i\mu}{(n\mu+m_r)^i}+\sum_{j=2}^{k_s'}\frac{b_j\mu}{(n\mu+m_s')^j} , & \hbox{if $m_r\ne m_s'$;} \\
\displaystyle \frac{\mu}{(n\mu+m)^{k_r+k_s'}} , & \hbox{if $m_r=m_s'$,}
 \end{array}
\right.
\end{equation}
where
$$A=a_1\left(\frac{\mu}{n\mu+m_r}-\frac{\mu}{n\mu+m_s'}\right).$$
Among the four possible types of terms produced in \eqref{equ:ProductTwoPowers}, we only need to consider
$A$ since the others can all be handled by the same method we used to rewrite \eqref{equ:stufflePf1}.
Thus, we need to consider
\begin{align*}
\sum_{0\leqslant n\leqslant M}\left(\frac{\mu}{n\mu+m_r}-\frac{\mu}{n\mu+m_s'}\right) \zeta^{\mu}_{n}(\bfk_1';\bfm_1')\cdot\zeta^{\mu}_{n}(\bfk_2';\bfm_2').
\end{align*}
By induction assumption, this can be written as
\begin{align*}
&\sum_{(\bfk;\bfm)} \st_{1;1}^{\mu}\left[\binom{\bfk_1'}{\bfm_1'},\binom{\bfk_2'}{\bfm_2'};\binom{\bfk}{\bfm}\right]
\sum_{0\leqslant n\leqslant M}\left(\frac{\mu}{n\mu+m_r}-\frac{\mu}{n\mu+m_s'}\right) \zeta^{\mu}_n(\bfk;\bfm) \\
+&\, \sum_{(\bfk;\bfm;\tbfm)} \st_{1;2}^{\mu}\left[\binom{\bfk_1'}{\bfm_1'}, \binom{\bfk_2'}{\bfm_2'};
 \binom{\bfk}{\bfm,\tbfm} \right]
\sum_{0\leqslant n\leqslant M}\left(\frac{\mu}{n\mu+m_r}-\frac{\mu}{n\mu+m_s'}\right)
\calD^{\mu}_n(\bfk;\bfm;\tbfm) \\
=&\, \sum_{(\bfk;\bfm)} \st_{1;1}^{\mu}\left[\binom{\bfk_1'}{\bfm_1'},\binom{\bfk_2'}{\bfm_2'};\binom{\bfk}{\bfm}\right]
\calD^{\mu}_M(\bfk,1;\bfm,m_r;\bfm,m_s')\Big) \\
+&\, \sum_{(\bfk;\bfm;\tbfm)} \st_{1;2}^{\mu}\left[\binom{\bfk_1'}{\bfm_1'}, \binom{\bfk_2'}{\bfm_2'};
 \binom{\bfk}{\bfm,\tbfm} \right]\\
&\, \hskip5cm \times\Big(\calD^{\mu}_M(\bfk,1;\bfm,m_r;\bfm,m_s')+\calD^{\mu}_M(\bfk,1;\bfm',m_s';\bfm',m_r) \Big)
\end{align*}
which has the desired form. This completes the proof of (1) and (2).

Now we prove (3) by induction on $\ell(\bfk_1)+\ell(\bfk_2)$ again.
When $\ell(\bfk_1)+\ell(\bfk_2)=2$ then we must have $\ell(\bfk_1)=\ell(\bfk_2)=1$. Thus $\bfk_1=\bfk_2=(1)$ and we get
\begin{align*}
&\calD_M^{\mu}(\bfk_1;\bfm_1;\tbfm_1) \cdot \calD_M^{\mu}(\bfk_2;\bfm_2;\tbfm_2) \notag\\
=&\sum_{0\leqslant n_1,n_2\leqslant M} \left(\frac{\mu}{n_1\mu+m_1}-\frac{\mu}{n_1\mu+\tm_1}\right)\left(\frac{\mu}{n_2\mu+m_2}-\frac{\mu}{n_2\mu+\tm_2}\right) \notag\\
=&\, \sum_{1 \leftrightarrow 2} \Big(\calD_M^{\mu}(1,1;m_1,m_2;m_1,\tm_2)
-\calD_M^{\mu}(1,1;\tm_1,m_2;\tm_1,\tm_2)\Big) \\
& \,
-\sum_{0\leqslant n\leqslant M} \left(\frac{\mu}{n\mu+m_1}-\frac{\mu}{n\mu+\tm_1}\right)
 \left(\frac{\mu}{n\mu+m_2}-\frac{\mu}{n\mu+\tm_2}\right)\\
=&\, \sum_{1 \leftrightarrow 2} \Big(\calD_M^{\mu}(1,1;m_1,m_2;m_1,\tm_2)
+\calD_M^{\mu}(1,1;\tm_1,\tm_2;\tm_1,m_2)\Big) \\
& \,
-\frac{\mu\cdot \delta_{m_1\ne m_2}}{m_2-m_1} \calD_M^{\mu}(1;m_1;m_2)
+\frac{\mu\cdot \delta_{m_1\ne \tm_2}}{\tm_2-m_1} \calD_M^{\mu}(1;m_1;\tm_2)\\
& \, +\frac{\mu\cdot \delta_{\tm_1\ne m_2}}{m_2-\tm_1} \calD_M^{\mu}(1;\tm_1;m_2)
-\frac{\mu\cdot \delta_{\tm_1\ne \tm_2}}{\tm_2-\tm_1} \calD_M^{\mu}(1;\tm_1;\tm_2),
\end{align*}
where $\delta_{a\ne b}=1$ if $a\ne b$ and $\delta_{a\ne b}=0$ if $a=b$. Here, we have used the fact that
if $a\ne b$ then
\begin{equation} \label{equ:basecase}
 \frac{\mu}{(n\mu+a)(n\mu+b)}=\frac{1}{b-a}\left(\frac{\mu}{n\mu+a}-\frac{\mu}{n\mu+b}\right).
\end{equation}
We see that \eqref{equ:MSpecialPstf3} holds in this case. In fact, the right-hand side involve only special triples.

Suppose now (3) holds if $\ell(\bfk_1)+\ell(\bfk_2)<l$ and assume now $\ell(\bfk_1)+\ell(\bfk_2)=l$.
Write $\tbfm_1=(\bfm_1',\tm_r)$ and $\tbfm_2=(\bfm_2',\tm_s')$. Then
\begin{align*}
\calD_M^{\mu}(\bfk_1;\bfm_1;\tbfm_1)\cdot\calD_M^{\mu}(\bfk_2;\bfm_2;\tbfm_2)={\rm I+II}-\mu \cdot {\rm III}
\end{align*}
where
\begin{align*}
{\rm I}=&\, \sum_{0\leqslant n_r\leqslant M}\left(\frac{\mu}{n_r\mu+m_r}-\frac{\mu}{n_r\mu+\tm_r}\right)
\zeta^{\mu}_{n_r}(\bfk_1';\bfm_1')\cdot \calD_{n_r}^{\mu}(\bfk_2;\bfm_2;\tbfm_2),\\
{\rm II}=&\, \sum_{0\leqslant n_s'\leqslant M}\left(\frac{\mu}{n_s'\mu+m_s'}-\frac{\mu}{n_s'\mu+\tm_s'}\right) \calD_{n_s'}^{\mu}(\bfk_1;\bfm_1;\tbfm_1) \cdot\zeta^{\mu}_{n_s'}(\bfk_2';\bfm_2'),\\
{\rm III}=&\, \sum_{0\leqslant n\leqslant M}\left(\frac{\mu}{n\mu+m_r}-\frac{\mu}{n\mu+\tm_r}\right)\left(\frac{1}{n\mu+m_s'}-\frac{1}{n\mu+\tm_s'}\right)
 \zeta^{\mu}_{n}(\bfk_1';\bfm_1')\cdot\zeta^{\mu}_{n}(\bfk_2';\bfm_2').
\end{align*}
For the first two sums, it is easy to see by (1) and (2) that they reduce to the case of $\bfk_1=(1)$
when we have
\begin{align*}
{\rm I}=&\, \sum_{0\leqslant n_r\leqslant M}\left(\frac{\mu}{n_r\mu+m_r}-\frac{\mu}{n_r\mu+\tm_r}\right)
\cdot \calD_{n_r}^{\mu}(\bfk_2;\bfm_2;\tbfm_2)\\
=&\, \calD_M^{\mu}(\bfk_2,1;\bfm_2,m_r;\bfm_2,\tm_r)+\calD_M^{\mu}(\bfk_2,1;\tbfm_2,\tm_r;\tbfm_2,m_r)
\end{align*}
and
\begin{align*}
{\rm II}=&\, \sum_{0\leqslant n_s'\leqslant M}\left(\frac{\mu}{n_s'\mu+m_s'}-\frac{\mu}{n_s'\mu+\tm_s'}\right) \left(\frac{\mu}{n_s'\mu+m_1}-\frac{\mu}{n_s'\mu+\tm_1}\right) \cdot\zeta^{\mu}_{n_s'}(\bfk_2';\bfm_2') .
\end{align*}
By \eqref{equ:basecase} we only need to consider the terms such as
\begin{align*}
 \sum_{0\leqslant n_s'\leqslant M}\left(\frac{\mu}{n_s'\mu+m_s'}\cdot\frac{\mu}{n_s'\mu+\tm_1}\right) \cdot\zeta^{\mu}_{n_s'}(\bfk_2';\bfm_2').
\end{align*}
If $m_s'=\tm_1$ then this is equal to $\mu\zeta^{\mu}_M(\bfk_2',2;\bfm_2)$.
If $m_s'\ne \tm_1$ then \eqref{equ:basecase} produces two terms each of which has the total length $1+\ell(\bfk_2')=\ell(\bfk_2)<\ell(\bfk_2)+1$ so that by the induction assumption
it can be expressed in the form of \eqref{equ:MSpecialPstf3}. This settles II.

For III, first we can use (1) to expand the product $\zeta^{\mu}_{n}(\bfk_1';\bfm_1')\cdot\zeta^{\mu}_{n}(\bfk_2';\bfm_2')$
so that each term, say $\alpha$, has its length bounded by $\ell(\bfk_1')+\ell(\bfk_2')=l-2$.
Then we expand the product
$$\left(\frac{\mu}{n_r\mu+m_r}-\frac{\mu}{n_r\mu+\tm_r}\right)\left(\frac{1}{n\mu+m_s'}-\frac{1}{n\mu+\tm_s'}\right)$$
by \eqref{equ:basecase} so that each term, say $\beta$, has length 1.
Then it is readily seen that III can be written in the desired form by induction assumption
since the total length is decreased by at least 1 for each product $\alpha\beta$.

We have now completed the proof of the theorem.
\end{proof}

\begin{eg}\label{firn}
If we take $(\bfk_1;\bfm_1)=(2;m_1)$ and $(\bfk_2;\bfm_2)=(2;m_2)$, then we have
\begin{align*}
&\zeta_M^{\mu}(2;m_1)\cdot\zeta_M^{\mu}(2;m_2)\\&=\left(\sum_{0\leqslant n_1\leqslant M}\frac{\mu}{(n_1\mu+m_1)^2}\right)\cdot\left(\sum_{0\leqslant n_2\leqslant M}\frac{\mu}{(n_2\mu+m_2)^2}\right)\\
&=\sum_{0\leqslant n_1\leqslant n_2\leqslant M}\frac{\mu^2}{(n_1\mu+m_1)^2(n_2\mu+m_2)^2}+\sum_{0\leqslant n_2\leqslant n_1\leqslant M}\frac{\mu^2}{(n_1\mu+m_1)^2(n_2\mu+m_2)^2}\\
&\quad-\sum_{0\leqslant n\leqslant M}\frac{\mu^2}{(n\mu+m_1)^2(n\mu+m_2)^2}\\
&=\zeta_M^{\mu}(2,2;m_1,m_2)+\zeta_M^{\mu}(2,2;m_2,m_1)-\mu\sum_{0\leqslant n\leqslant M}\frac{\mu}{(n\mu+m_1)^2(n\mu+m_2)^2}.
\end{align*}
Notice that if $m_1\neq m_2$ then
\begin{align*}
\frac{1}{(n\mu+m_1)^2(n\mu+m_2)^2}&\, =\frac{1}{(m_1-m_2)^2}\cdot\frac{1}{(n\mu+m_1)^2}+\frac{1}{(m_1-m_2)^2}\cdot\frac{1}{(n\mu+m_2)^2}\\
&\quad+\frac{2}{(m_1-m_2)^3}\cdot\frac{1}{n\mu+m_1}-\frac{2}{(m_1-m_2)^3}\cdot\frac{1}{n\mu+m_2}.
\end{align*}
Hence,
\begin{align*}
\sum_{0\leqslant n\leqslant M}\frac{\mu}{(n\mu+m_1)^2(n\mu+m_2)^2}&\, =\frac{1}{(m_1-m_2)^2}(\zeta_M^{\mu}(2;m_1)+\zeta_M^{\mu}(2;m_2))\\
&\quad+\frac{2}{(m_1-m_2)^3}\calD_M^{\mu}(1;m_1;m_2).
\end{align*}
Finally, if $m_1=m_2=m$, then
$$\zeta_M^{\mu}(2;m_1)\cdot\zeta_M^{\mu}(2;m_2)=\zeta_M^{\mu}(2,2;m_1,m_2)+\zeta_M^{\mu}(2,2;m_2,m_1)-\mu\zeta^{\mu}_M(4;m),$$
and if $m_1\neq m_2$, then
\begin{multline*}
\zeta_M^{\mu}(2;m_1)\cdot\zeta_M^{\mu}(2;m_2) =\zeta_M^{\mu}(2,2;m_1,m_2)+\zeta_M^{\mu}(2,2;m_2,m_1)\\
-\frac{\mu}{(m_1-m_2)^2}(\zeta_M^{\mu}(2;m_1)+\zeta_M^{\mu}(2;m_2))
 -\frac{2\mu}{(m_1-m_2)^3}\calD_M^{\mu}(1;m_1;m_2).
\end{multline*}
\end{eg}

\begin{thm}\label{qqox}
$(\calZ^{\mu},\cdot)$ is a $\Q[\mu]$-algebra, that is, we have
\begin{align*}
\zeta^{\mu}(\bfk_1;\bfm_1)\cdot\zeta^{\mu}(\bfk_2;\bfm_2)&\in\calZ^{\mu}
\end{align*}
for all admissible $(\bfk_1;\bfm_1)$ and $(\bfk_2;\bfm_2)$,
\begin{align*}
\zeta^{\mu}(\bfk_1;\bfm_1) \cdot\calD^{\mu}(\bfk_2;\bfm_2;\tbfm_2)&\in\calZ^{\mu}
\end{align*}
for all admissible $(\bfk_1;\bfm_1)$ and special triples $(\bfk_2;\bfm_2;\tbfm_2)$, and
\begin{align*}
\calD^{\mu}(\bfk_1;\bfm_1;\tbfm_1)\cdot \calD^{\mu}(\bfk_2;\bfm_2;\tbfm_2)&\in\calZ^{\mu}
\end{align*}
for all special triples $(\bfk_1;\bfm_1;\tbfm_1)$ and $(\bfk_2;\bfm_2;\tbfm_2)$.
\end{thm}

\begin{proof}
This follows immediately by taking $M\to\infty$ in Theorem~\ref{thm:ddog}.
\end{proof}

\begin{eg}\label{ofns}
If we take $(\bfk_1;\bfm_1)=(2;m_1)$ and $(\bfk_2;\bfm_2)=(2;m_2)$, by Example \ref{firn}, we have

1. If $m_1=m=m_2$, then
$$\zeta^{\mu}(2;m_1)\cdot\zeta^{\mu}(2;m_2)=\zeta^{\mu}(2,2;m_1,m_2)+\zeta^{\mu}(2,2;m_2,m_1)-\mu\zeta^{\mu}(4;m).$$

2. If $m_1\neq m_2$, then
\begin{multline*}
\zeta^{\mu}(2;m_1)\cdot\zeta^{\mu}(2;m_2)=\zeta^{\mu}(2,2;m_1,m_2)+\zeta^{\mu}(2,2;m_2,m_1)\\
-\frac{\mu}{(m_1-m_2)^2}(\zeta^{\mu}(2;m_1)+\zeta^{\mu}(2;m_2))
-\frac{2\mu}{(m_1-m_2)^3}\calD^{\mu}(1;m_1;m_2).
\end{multline*}
\end{eg}

\subsection{Integral representation and shuffle product on $\calZ^{\mu}$}\label{sec:intRep}

For the classical MZVs, Kontsevich and Drinfel'd discovered the following integral representation.

\begin{thm}
Let $\bfk=(k_1,\cdots,k_r)$ be an admissible bi-index ($k_r>1$). Then we have
$$\zeta(\bfk)=\int_{\Delta^{|\bfk|}}\omega_{\bfk}
=\int_0^1 \left(\frac{dt}{t}\right)^{k_1-1}\frac{dt}{1-t} \cdots \left(\frac{dt}{t}\right)^{k_r-1}\frac{dt}{1-t}.$$
\end{thm}

\begin{proof} See \cite[Chapter 1]{bj} and \cite[Chapter 2]{ZhaoBook}.
\end{proof}

For the $\mu$-MHZV $\zeta^{\mu}(\bfk;\bfm)$, we have a similar result. To state it, we need some additional definitions
including the ``$\mu$-multiple polylogarithm function''.

\begin{defn}
Let $r,s\in\Z_{\geqslant0}$. A permutation of the set
$\{1,2,\cdots,r+s\}$ is called a shuffle of type $(r, s)$ if the following two conditions
are satisfied:
$$\sigma(1)<\sigma(2)<\cdots<\sigma(r),\quad\text{and}\quad\sigma(r+1)<\sigma(r+2)<\cdots<\sigma(r+s).$$
We denote the set of all shuffles of type $(r,s)$ by $\shuffle(r, s)$. That is
$$\shuffle(r,s):=\left\{\sigma\in S_{r+s}\Bigg|\
\begin{aligned}
&\sigma(1)<\sigma(2)<\cdots<\sigma(r),\\
&\sigma(r+1)<\sigma(r+2)<\cdots<\sigma(r+s)
\end{aligned}\right\}.
$$
\end{defn}
Notice that the subset $\shuffle(r,s)$ is not a subgroup of the permutation group $S_{r+s}$.

\begin{eg}
If $(r,s)=(2,2)$, then
$$\shuffle(2,2)=\{(1),(23),(243),(123),(1243),(13)(24)\}\subset S_4.$$
\end{eg}

\begin{defn}
Let $(\bfk;\bfm)=(k_1,\cdots,k_r;m_1,\cdots,m_r)$ be a positive bi-index. We define
$$\Li^{\mu}(\bfk;\bfm;t):=\sum_{0\leqslant n_1\leqslant\cdots\leqslant n_r}
 \frac{\mu^rt^{n_r\mu}}{(n_1\mu+m_{1})^{k_{1}}\cdots(n_r\mu+m_r)^{k_r}},\qquad 0\leqslant t<1,$$
and call it the \emph{$\mu$-multiple polylogarithm function}.
\end{defn}
It is easy to see that the ordinary multiple polylogarithm of single variable \cite[p.\ 78, (3.20)]{ZhaoBook}
\begin{equation}\label{equ:MPLrel}
\Li_\bfk(t):=\sum_{0<n_1<\dots<n_r}\frac{t^{n_r}}{n_1^{k_1}\cdots n_r^{k_r} }=\Li^{\mu=1}(\bfk;1,2,\cdots,r;t).
\end{equation}

Obviously, if $(\bfk;\bfm)$ is an admissible bi-index, then
$$\lim_{t\to1^{-}}\Li^{\mu}(\bfk;\bfm;t)=\zeta^{\mu}(\bfk;\bfm).$$

An important property of the classical multiple polylogarithms is that they satisfy many functional equations.
But more useful to us is that they possess an iterative structure (from which its name comes), which we now extend. Recall that $\bfk'$
is obtained from $\bfk$ by removing its last component.

\begin{lem}\label{lem:Li-Iteration}
Let $(\bfk;\bfm)=(k_1,\cdots,k_r;m_1,\cdots,m_r)$ be a positive bi-index. For all $0\leqslant t<1$, we have
\begin{align*}
t^{m_r}\Li^{\mu}(\bfk',k_r+1;\bfm;t)&=\int_0^t\Li^{\mu}(\bfk;\bfm;t_1)t_1^{m_r}\frac{dt_1}{t_1},\\
t^{m_{r+1}}\Li^{\mu}(\bfk,1;\bfm,m_{r+1};t)&=\int_0^t\Li^{\mu}(\bfk;\bfm;t_1)t_1^{m_{r+1}}\frac{\mu dt_1}{t_1(1-t_1^{\mu})}.
\end{align*}
\end{lem}

\begin{proof}
We have
\begin{align*}
\int_0^t\Li^{\mu}(\bfk;\bfm;t_1)t_1^{m_r}\frac{dt_1}{t_1}
&=\int_0^t\left(\sum_{0\leqslant n_1\leqslant\cdots\leqslant n_r}\frac{\mu^r t_1^{n_r\mu}}{(n_{1}\mu+m_{1})^{k_{1}}\cdots(n_r\mu+m_r)^{k_r}}\right)t_1^{m_r}\frac{dt_1}{t_1}\\
&=\sum_{0\leqslant n_1\leqslant\cdots\leqslant n_r}\int_0^t\frac{\mu^r t_1^{n_r\mu+m_r-1}}{(n_{1}\mu+m_{1})^{k_{1}}\cdots(n_r\mu+m_r)^{k_r}}dt_1\\
&=\sum_{0\leqslant n_1\leqslant\cdots\leqslant n_r}\frac{\mu^r t_1^{n_r\mu+m_r}}{(n_{1}\mu+m_{1})^{k_{1}}\cdots(n_r\mu+m_r)^{k_r+1}}\\
&=t^{m_r}\Li^{\mu}(\bfk',k_r+1;m_1,\cdots,m_r;t)
\end{align*}
which proves the first identity. For the second identity, we have
\begin{align*}
&\int_0^t\Li^{\mu}(\bfk;\bfm;t_1)t_1^{m_{r+1}}\frac{\mu dt_1}{t_1(1-t_1^{\mu})}\\
&=\int_0^t\left(\sum_{0\leqslant n_1\leqslant\cdots\leqslant n_r}\frac{\mu^{r} t_1^{n_r\mu}}{(n_{1}\mu+m_{1})^{k_{1}}\cdots(n_r\mu+m_r)^{k_r}}\right)t_1^{m_{r+1}}\frac{\mu dt_1}{t_1(1-t_1^{\mu})}\\
&=\int_0^t\left(\sum_{0\leqslant n_1\leqslant\cdots\leqslant n_r}\frac{\mu^{r} t_1^{n_r\mu}}{(n_{1}\mu+m_{1})^{k_{1}}\cdots(n_r\mu+m_r)^{k_r}}\cdot\frac{\mu t_1^{m_{r+1}} }{t_1(1-t_1^{\mu})}\right)dt_1\\
&=\int_0^t\left(\sum_{0\leqslant n_1\leqslant\cdots\leqslant n_r}\frac{\mu^{r+1} t_1^{n_r\mu}}{(n_{1}\mu+m_{1})^{k_{1}}\cdots(n_r\mu+m_r)^{k_r}}\cdot\sum_{m=0}^{\infty}t_1^{m\mu+m_{r+1}-1}\right)dt_1\\
&=\int_0^t\left(\sum_{0\leqslant n_1\leqslant\cdots\leqslant n_r\leqslant n_{r+1}}\frac{\mu^{r+1} t_1^{n_{r+1}\mu+m_{r+1}-1}}{(n_{1}\mu+m_{1})^{k_{1}}\cdots(n_r\mu+m_r)^{k_r}}\right)dt_1\qquad (n_{r+1}=n_r+m)\\
&=\sum_{0\leqslant n_1\leqslant\cdots\leqslant n_r\leqslant n_{r+1}}\int_0^t\frac{\mu^{r+1} t_1^{n_{r+1}\mu+m_{r+1}-1}}{(n_{1}\mu+m_{1})^{k_{1}}\cdots(n_r\mu+m_r)^{k_r}}dt_1\\
&=\sum_{0\leqslant n_1\leqslant\cdots\leqslant n_r\leqslant n_{r+1}}\frac{\mu^{r+1} t^{n_{r+1}\mu+m_{r+1}}}{(n_{1}\mu+m_{1})^{k_{1}}\cdots(n_r\mu+m_r)^{k_r}(n_{r+1}\mu+m_{r+1})}\\
&=t^{m_{r+1}}\Li^{\mu}(\bfk,1;\bfm,m_{r+1};t),
\end{align*}
as desired.
\end{proof}

Another notation is needed to describe
the general integral representation of $\mu$-MHZVs.

\begin{defn}
Given a real number $0<t\leqslant1$, we define
$$\Delta^r(t):=\left\{(t_1,\cdots,t_r)\in\R^r\big|0<t_r<\cdots<t_1<t\right\},\qquad0<t\leqslant 1.$$
\end{defn}

When $t=1$, we will simply write $\Delta^r:=\Delta^r(1)$. Furthermore, consider the 1-forms on the open interval $(0,1)$
$$\omega_0(t):=\frac{dt}{t},\qquad\omega_{1,\mu}^{(m)}(t):=\frac{\mu t^mdt}{t(1-t^{\mu})},$$
where $m\in\Z$. Let $(\bfk;\bfm)=(k_1,\cdots,k_r;m_1,\cdots,m_r)\in\N^r\times\N^r$ is a positive bi-index.
Put $s_i=k_1+\cdots+k_i (i=1,\cdots,r)$. For convenience, we write $s_0=0,m_0=0$. Let $\omega_{(\bfk;\bfm)}$
be the measure on the interior of the simplex $\Delta^{|\bfk|}$ given by
$$
\omega^{\mu}_{(\bfk;\bfm)}=
\prod_{i=1}^r\omega_{1,\mu}^{(m_i-m_{i-1})}(t_{s_{r+1-i}})\omega_0(t_{s_{r+1-i}-1})\cdots\omega_0(t_{s_{r-i}+1}).
$$
For example, one has:
\begin{align*}
\omega^{\mu}_{(2;m_1)}&=\frac{\mu t_2^{m_1}dt_2}{t_2(1-t_2^{\mu})}\ \frac{dt_1}{t_1}\\
&=\omega_{1,\mu}^{(m_1)}(t_2)\omega_0(t_1),\\
\omega^{\mu}_{(2,2;m_1,m_2)}&=\frac{\mu t_4^{m_1}dt_4}{t_4(1-t_4^{\mu})}\ \frac{dt_3}{t_3}\ \frac{\mu t_2^{m_2-m_1}dt_2}{t_2(1-t_2^{\mu})}\ \frac{dt_1}{t_1}\\
&=\omega_{1,\mu}^{(m_1)}(t_4)\omega_0(t_3)\omega_{1,\mu}^{(m_2-m_1)}(t_2)\omega_0(t_1),\\
\omega^{\mu}_{(1,3;m_1,m_2)}&=\frac{\mu t_4^{m_1}dt_4}{t_4(1-t_4^{\mu})}\ \frac{\mu t_3^{m_2-m_1}dt_3}{t_3(1-t_3^{\mu})}\ \frac{dt_2}{t_2}\ \frac{dt_1}{t_1}\\
&=\omega_{1,\mu}^{(m_1)}(t_4)\omega_{1,\mu}^{(m_2-m_1)}(t_3)\omega_0(t_2)\omega_0(t_1),\\
\omega^{\mu}_{(2,1,2;m_1,m_2,m_3)}&=\frac{\mu t_5^{m_1}dt_5}{t_5(1-t_5^{\mu})}\ \frac{dt_4}{t_4}\ \frac{\mu t_3^{m_2-m_1}dt_3}{t_3(1-t_3^{\mu})}\ \frac{\mu t_2^{m_3-m_2}dt_2}{t_2(1-t_2^{\mu})}\ \frac{dt_1}{t_1}\\
&=\omega_{1,\mu}^{(m_1)}(t_5)\omega_0(t_4)\omega_{1,\mu}^{(m_2-m_1)}(t_3)\omega_{1,\mu}^{(m_3-m_2)}(t_2)\omega_0(t_1).
\end{align*}

Now the integral representation is a particular case of the next result.

\begin{thm}\label{pzci}
If $(\bfk;\bfm)=(k_1,\cdots,k_r;m_1,\cdots,m_r)$ is a positive bi-index and $0<t<1$, then the following identity holds:
$$t^{m_r}\Li^{\mu}(\bfk;\bfm;t)=\int_{\Delta^{|\bfk|}(t)}\omega^{\mu}_{(\bfk;\bfm)}.$$
\end{thm}

\begin{proof} The proof is by induction on the weight $|\bfk|$. If $|\bfk|=1$, then $(\bfk;\bfm)=(1;m_1)$ and the statement is just the second identity in Lemma~\ref{lem:Li-Iteration}. The inductive step follows from
the iterative procedure given by Lemma~\ref{lem:Li-Iteration}. Indeed, let $(\bfk;\bfm)=(k_1,\cdots,k_r;m_1,\cdots,m_r)$ be a
positive bi-index and assume that the result is true for all bi-indices of lower
weights. If $k_r>1$, then by writing
$$(\bfk_{-};\bfm)=(k_1,\cdots,k_r-1;m_1,\cdots,m_r)$$
and applying the first identity in Lemma~\ref{lem:Li-Iteration}, we have, by the
induction assumption,
$$t^{m_r}\Li^{\mu}(\bfk;\bfm;t)=\int_0^tt_1^{m_r}\Li^{\mu}(\bfk_{-};\bfm;t_1)\frac{dt_1}{t_1}
=\int_0^t\int_{\Delta^{|\bfk_{-}|}(t_1)}\omega^{\mu}_{(\bfk_{-};\bfm)}\frac{dt_1}{t_1}
=\int_{\Delta^{|\bfk|}(t)}\omega^{\mu}_{(\bfk;\bfm)}.$$
If $k_r=1$ then by the second identity in Lemma~\ref{lem:Li-Iteration} and the
induction assumption
\begin{align*}
t^{m_r}\Li^{\mu}(\bfk;\bfm;t)&=\int_0^tt_1^{m_{r-1}}\Li^{\mu}(\bfk_{-};\bfm';t_1)\frac{\mu t_1^{m_r-m_{r-1}}dt_1}{t_1(1-t_1^{\mu})}\\
&=\int_0^t\left(\int_{\Delta^{|\bfk_{-}|}(t_1)}\omega^{\mu}_{(\bfk';\bfm')}\right)\frac{\mu t_1^{m_r-m_{r-1}}dt_1}{t_1(1-t_1^{\mu})}\\
&=\int_{\Delta^{|\bfk|}(t)}\omega^{\mu}_{(\bfk;\bfm)}.
\end{align*}
We have now finished the proof of the theorem.
\end{proof}

\begin{cor}\label{cor:muMZVitIntegral}
If $(\bfk;\bfm)$ is an admissible bi-index, then the following identity holds:
$$\zeta^{\mu}(\bfk;\bfm)=\int_{\Delta^{|\bfk|}}\omega^{\mu}_{(\bfk;\bfm)}.$$
For any special triple $(\bfk;\bfm;\tbfm)$, we have
\begin{equation*}
\calD^{\mu}(\bfk;\bfm;\tbfm)=\int_{\Delta^{|\bfk|}}(\omega^{\mu}_{(\bfk;\bfm)}-\omega^{\mu}_{(\bfk;\tbfm)}).
\end{equation*}
\end{cor}

\begin{proof} By assumption, if $(\bfk;\bfm)$ is an admissible bi-index then
\begin{equation*}
\zeta^{\mu}(\bfk;\bfm)=\lim_{t\to1^{-}}t^{m_r}\Li^{\mu}(\bfk;\bfm;t)=
\lim_{t\to1^{-}}\int_{\Delta^{|\bfk|}(t)}\omega^{\mu}_{(\bfk;\bfm)}=
\int_{\Delta^{|\bfk|}}\omega^{\mu}_{(\bfk;\bfm)}.
\end{equation*}
If $(\bfk;\bfm;\tbfm)$ is a special triple, then
\begin{align*}
\calD^{\mu}(\bfk;\bfm;\tbfm)&
=\lim_{t\to1^{-}}(t^{m_r}\Li^{\mu}(\bfk;\bfm;t)-t^{\tm_r}\Li^{\mu}(\bfk;\tbfm;t))\\
&=\lim_{t\to1^{-}}\int_{\Delta^{|\bfk|}(t)}(\omega^{\mu}_{(\bfk;\bfm)}-\omega^{\mu}_{(\bfk;\tbfm)}),
\end{align*}
as desired.
\end{proof}

\begin{eg} We have the following iterated integral expressions of $\mu$-MHZVs:
\begin{align*}
\zeta^{\mu}(2;m_1)&=\int\limits_{0<t_2<t_1<1}\frac{\mu t_2^{m_1}dt_2}{t_2(1-t_2^{\mu})}\ \frac{dt_1}{t_1},\\
\zeta^{\mu}(2,2;m_1,m_2)&=\int\limits_{0<t_4<t_3<t_2<t_1<1}\frac{\mu t_4^{m_1}dt_4}{t_4(1-t_4^{\mu})}\ \frac{dt_3}{t_3}\ \frac{\mu t_2^{m_2-m_1}dt_2}{t_2(1-t_2^{\mu})}\ \frac{dt_1}{t_1},\\
\zeta^{\mu}(1,3;m_1,m_2)&=\int\limits_{0<t_4<t_3<t_2<_1<1}\frac{\mu t_4^{m_1}dt_4}{t_4(1-t_4^{\mu})}\ \frac{\mu t_3^{m_2-m_1}dt_3}{t_3(1-t_3^{\mu})}\ \frac{dt_2}{t_2}\ \frac{dt_1}{t_1},\\
\zeta^{\mu}(2,1,2;m_1,m_2,m_3)&=\int\limits_{0<t_5<t_4<t_3<t_2<t_1<1}\frac{\mu t_5^{m_1}dt_5}{t_5(1-t_5^{\mu})}\ \frac{dt_4}{t_4}\ \frac{\mu t_3^{m_2-m_1}dt_3}{t_3(1-t_3^{\mu})}\ \frac{\mu t_2^{m_3-m_2}dt_2}{t_2(1-t_2^{\mu})}\ \frac{dt_1}{t_1},\\
\calD^{\mu}(2,1;m,m_1;m,\tm_1)&=\int\limits_{0<t_3<t_2<t_1<1}\frac{\mu t_3^{m}dt_3}{t_3(1-t_3^{\mu})}\ \frac{dt_2}{t_2}\left(\frac{\mu t_1^{m_1-m}dt_1}{t_1(1-t_1^{\mu})}-\frac{\mu t_1^{\tm_1-m}dt_1}{t_1(1-t_1^{\mu})}\right).
\end{align*}
\end{eg}

\subsection{Bi-indices and binary sequences}
To exploit the preceding results to derive relations among polylogarithms, and in particular among $\mu$-MHZVs,
we need a new notation. This will enable us to go from bi-indices to some binary sequences
and vice versa.

\begin{defn}
To each positive bi-index $(\bfk;\bfm)=(k_1,\cdots,k_r;m_1,\cdots,m_r)$ we attach the
positive binary sequence
$$\bs(\bfk;\bfm)=(1,0_{k_1-1},1,0_{k_2-1},\cdots,1,0_{k_r-1};
 m_1,0_{k_1-1},m_2-m_1,0_{k_2-1},\cdots,m_r-m_{r-1},0_{k_r-1})$$
where $0_{k}$ means that the entry zero is repeated $k$ times.
\end{defn}
It clearly has an inverse map
$$\text{bs}^{-1}(1,0_{k_1},1,0_{k_2},\cdots,1,0_{k_r};m_1,0_{k_1},m_2,0_{k_2},\cdots,m_r,0_{k_r})
=(k_1',\cdots,k_r';m_1',\cdots,m_r')$$
where $k_i'=k_i+1,m_i'=m_1+\cdots+m_i,i=1,\cdots,r$.

\begin{eg}
\begin{align*}
\bs(2;m_1)&=(1,0;m_1,0),\\
\bs(2,2;m_1,m_2)&=(1,0,1,0;m_1,0,m_2-m_1,0),\\
\bs(1,3;m_1,m_2)&=(1,1,0,0;m_1,m_2-m_1,0,0),\\
\bs(2,1,2;m_1,m_2,m_3)&=(1,0,1,1,0;m_1,0,m_2-m_1,m_3-m_2,0).
\end{align*}
\end{eg}

The following shuffle product formula for iterated integral was first discovered by K.T.\ Chen \cite[(1.5.1)]{KTChen1971}.

\begin{lem}\label{oogj}
For any $r,s\in\Z_{\geqslant0}$, $0<t<1$, and differential 1-forms $\phi_i$ $(i=1,\cdots,r+s)$ on $(0,1)$, we have
\begin{align*}
&\int_{\Delta^r(t)}\phi_1(t_1)\cdots\phi_r(t_r)\cdot\int_{\Delta^s(t)}\phi_{r+1}(t_{r+1})\cdots\phi_{r+s}(t_{r+s})
=\sum_{\sigma\in\shuffle(r,s)}\int_{\Delta^{r+s}(t)}\phi_{\sigma^{-1}}(t_1)\cdots\phi_{\sigma^{-1}(r+s)}(t_{r+s})
\end{align*}
\end{lem}

\begin{defn}
Let $(\boldsymbol{s}_1;\boldsymbol{\tau}_1)=(s_1,\cdots,s_r;\tau_1,\cdots,\tau_r)$ and $(\boldsymbol{s}_2;\boldsymbol{\tau}_2)=(s_{r+1},\cdots,s_{r+s};$$\tau_{r+1},\cdots,\tau_{r+s})$ be two binary sequences. Define
\begin{multline*}
\shuffle((\boldsymbol{s}_1;\boldsymbol{\tau}_1),(\boldsymbol{s}_2;\boldsymbol{\tau}_2);(\boldsymbol{s};\boldsymbol{\tau}))\\
:={\rm Card}\Big(\big\{\sigma\in\shuffle(r,s)\big|(\boldsymbol{s};\boldsymbol{\tau})=
(s_{\sigma^{-1}(1)},\cdots,s_{\sigma^{-1}(r+s)};\tau_{\sigma^{-1}(1)},\cdots,\tau_{\sigma^{-1}(r+s)})\big\}\Big),
\end{multline*}
where Card($S$) is the cardinality of the set $S$.
\end{defn}

\begin{thm}\label{ignn}
Let $(\bfk_1;\bfm_1)$ and $(\bfk_2;\bfm_2)$ be positive bi-indices. Then for all $0<t<1$ we have
\begin{align*}
\Li^{\mu}(\bfk_1;\bfm_1;t)\cdot\Li^{\mu}(\bfk_2;\bfm_2;t)=
\sum_{(\bfk;\bfm)}\shuffle(\bs(\bfk_1;\bfm_1),\bs(\bfk_2;\bfm_2);\bs(\bfk;\bfm))\Li^{\mu}(\bfk;\bfm;t).
\end{align*}
\end{thm}

\begin{proof} This follows from Lemma \ref{oogj}.
\end{proof}

\begin{cor}\label{utgv}
The integral representations of (quasi-)$\mu$-MHZVs in Corollary \ref{cor:muMZVitIntegral}
provides another product on $\calZ^{\mu}$. In particular, we have
\begin{align*}
\zeta^{\mu}(\bfk_1;\bfm_1)\cdot\zeta^{\mu}(\bfk_2;\bfm_2)=
\sum_{(\bfk;\bfm)}\shuffle(\bs(\bfk_1;\bfm_1),\bs(\bfk_2;\bfm_2);\bs(\bfk;\bfm))\zeta^{\mu}(\bfk;\bfm),
\end{align*}
where $\bfk$ must be admissible if $\bfk_1$ and $\bfk_2$ are both admissible. And we have
similar results for product of quasi-$\mu$-MHZVs.
\end{cor}
\begin{proof}
This follows from Theorem~\ref{ignn} immediately by taking $t\to 1^-$. Notice that
$\bfk$ is non-admissible if and only if the last 1-form in the integral representation of $(\bfk;\bfm)$
has the form $\om_{1,\mu}^{(\cdots)}$. By the shuffle property, this 1-form comes from either of the two
last 1-forms of the integral representations for $\zeta^{\mu}(\bfk_j;\bfm_j)$, $j=1,2$.
\end{proof}

\begin{eg}\label{dfnc}
We have
\begin{align*}
\zeta^{\mu}(2;m_1)\cdot\zeta^{\mu}(2;m_2)&=\zeta^{\mu}(2,2;m_1,m_1+m_2)+\zeta^{\mu}(2,2;m_2,m_1+m_2)\\
&\quad+2\zeta^{\mu}(1,3;m_1,m_1+m_2)+2\zeta^{\mu}(1,3;m_2,m_1+m_2).
\end{align*}
If we take $m_1=m_2=m$, then
$$\zeta^{\mu}(2;m)\cdot\zeta^{\mu}(2;m)=2\zeta^{\mu}(2,2;m,2m)+4\zeta^{\mu}(1,3;m,2m),$$
which generalizes the classical result that
$$\zeta(2)^2=2\zeta(2,2)+4\zeta(1,3).$$
As another example, we have
\begin{align*}
\zeta^{\mu}(2;m)\cdot\calD^{\mu}(1;m_1;m_2)&=\zeta^{\mu}(1,2;m_1,m_1+m)
-\zeta^{\mu}(1,2;m_2,m_2+m)+\zeta^{\mu}(1,2;m,m_1+m)\\
&\quad-\zeta^{\mu}(1,2;m,m_2+m)+\calD^{\mu}(2,1;m,m_1+m;m,m_2+m).
\end{align*}
\end{eg}

\section{Word algebra of $\mu$-multiple Hurwitz zeta values} \label{sec:algSetup}
In the previous
sections, we saw two methods to express a product of $\mu$-MHZVs as a
linear combination of $\mu$-MHZVs. As we saw
in Example~\ref{ofns} and Example~\ref{dfnc}, they may give different $\Q[\mu]$-linear combinations
for the same product of $\mu$-MHZVs, thus leading to linear relations among
them. The $\mu$-stuffle multiplication can be easily written in terms of bi-indices as in
Theorem~\ref{qqox}, while the shuffle multiplication is expressed more conveniently using
binary sequences as in Theorem~\ref{ignn}. We now set up the algebraic framework to elucidate the
combinatorial structure of these two kinds of products. We first transform them to products in
two formal word algebras that encode the $\mu$-stufffle and the shuffle multiplications, respectively.

\subsection{Word algebra for $\mu$-stuffle product}
Let $Y=\{y_{k,m}|k\in\N,m\in\Q_{>0}\}$ be a countable alphabet whose elements are called \emph{letters}.
Let $\Q[\mu]Y$ be the free $\Q[\mu]$-module with $Y$ as a basis. For any positive bi-index
$(\bfk;\bfm)=(k_1,\cdots,k_r;m_1,\cdots,m_r)$ we write $y_{\bfk;\bfm}= y_{k_1,m_1}\cdots y_{k_r,m_r}$.
 Let $\Q[\mu]\langle Y\rangle$ be the non-commutative polynomial algebra on the words over $Y$, that is,
\begin{equation*}
 \Q[\mu]\langle Y\rangle:=\langle y_{\bfk;\bfm}|(\bfk;\bfm): \text{ positive bi-index} \rangle_{\Q[\mu]}
\end{equation*}
is the $\Q[\mu]$-module with the set of \emph{words} in the letters of $Y$ as a basis, along with the concatenation product
$$
(y_{k_1,m_1}\cdots y_{k_r,m_r})\cdot(y_{k_1',m_1'}\cdots y_{k_s',m_s'})=y_{k_1,m_1}\cdots y_{k_r,m_r}y_{k_1',m_1'}\cdots y_{k_s',m_s'}.
$$
We say that a word $w=y_{k_1,m_1}\cdots y_{k_r,m_r}$ has \emph{length} $\ell(w)=r$. We put 1 as the empty word and set $\ell(1)=0$.

\begin{defn}
Let $\Q[\mu]Y$ as above. We define a $\Q[\mu]$-bilinear map
$$-\odot-:\Q[\mu]Y\times\Q[\mu]Y\longrightarrow\Q[\mu]Y$$
as follows:
$$y_{k_1,m_1}\odot y_{k_2,m_2}=
\left\{
 \begin{array}{ll}
 \displaystyle \sum\limits_{i=1}^{k_1}a_i \, y_{i,m_1}+\sum\limits_{j=1}^{k_2}b_jy_{j,m_2}, & \hbox{if $m_1\neq m_2$;} \\
 y_{k_1+k_2,m}, & \hbox{if $m_1=m_2=m$,}
 \end{array}
\right. $$
where the coefficients $a_1,\cdots,a_{k_1},b_1,\cdots,b_{k_2}\in\Q$ are determined by the following decomposition
\begin{equation*}
\frac{1}{(x+m_1)^{k_1}(x+m_2)^{k_2}}=\sum_{i=1}^{k_1}\frac{a_i}{(x+m_1)^i}+\sum_{j=1}^{k_2}\frac{b_j}{(x+m_2)^j},
\end{equation*}
according to Lemma~\ref{lem:parFrac}
\end{defn}

\begin{lem}
The $
\Q[\mu]$-bilinear map $-\odot-:\Q[\mu]Y\times\Q[\mu]Y\longrightarrow\Q[\mu]Y$ is a commutative and associative product.
\end{lem}

\begin{proof}
The commutativity is clear so that we only need to prove the associativity.
Further, by commutativity, we only need to consider the following three cases:

(i). If $m_1=m_2=m_3=m$, then we have
\begin{align*}
(y_{k_1,m}\odot y_{k_2,m})\odot y_{k_3,m}&=y_{k_1+k_2,m}\odot y_{k_3,m}\\
&=y_{k_1+k_2+k_3,m}\\
&=y_{k_1,m}\odot y_{k_2+k_3,m}\\
&=y_{k_1,m}\odot(y_{k_2,m}\odot y_{k_3,m}).
\end{align*}

(ii). If $m_1=m_2=m\neq m_3$, then we have two different decomposition
\begin{align*}
\frac{1}{(x+m)^{k_1+k_2}(x+m_3)^{k_3}}=\sum_{i=1}^{k_1+k_2}\frac{a_i}{(x+m)^i}+\sum_{j=1}^{k_3}\frac{b_j}{(x+m_3)^j},
\end{align*}
and
\begin{align*}
&\frac{1}{(x+m)^{k_1+k_2}(x+m_3)^{k_3}}\\
&=\frac{1}{(x+m)^{k_1}}\left(\frac{1}{(x+m)^{k_2}(x+m_3)^{k_3}}\right)\\
&=\frac{1}{(x+m)^{k_1}}\left(\sum_{i=1}^{k_2}\frac{a_{k_1+i}'}{(x+m)^i}+\sum_{j=1}^{k_3}\frac{b_j'}{(x+m_3)^j}\right)\\
&=\sum_{i=1}^{k_2}\frac{a_{k_1+i}'}{(x+m)^{k_1+i}}+\sum_{j=1}^{k_3}\frac{b_j'}{(x+m)^{k_1}(x+m_3)^j}\\
&=\sum_{i=1}^{k_2}\frac{a_{k_1+i}'}{(x+m)^{k_1+i}}+\sum_{j=1}^{k_3}b_j'\left(\sum_{t=1}^{k_1}\frac{c_t}{(x+m)^t}+\sum_{s=1}^{j}\frac{d_s}{(x+m_3)^s}\right)\\
&=\sum_{i=1}^{k_2}\frac{a_{k_1+i}'}{(x+m)^{k_1+i}}+\sum_{t=1}^{k_1}\left(c_t\sum_{j=1}^{k_3}b_j'\right)\frac{1}{(x+m)^t}
 +\sum_{j=1}^{k_3}\left(d_j\sum_{s=j}^{k_3}b_s'\right)\frac{1}{(x+m_3)^j}.
\end{align*}
We conclude that
$$a_i=\begin{cases}
c_i\sum\limits_{j=1}^{k_3}b_j',\qquad& \text{if }1\leqslant i \leqslant k_1;\\
a_{k_1+i}', & \text{if } k_1+1\leqslant i\leqslant k_2,
\end{cases}$$
and $b_j=d_j\sum\limits_{s=j}^{k_3}b_s'$ for all $1\leqslant j\leqslant k_3$. Hence,
\begin{align*}
(y_{k_1,m}\odot y_{k_2,m})\odot y_{k_3,m_3}=y_{k_1+k_2,m}\odot y_{k_3,m_3}=\sum_{i=1}^{k_1+k_2}a_i \, y_{i,m}+\sum_{j=1}^{k_3}b_jy_{j,m_3},
\end{align*}
and
\begin{align*}
y_{k_1,m}\odot\left(y_{k_2,m}\odot y_{k_3,m_3}\right)&=y_{k_1,m}\odot\left(\sum_{i=1}^{k_2}a_{k_2+i}'y_{i,m}+\sum_{j=1}^{k_3}b_j'y_{j,m_3}\right)\\
&=\sum_{i=1}^{k_2}a_{k_2+i}'y_{k_1+i,m}+\sum_{j=1}^{k_3}b_j'y_{k_1,m}\odot y_{j,m_3}\\
&=\sum_{i=1}^{k_2}a_{k_2+i}'y_{k_1+i,m}+\sum_{j=1}^{k_3}b_j'\left(\sum_{t=1}^{k_1}c_ty_{t,m}+\sum_{s=1}^jd_sy_{s,m_3}\right)\\
&=\sum_{i=1}^{k_2}a_{k_1+i}'y_{k_1+i,m}+\sum_{t=1}^{k_1}\left(c_t\sum_{j=1}^{k_3}b_j'\right)y_{t,m}+\sum_{j=1}^{k_3}\left(d_j\sum_{s=j}^{k_3}b_s'\right)y_{j,m_3}\\
&=\sum_{i=k_1+1}^{k_2}a_i \, y_{i,m}+\sum_{i=1}^{k_1}a_i \, y_{i,m}+\sum_{j=1}^{k_3}b_jy_{j,m_3}\\
&=(y_{k_1,m}\odot y_{k_2,m})\odot y_{k_3,m_3}.
\end{align*}

(iii). If $m_i\neq m_j,i\neq j,1\leqslant i,j\leqslant3$, then we can apply the same technique to get
\begin{align*}
\frac{1}{(x+m_1)^{k_1}(x+m_2)^{k_2}(x+m_3)^{k_3}}&=\left(\sum_{i=1}^{k_1}\frac{a_i}{(x+m_1)^i}+\sum_{j=1}^{k_2}\frac{b_j}{(x+m_2)^j}\right)\frac{1}{(x+m_3)^{k_3}}\\
&=\sum_{i=1}^{k_1}a_i\left(\sum_{i_1=1}^i\frac{c_{i_1}}{(x+m_1)^{i_1}}+\sum_{i_2=1}^{k_3}\frac{d_{i_2}}{(x+m_3)^{i_2}}\right)\\
&\quad+\sum_{j=1}^{k_2}b_j\left(\sum_{j_1=1}^j\frac{e_{j_1}}{(x+m_2)^{j_1}}+\sum_{j_2=1}^{k_3}\frac{f_{j_2}}{(x+m_3)^{j_2}}\right).
\end{align*}
Thus,
\begin{align*}
(y_{k_1,m_1}\odot y_{k_2,m_2})\odot y_{k_3,m_3}&=\sum_{i=1}^{k_1}a_i\left(\sum_{i_1=1}^ic_{i_1}y_{i_1,m_1}+\sum_{i_2=1}^{k_2}d_{i_2}y_{i_2,m_3}\right)\\
&\quad+\sum_{j=1}^{k_2}b_j\left(\sum_{j_1=1}^je_{j_1}y_{j_1,m_2}+\sum_{j_2=1}^{k_3}f_{j_2}y_{j_2,m_3}\right).
\end{align*}
On the other hand,,
\begin{align*}
\frac{1}{(x+m_1)^{k_1}(x+m_2)^{k_2}(x+m_3)^{k_3}}&=\frac{1}{(x+m_1)^{k_1}}\left(\sum_{i=1}^{k_2}\frac{a_i'}{(x+m_2)^i}+\sum_{j=1}^{k_3}\frac{b_j'}{(x+m_3)^j}\right)\\
&=\sum_{i=1}^{k_2}a_i'\left(\sum_{i_1=1}^{k_1}\frac{c_{i_1}'}{(x+m_1)^{i_1}}+\sum_{i_2=1}^{i}\frac{d'_{i_2}}{(x+m_2)^{i_2}}\right)\\
&\quad+\sum_{j=1}^{k_3}b_j'\left(\sum_{j_1=1}^{k_1}\frac{e'_{j_1}}{(x+m_1)^{j_1}}+\sum_{j_2=1}^{j}\frac{f'_{j_2}}{(x+m_3)^{j_2}}\right).
\end{align*}
Hence
\begin{align*}
y_{k_1,m_1}\odot(y_{k_2,m_2}\odot y_{k_3,m_3})&=\sum_{i=1}^{k_2}a_i'\left(\sum_{i_1=1}^{k_1}c_{i_1}'y_{i_1,m_1}+\sum_{i_2=1}^{i}d'_{i_2}y_{i_2,m_2}\right)\\
&\quad+\sum_{j=1}^{k_3}b_j'\left(\sum_{j_1=1}^{k_1}e'_{j_1}y_{j_1,m_1}+\sum_{j_2=1}^{j}f'_{j_2}y_{j_2,m_3}\right).
\end{align*}
By comparing the coefficients, we arrive at the following conclusion
$$
(y_{k_1,m_1}\odot y_{k_2,m_2})\odot y_{k_3,m_3}=y_{k_1,m_1}\odot(y_{k_2,m_2}\odot y_{k_3,m_3}).
$$
\end{proof}

\begin{defn}
Let $\Q[\mu]\langle Y\rangle$ be as above. The $\mu$-stuffle product $\must$ on $\Q[\mu]\langle Y\rangle$
is a $\Q[\mu]$-bilinaer map
$$
-\must-:\Q[\mu]\langle Y\rangle\times\Q[\mu]\langle Y\rangle\longrightarrow\Q[\mu]\langle Y\rangle,
$$
defined recursively by
\begin{align*}
1\must w&\, =w\must1=w\qquad \forall w\in\Q[\mu]\langle Y\rangle,\\
ux\must vy &\, =(u\must vy)x+(ux\must v)y-\mu\cdot(u\must v)(x\odot y)
\end{align*}
where $x,y\in Y$ and $u,v$ are words in $\Q[\mu]\langle Y\rangle$.
\end{defn}

\begin{eg}\label{1212} We consider the product $y_{2,m_1}\must y_{2,m_2}$. If $m_1=m=m_2$, then
\begin{align*}
y_{2,m_1}\must y_{2,m_2}&=y_{2,m}\must y_{2,m}\\
&=(1\must y_{2,m})y_{2,m}+(y_{2,m}\must 1)y_{2,m}-\mu\cdot(1\must 1)y_{4,m}\\
&=y_{2,m}^2+y_{2,m}^2-\mu\cdot y_{4,m}\\
&=2y_{2,m}^2-\mu\cdot y_{4,m}.
\end{align*}
If $m_1\neq m_2$, then
\begin{align*}
y_{2,m_1}\must y_{2,m_2}&=(1\must y_{2,m_2})y_{2,m_1}+(y_{2,m_1}\must 1)y_{2,m_2}
-\mu\cdot(1\must 1)\left(y_{2,m_1}\odot y_{2,m_2}\right)\\
&=y_{2,m_2}y_{2,m_1}+y_{2,m_1}y_{2,m_2}\\
&\quad-\mu\left(\frac{2}{(m_1-m_2)^3}(y_{1,m_1}-y_{1,m_2})+\frac{1}{(m_1-m_2)^2}(y_{2,m_1}+y_{2,m_2})\right).
\end{align*}
\end{eg}

\begin{thm}
$(\Q[\mu]\langle Y\rangle,\must)$ is a commutative and associative $\Q[\mu]$-algebra with unit.
\end{thm}

\begin{proof} First, we check the commutativity
\begin{equation}\label{equ:stuffleCommute}
a\must b=b\must a
\end{equation}
by induction on $\ell(a)+\ell(b)$. If either $a$ or $b$ is the empty word (i.e., 1), then \eqref{equ:stuffleCommute} is trivial.
It thus suffices to consider the case $a=ux, b=vy$, where $x,y$ are letters in $Y$ and $u,v$ are words in $\Q[\mu]\langle Y\rangle$.
Then, by definition of the product $\must $ we get
\begin{align*}
a\must b-b\must a =((u\must b)-(b\must u))x+(a\must v-v\must a)y -\mu(u\must v)(x\odot y-y\odot x)=0
\end{align*}
by induction assumption and the commutativity of $\odot$. This completes the proof of \eqref{equ:stuffleCommute}.

Next, we check the associativity
$$(a\must b)\must c=a\must (b\must c)$$
also by induction on $\ell(a)+\ell(b)+\ell(c)$. If one of $a,b,c$ is the empty word (i.e., 1), then it's clear. It thus suffices to consider the case $a=ux, b=vy$ and $c=wz$
with letters $x,y,z\in Y$ and words $u,v,w\in\Q[\mu]\langle Y\rangle$. Then, by definition of the product $\must $ and the induction hypothesis, we get
\begin{align*}
(a\must b)\must c&=[(u\must b)x+(a\must v)y-\mu\cdot(u\must v)(x\odot y)]\must c\\
&=((u\must b)\must c)x+((u\must b)x\must w)z-\mu((u\must b)\must {w})(x\odot z)\\
&\quad+((a\must v)\must c)y+((a\must v)y\must w)z-\mu((a\must v)\must {w})(y\odot z)\\
&\quad-\mu ((u\must v)\must c)(x\odot y)-\mu((u\must v)(x\odot y)\must w)z+\mu^2((u\must v)\must w)((x\odot y)\odot z),
\end{align*}
and
\begin{align*}
a\must (b\must c)&=a\must [(v\must c)y+(b\must w)z-\mu(v\must w)(y\odot z)]\\
&=(u\must (v\must c)y)x+(a\must (v\must c))y-\mu(u\must (v\must c))(x\odot y)\\
&\quad+(u\must (b\must w)z)x+(a\must (b\must w))z-\mu(u\must (b\must w))(x\odot z)\\
&\quad-\mu(u\must (v\must w)(y\odot z))x-\mu(a\must (v\must w))(y\odot z)+\mu^2(u\must (v\must w))(x\odot(y\odot z)),
\end{align*}
Hence, by induction hypothesis
\begin{align*}
&(a\must b)\must c-a\must (b\must c)\\
&=((u\must b)\must c)x+((u\must b)x\must w)z+((a\must v)y\must w)z-\mu((u\must v)(x\odot y)\must w)z\\
&\quad-(u\must (v\must c)y)x-(u\must (b\must w)z)x-(a\must (b\must w))z+\mu(u\must (v\must w)(y\odot z))x\\
&=((u\must b)\must c)x+([(u\must b)x+(a\must v)y-\mu(u\must v)(x\odot y)]\must w)z\\
&\quad-(a\must (b\must w))z-(u\must [(v\must c)y+(b\must w)z-\mu(v\must w)(y\odot z)])x\\
&=((u\must b)\must c)x+((a\must b)\must w)z-(a\must (b\must w))z-(u\must (b\must c))x\\
&=0.
\end{align*}
We have now completed the proof.
\end{proof}

\textbf{Question}. Describe $\text{Aut}((\Q[\mu]\langle Y\rangle,\must))$, $\text{End}((\Q[\mu]\langle Y\rangle,\must))$, and $\text{Der}((\Q[\mu]\langle Y\rangle,\must))$.

\begin{defn}\label{defn:Yadmissible}
A word $y_{\bfk;\bfm}$ is called \emph{admissible} if $(\bfk;\bfm)$ is an admissible bi-index.
By convention, the empty word 1 will also be considered to be admissible. We will denote by $\Q[\mu]\langle Y\rangle_{\rm pre}^0$
the $\Q[\mu]$-submodule of $\Q[\mu]\langle Y\rangle$ generated by admissible words.
\end{defn}

\textbf{Unfortunately, $(\Q[\mu]\langle Y\rangle_{\rm pre}^0,\must)$ does not form a subalgebra of $(\Q[\mu]\langle Y\rangle,\must)$}. For instance, if $m_1\neq m_2$, Example~\ref{1212} showed that
\begin{align*}
y_{2,m_1}\must y_{2,m_2}&=y_{2,m_1}y_{2,m_2}+y_{2,m_2}y_{2,m_1}\\
&\quad-\mu\left(\frac{2}{(m_1-m_2)^3}(y_{1,m_1}-y_{1,m_2})+\frac{1}{(m_1-m_2)^2}(y_{2,m_1}+y_{2,m_2})\right).
\end{align*}
To remedy this situation, we introduce another notation.

\begin{defn}\label{defn:Yquasi-admissible}
Let $(\bfk;\bfm;\tbfm)$ be a special triple where
$\bfk'=(k_1,\cdots,k_{r-1})$, $\bfm=(m_1,\cdots,m_r)$, and $\tbfm=(\bfm',\tm_r)$. Then the element
$$
y_{\bfk;\bfm}-y_{\bfk;\tbfm} :=y_{k_1,m_1}\cdots y_{k_{r-1},m_{r-1}}(y_{1,m_{r}}-y_{1,\tm_r})
$$
is called a \emph{quasi-admissible} word. The submodule generated by all admissible words and all
quasi-admissible words over $\Q[\mu]$ is denoted by $\Q[\mu]\langle Y\rangle^0$.
\end{defn}

\begin{defn}
Fix $\mu\in\CC$ with $\Ree(\mu)>0$. Let $M\in\R_{\geqslant0}$ and $(\bfk;\bfm)=(k_1,\cdots,k_r;m_1,\cdots,m_r)$
be a positive bi-index. Define the $\Q[\mu]$-linear map
\begin{align}\label{equ:zeta*M}
\zeta^{\mu}_{*,M}:(\Q[\mu]\langle Y\rangle,\must)&\longrightarrow\R,\\
y_{\bfk;\bfm}&\longmapsto\zeta^{\mu}_{M}(\bfk;\bfm).     \notag
\end{align}
\end{defn}

\begin{thm}
The map $\zeta^{\mu}_{*,M}:(\Q[\mu]\langle Y\rangle,\must)\longrightarrow\R$ is a $\Q[\mu]$-algebra homomorphism
\end{thm}

\begin{proof} This follows from Theorem~\ref{thm:ddog}.
\end{proof}

\begin{prop} \label{prop:Y0subalgebra}
The $\Q[\mu]$-submodule $(\Q[\mu]\langle Y\rangle^0,\must)$ is a subalgebra of $(\Q[\mu]\langle Y\rangle,\must)$.
\end{prop}

\begin{proof}
We only need to show that
$$a\must b\in\Q[\mu]\langle Y\rangle^0\qquad\forall a,b\in\Q[\mu]\langle Y\rangle^0$$
we consider the following three cases:

(i). If $a=uy_{k_r,m_r}, b=vy_{k_s',m_s'}, k_r,k_s'>1$, then we have
$$a\must b=(u\must b)y_{k_r,m_r}+(a\must v)y_{k_s',m_s'}-\mu(u\must v)(y_{k_r,m_r}\odot y_{k_s',m_s'})$$
form the definition of $\odot$. Thus
$$y_{k_r,m_r}\odot y_{k_s',m_s'}=
\begin{cases}
a_1y_{1,m_r}+b_1y_{1,m_s'}+\sum\limits_{i=2}^{k_r}a_i \, y_{i,m_r}+\sum\limits_{j=2}^{k_s'}b_jy_{j,m_s'},\qquad& \text{if }m_r\neq m_s';\\
y_{k_r+k_s',m},& \text{if }m_r=m=m_s'.
\end{cases}$$
Noticing that $b_1=-a_1$, we get
$$y_{k_r,m_r}\odot y_{k_s',m_s'}=
\begin{cases}
a_1(y_{1,m_r}-y_{1,m_s'})+\sum\limits_{i=2}^{k_r}a_i \, y_{i,m_r}+\sum\limits_{j=2}^{k_s'}b_jy_{j,m_s'},\qquad& \text{if }m_r\neq m_s';\\
y_{k_r+k_s',m},& \text{if }m_r=m=m_s'.
\end{cases}$$
Hence, $a\must b\in\Q[\mu]\langle Y\rangle^0$.

(ii). If $a=uy_{k_r,m_r},b=v(y_{1,m_s'}-y_{1,m_s''}),k_r>1$, then we have
\begin{align*}
a\must b&=(u\must b)y_{k_r,m_r}+(a\must v)(y_{1,m_s'}-y_{1,m_s''})-\mu(u\must v)(y_{k_r,m_r}\odot(y_{1,m_s'}-y_{1,m_s''}))\\
&=(u\must b)y_{k_r,m_r}+(a\must v)(y_{1,m_s'}-y_{1,m_s''})-\mu(u\must v)(y_{k_r,m_r}\odot y_{1,m_s'}-y_{k_r,m_r}\odot y_{1,m_s''}).
\end{align*}
Thus
$$y_{k_r,m_r}\odot y_{1,m_s'}=\begin{cases}
a_1(y_{1,m_r}-y_{1,m_s'})+\sum\limits_{i=2}^{k_r}a_i \, y_{i,m_r},\qquad& \text{if }m_r\neq m_s';\\
y_{k_r+1,m},& \text{if }m_r=m=m_s',
\end{cases}$$
and
$$y_{k_r,m_r}\odot y_{1,m_s''}=\begin{cases}
a_1'(y_{1,m_r}-y_{1,m_s''})+\sum\limits_{i=2}^{k_r}a_i'y_{i,m_r},\qquad& \text{if }m_r\neq m_s'';\\
y_{k_r+1,m},& \text{if }m_r=m=m_s''.
\end{cases}$$
Hence $a\must b\in\Q[\mu]\langle Y\rangle^0$.

(iii). If $a=u(y_{1,m_r}-y_{1,\tm_r}),b=v(y_{1,m_s''}-y_{1,m_s'''})$, then we have
\begin{multline*}
a\must b=(u\must b)(y_{1,m_r}-y_{1,\tm_r})+(a\must v)(y_{1,m_s''}-y_{1,m_s'''})\\
 -\mu(u\must v)((y_{1,m_r}-y_{1,\tm_r})\odot(y_{1,m_s''}-y_{1,m_s'''})).
\end{multline*}
Notice that
\begin{align*}
(y_{1,m_r}-y_{1,\tm_r})\odot(y_{1,m_s''}-y_{1,m_s'''})=&\, y_{1,m_r}\odot y_{1,m_s''}-y_{1,m_r}\odot y_{1,m_s'''}\\
& -y_{1,\tm_r}\odot y_{1,m_s''}+y_{1,m_s'}\odot y_{1,m_s'''}.
\end{align*}
From the definition of $\odot$, we have
$$y_{1,m}\odot y_{1,m'}=\begin{cases}
\frac{1}{m'-m}(y_{1,m}-y_{1,m'}),\qquad&\text{if } m\neq m';\\
y_{2,m},&\text{if }m=m'.
\end{cases}$$
Hence $a\must b\in\Q[\mu]\langle Y\rangle^0$.

\end{proof}

\subsection{Admissible words for $\mu$-stuffle product}

In this section, we will define the evaluation map $\zeta^{\mu}_{*}:(\Q[\mu]\langle Y\rangle^0,\must)\, \longrightarrow(\calZ^{\mu},\cdot)$. First, we notice that there is a natural
grading on $\Q[\mu]\langle Y\rangle$ by the length
\begin{equation*}
\Q[\mu]\langle Y\rangle=\bigoplus_{r=0}^{\infty}\Q[\mu]\langle Y\rangle_r
\end{equation*}
where $\Q[\mu]\langle Y\rangle_0=\Q 1$ with $1$ being the empty word and for all $r>0$
\begin{equation*}
\Q[\mu]\langle Y\rangle_r:=\big\langle y_{\bfk;\bfm}\big| (\bfk;\bfm): \text{ positive bi-index},\, \ell(\bfk)=r \big\rangle_{\Q[\mu]}.
\end{equation*}
Let $\Q[\mu]\langle Y\rangle^{0,\rm quasi}$ denote the $\Q[\mu]$-subspace generated by all quasi-admissible word, i.e.,
\begin{equation*}
\Q[\mu]\langle Y\rangle^{0,\rm quasi}:=\big\langle y_{\bfk;\bfm}-y_{\bfk;\tbfm}\big| (\bfk;\bfm;\tbfm): \text{ special triple}\big\rangle_{\Q[\mu]}.
\end{equation*}

Recall that $\Q[\mu]\langle Y\rangle_{\rm pre}^0$ (including 1) is
the $\Q[\mu]$-submodule of $\Q[\mu]\langle Y\rangle$ generated by admissible words.
We then have the following direct sum decomposition
\begin{equation*}
\Q[\mu]\langle Y\rangle^0=\Q[\mu]\langle Y\rangle_{\rm pre}^0\bigoplus\Q[\mu]\langle Y\rangle^{0,\rm quasi}.
\end{equation*}

Define the $\Q[\mu]$-linear space generated by all quasi-admissible words with length $r$ by
\begin{equation*}
\Q[\mu]\langle Y\rangle^{0,\rm quasi}_r:=\big\langle y_{\bfk;\bfm}-y_{\bfk;\tbfm}\big| (\bfk;\bfm;\tbfm): \text{ special triple}, \, \ell(\bfk)=r \big\rangle_{\Q[\mu]}.
\end{equation*}
From the direct sum decomposition of $\Q[\mu]\langle Y\rangle$, we have the following direct sum decomposition
\begin{equation*}
\Q[\mu]\langle Y\rangle^{0,\rm quasi}=\bigoplus_{r=1}^{\infty}\Q[\mu]\langle Y\rangle^{0,\rm quasi}_r.
\end{equation*}
Similarly, for any $r\in\N$ and any fixed positive bi-index $(\bfk_0;\bfm_0)$ with $\ell(\bfk_0)=r-1$, we define
\begin{equation*}
\Q[\mu]\langle Y\rangle^{0,\rm quasi}_{r,(\bfk_0;\bfm_0)}:=\big\langle y_{\bfk;\bfm}-y_{\bfk;\tbfm} \big| (\bfk;\bfm;\tbfm): \text{ special triple}, \, (\bfk';\bfm')=(\bfk_0;\bfm_0)  \big\rangle_{\Q[\mu]},
\end{equation*}
where, as before, $\bfv'$ denotes the index obtained from $\bfv$ by removing its last component.

\begin{prop}
For any $r\in\N$ we have
\begin{equation*}
\Q[\mu]\langle Y\rangle^{0,\rm quasi}_{r}=\bigoplus_{(\bfk;\bfm)}\Q[\mu]\langle Y\rangle^{0,\rm quasi}_{r,(\bfk;\bfm)}.
\end{equation*}
\end{prop}

\begin{proof}
Step 1: Fix $(\bfk_0;\bfm_0)$. We need to show that
\begin{equation*}
\Q[\mu]\langle Y\rangle^{0,\rm quasi}_{r,(\bfk_0;\bfm_0)}\bigcap\sum_{(\bfk;\bfm)\neq(\bfk_0;\bfm_0)}\Q[\mu]\langle Y\rangle^{0,\rm quasi}_{r,(\bfk;\bfm)}=\{0\}.
\end{equation*}
If $r=1$, then $(\bfk;\bfm)=\phi$, we have
\begin{equation*}
\sum_{(\bfk;\bfm)\neq(\bfk_0;\bfm_0)}\Q[\mu]\langle Y\rangle^{0,\rm quasi}_{r,(\bfk;\bfm)}=\{0\}.
\end{equation*}
If $r>1$, then for any
\begin{equation*}
w\in\Q[\mu]\langle Y\rangle^{0,\rm quasi}_{r,(\bfk_0;\bfm_0)}\bigcap\sum_{(\bfk;\bfm)\neq(\bfk_0;\bfm_0)}\Q[\mu]\langle Y\rangle^{0,\rm quasi}_{r,(\bfk;\bfm)},
\end{equation*}
we can write
\begin{equation*}
w=y_{\bfk_0;\bfm_0}\left(\sum_{i} a_i \, y_{1,m_i}\right).
\end{equation*}
We can also write
\begin{equation*}
w=\sum_{i\neq 0} a_i \, y_{\bfk_i;\bfm_i}y_{1,m_i}.
\end{equation*}
Noticing that $(\bfk_0;\bfm_0)\neq(\bfk_i;\bfm_i),\forall i\neq0$, we get $w=0$.

Step 2: We need to show that
\begin{equation*}
\Q[\mu]\langle Y\rangle^{0,\rm quasi}_{r}=\sum_{(\bfk;\bfm)}\Q[\mu]\langle Y\rangle^{0,\rm quasi}_{r,(\bfk;\bfm)}.
\end{equation*}
If $r=1$, then $(\bfk;\bfm)=\phi$, there's nothing need to prove. If $r>1$, for any element
\begin{equation*}
w\in \Q[\mu]\langle Y\rangle^{0,\rm quasi}_{r}
\end{equation*}
we can write $w$ as
\begin{equation*}
w=\sum_i \, a_i \, y_{k_{1}^{(i)},m_{1}^{(i)}}\cdots y_{k_{r-1}^{(i)},m_{r-1}^{(i)}}(y_{1,m_i}-y_{1,\tm_i}).
\end{equation*}
Re-grouping, we see that
\begin{equation*}
w=\sum_j y_{\bfk_j;\bfm_j}\left(\sum_i a_i^{(j)}(y_{1,m_i}-y_{1,\tm_i})\right).
\end{equation*}
For all $j$ we have
\begin{equation*}
y_{\bfk_j;\bfm_j}\left(\sum_i a_i^{(j)}(y_{1,m_i}-y_{1,\tm_i})\right)\in\Q[\mu]\langle Y\rangle^{0,\rm quasi}_{r,(\bfk_j;\bfm_j)}
\end{equation*}
and therefore the proposition follows immediately.
\end{proof}

For any $r\in\N$ and positive bi-index $(\bfk_0;\bfm_0)$, we set
\begin{equation*}
\Q[\mu]\langle Y\rangle_{r,(\bfk_0;\bfm_0)}:=\big\langle  y_{\bfk;\bfm} \big| (\bfk;\bfm): \text{ positive bi-index}, \, (\bfk';\bfm')=(\bfk_0;\bfm_0) \big\rangle_{\Q[\mu]}.
\end{equation*}

\begin{prop}\label{prop:exactSeq}
We have the following exact sequence
$$0\longrightarrow\Q[\mu]\langle Y\rangle^{0,\rm quasi}_{r,(\bfk;\bfm)}\xrightarrow{\ \psi\ }\Q[\mu]\langle Y\rangle_{r,(\bfk;\bfm)}\xrightarrow{\ \phi\ }\Q[\mu]\longrightarrow 0$$
where $\psi$ is the embedding map, and
\begin{align*}
\phi:\Q[\mu]\langle Y\rangle_{r,(\bfk;\bfm)}&\longrightarrow\Q[\mu]\\
\sum_{i=1}^n a_i \, y_{\bfk,1;\bfm,m_i} &\longmapsto\sum_{i=1}^n a_i.
\end{align*}
\end{prop}
\begin{proof}
It is obvious that $\phi\circ\psi=0$. Suppose that $w=\sum\limits_{i=1}^n a_i \, y_{\bfk,1;\bfm,m_i}\in \Q[\mu]\langle Y\rangle_{r,(\bfk;\bfm)}$ such that $\sum\limits_{i=1}^n a_i=0$. Then we have
\begin{equation*}
w=\sum_{i=1}^n a_i \, y_{\bfk,1;\bfm,m_i}=\sum_{i=1}^n a_i \, \Big(y_{\bfk,1;\bfm,m_i}-y_{\bfk,1;\bfm,1}\Big)\in
\Q[\mu]\langle Y\rangle^{0,\rm quasi}_{r,(\bfk;\bfm)}.
\end{equation*}
Hence, the sequence in the proposition is exact.
\end{proof}

\begin{thm}\label{thm:zetamuAstWellDefined}
The following map is well-defined:
\begin{align*}
\zeta^{\mu}_{*}:\Q[\mu]\langle Y\rangle^{0} &\longrightarrow(\mathcal{Z}^{\mu},\cdot)\\
w &\longmapsto\lim\limits_{M\to+\infty}\zeta^{\mu}_{*,M}(w)
\end{align*}
where by \eqref{equ:zeta*M}, $\zeta^{\mu}_{*,M}(y_{\bfk;\bfm})=\zeta^{\mu}_{M}(\bfk;\bfm).$
\end{thm}

\begin{proof}
Clearly, $\zeta^{\mu}_{*}(\bfk;\bfm)$ is well-defined if $(\bfk;\bfm)$ is admissible.
We only need to prove the map
$$
\zeta^{\mu}_{*}:\Q[\mu]\langle Y\rangle^{0,\rm quasi}\longrightarrow(\mathcal{Z}^{\mu},\cdot)
$$
is well-defined. Indeed, by Proposition \ref{prop:exactSeq},
every element $w\in\Q[\mu]\langle Y\rangle^{0,\rm quasi}_{r,(\bfk;\bfm)}$
must have the following form
$$
w=\sum_{i=1}^n a_i \, y_{k_1,m_1}\cdots y_{k_{r-1},m_{r-1}}y_{1,m_{r,i}} \quad\text{with}\quad\sum_{i=1}^n a_i=0.
$$
By definition,
\begin{align*}
\zeta^{\mu}_{*}(w)=\lim\limits_{M\to+\infty}\zeta^{\mu}_{*,M}(w)
\end{align*}
where by \eqref{equ:zeta*M}
$$
\zeta^{\mu}_{*,M}(w)=\sum_{0\leqslant n_1\leqslant\cdots\leqslant n_r\leqslant M}\frac{1}{(n_1\mu+m_1)^{k_1}\cdots(n_{r-1}\mu+m_{r-1})^{k_{r-1}}}\left(\sum_{i=1}^n\frac{a_i}{n_r\mu+m_{r,i}}\right).
$$
Observe that
\begin{align*}
\sum_{i=1}^n\frac{a_i}{n_r\mu+m_{r,i}}
=&\, \sum_{i=1}^n a_i \left(\frac{1}{n_r\mu+m_{r,i}}-\frac{1}{n_r\mu+1}\right) \\
=&\, \sum_{i=1}^n a_i \left(\frac{1-m_{r,i}}{(n_r\mu+m_{r,i})(n_r\mu+1)}\right)
=O\left(\frac{1}{n_r^2}\right).
\end{align*}
We see that
$$\lim\limits_{M\to+\infty}\zeta^{\mu}_{*,M}(w)$$
is well-defined. Finally, the well-definedness of $\zeta^{\mu}_{*}$ on $\Q[\mu]\langle Y\rangle^{0,\rm quasi}$
follows from the direct sum decompositions
$$\Q[\mu]\langle Y\rangle^{0,\rm quasi}=
\bigoplus_{r=1}^{\infty} \bigoplus_{(\bfk;\bfm): \ell(\bfk)=r-1}\Q[\mu]\langle Y\rangle^{0,\rm quasi}_{r,(\bfk;\bfm)}.$$
We have thus finished the proof of the theorem.
\end{proof}

\subsection{Word algebra for shuffle product}
Next, we will look at another word algebra reflecting the essential properties of the shuffle product on $\calZ^{\mu}$.

\begin{defn}
Let $X^{\mu}=\{x,y^{\mu}_m|m\in\Q\}$ be a countable alphabet.
Let $X^{\mu,*}$ be the set of words over $X^{\mu}$, including the empty word 1. Let $\Q[\mu]\langle X^{\mu}\rangle$
be the non-commutative polynomial $\Q[\mu]$-algebra over $X^{\mu,*}$. The shuffle product $\shuffle$
on $\Q[\mu]\langle X^{\mu}\rangle$ is defined recursively by
\begin{alignat*}{4}
1\shuffle w= &\, w\shuffle1=w\qquad &&\forall w\in X^{\mu,*},\\
w_1u\shuffle w_2v= &\, (w_1\shuffle w_2v)u+(w_1u\shuffle w_2)v\qquad && \forall u,v\in X^{\mu}, w_1,w_2\in X^{\mu,*}.
\end{alignat*}
We then extend $\shuffle$ $\Q[\mu]$-linearly to $\Q[\mu]\langle X^{\mu}\rangle$. For each word $w$ in $X^{\mu,*}$
we denote by $\ell(w)$ the number of $y^{\mu}_i$'s appearing in it and call it the \emph{length} of $w$.
\end{defn}

\begin{prop}
$(\Q[\mu]\langle X^{\mu}\rangle,\shuffle)$ is a commutative and associative $\Q[\mu]$-algebra with unit.
\end{prop}

\begin{proof} One can check this result directly from the definition of the product $\shuffle$.
\end{proof}

\begin{defn}\label{defn:Xadmissible}
A word $w=y^{\mu}_{l_1}x^{t_1}y^{\mu}_{l_2}x^{t_2}\cdots y^{\mu}_{l_r}x^{t_r}$ is said to be \emph{positive} if $l_1>0,l_1+l_2>0,\cdots,l_1+\cdots+l_r>0$. It is said to be \emph{admissible} if it is positive
and, in addition, $t_r>1$. By convention, the word 1 will also be considered to be admissible.
We will denote by $\Q[\mu]\langle X^{\mu}\rangle_{\rm pre}^0$ the subspace of $\Q[\mu]\langle X^{\mu}\rangle$
generated by admissible words, and $\Q[\mu]\langle X^{\mu}\rangle^1$
the $\Q[\mu]$-submodule of $\Q[\mu]\langle X^{\mu}\rangle$ generated by positive words, i.e.
\begin{align*}
\Q[\mu]\langle X^{\mu}\rangle_{\rm pre}^0&:=\langle y_{l_1}x^{t_1}y_{l_2}x^{t_2}\cdots y_{l_r}x^{t_r}|l_1>0,l_1+l_2>0,\cdots,l_1+\cdots+l_r>0,t_r\geqslant1\rangle_{\Q[\mu]}\\
\Q[\mu]\langle X^{\mu}\rangle^1&:=\langle y_{l_1}x^{t_1}y_{l_2}x^{t_2}\cdots y_{l_r}x^{t_r}|l_1>0,l_1+l_2>0,\cdots,l_1+\cdots+l_r>0\rangle_{\Q[\mu]}
\end{align*}
For any positive bi-index $(\bfk;\bfm)=(k_1,\cdots,k_r;m_1,\cdots,m_r)$ we set
\begin{equation}\label{equ:yxMixed}
(y,x)^{\mu}_{\bfk;\bfm}:=y^{\mu}_{m_1}x^{k_1-1}y^{\mu}_{m_2-m_1}x^{k_2-1}\cdots y^{\mu}_{m_r-m_{r-1}}x^{k_r-1}.
\end{equation}

\end{defn}

\begin{prop}
We have the following results:

$(1).$ The $\Q[\mu]$-submodule $(\Q[\mu]\langle X^{\mu}\rangle_{\rm pre}^0,\shuffle)$ and $(\Q[\mu]\langle X^{\mu}\rangle^1,\shuffle)$ are subalgebra of $(\Q[\mu]\langle X^{\mu}\rangle,\shuffle)$, i.e., we have
$$(\Q[\mu]\langle X^{\mu}\rangle_{\rm pre}^0,\shuffle)\subset(\Q[\mu]\langle X^{\mu}\rangle^1,\shuffle)\subset(\Q[\mu]\langle X^{\mu}\rangle,\shuffle).$$

$(2).$ We have a homomorphism of $\Q[\mu]$-algebra
\begin{align*}
\zeta^{\mu}_{\shuffle}:(\Q[\mu]\langle X^{\mu}\rangle_{\rm pre}^0,\shuffle) &\, \longrightarrow(\calZ^{\mu},\cdot), \\
(y,x)^{\mu}_{\bfk;\bfm} &\, \longmapsto\zeta^{\mu}(\bfk;\bfm) .
\end{align*}
Namely, we have
$$\zeta^{\mu}_{\shuffle}(w\shuffle v)=\zeta^{\mu}_{\shuffle}(w)\cdot\zeta^{\mu}_{\shuffle}(v)\qquad\forall w,v\in\Q[\mu]\langle X^{\mu}\rangle_{\rm pre}^0.$$
\end{prop}

\begin{proof} The first statement can be checked directly from the definition of the
product $\shuffle$. The second statement follows from Corollary \ref{utgv}.
\end{proof}

\textbf{Unfortunately, the mapping $\zeta^{\mu}_{\shuffle}:(\Q[\mu]\langle X^{\mu}\rangle_{\rm pre}^0,\shuffle)\longrightarrow(\calZ^{\mu},\cdot)$ is not a surjection}. To remedy this situation, we introduce another notation.

\begin{defn}\label{defn:Xquasi-admissible}
Let $(\bfk;\bfm;\tbfm)$ be a special triple where
$\bfk=(k_1,\cdots,k_{r-1},1)$, $\bfm=(m_1,\cdots,m_{r-1},m_r)$, and $\tbfm=(\bfm',\tm_r).$
Then the element
$$(y,x)^{\mu}_{\bfk;\bfm}-(y,x)^{\mu}_{\bfk;\tbfm}:=
 y^{\mu}_{m_1}x^{k_1-1}\cdots y^{\mu}_{m_{r-1}-m_{r-2}}x^{k_{r-1}-1}(y^{\mu}_{m_{r}-m_{r-1}}-y^{\mu}_{\tm_r-m_{r-1}})$$
is called a \emph{quasi-admissible} word in $\Q[\mu]\langle X^{\mu}\rangle$. We denote by $\Q[\mu]\langle X^{\mu}\rangle^0$
the submodule generated by all admissible words and all quasi-admissible words over $\Q[\mu]$.
\end{defn}

\begin{prop}\label{prop:X0subalgebra}
Notation is as before.

$(1).$ The submodule $(\Q[\mu]\langle X^{\mu}\rangle^0 ,\shuffle)$ is a subalgebra of $(\Q[\mu]\langle X^{\mu}\rangle^1,\shuffle)$ and
$$
(\Q[\mu]\langle X^{\mu}\rangle_{\rm pre}^0,\shuffle)\subset(\Q[\mu]\langle X^{\mu}\rangle^0 ,\shuffle)\subset(\Q[\mu]\langle X^{\mu}\rangle^1,\shuffle).
$$

$(2).$ We have a homomorphism of $\Q[\mu]$-algebras
\begin{alignat*}{4}
\zeta^{\mu}_{\shuffle}:(\Q[\mu]\langle X^{\mu}\rangle^0 ,\shuffle)&\, \longrightarrow(\calZ^{\mu},\cdot), && \\
(y,x)^{\mu}_{\bfk;\bfm} &\, \longmapsto\zeta^{\mu}(\bfk;\bfm)\quad && \forall \text{admissible }\bfk, \\
(y,x)^{\mu}_{\bfk;\bfm}-(y,x)^{\mu}_{\bfk;\tbfm} &\, \longmapsto \calD^{\mu}(\bfk;\bfm;\tbfm)
 \quad&& \forall \text{special triple }(\bfk;\bfm;\tbfm).
\end{alignat*}
Namely,
$$\zeta^{\mu}_{\shuffle}(w\shuffle v)=\zeta^{\mu}_{\shuffle}(w)\cdot\zeta^{\mu}_{\shuffle}(v)\qquad
\forall w,v\in\Q[\mu]\langle X^{\mu}\rangle^0. $$
\end{prop}

\begin{proof} The first statement can be checked directly from the definition of the
product $\shuffle$. The second statement follows from Corollary~\ref{utgv}.
\end{proof} The well-definedness
of $\zeta^{\mu}_{\shuffle}$ can be proved in the same ways as in the proof of Theorem~\ref{thm:zetamuAstWellDefined}.

In the previous sections, we have seen that there is a bijection between positive bi-indices and binary sequences
which should lead to a natural map between $\Q[\mu]\langle Y\rangle$ and $\Q[\mu]\langle X^{\mu}\rangle^1$.
Inspired by this fact, we define a $\Q[\mu]$-linear map between them.

\begin{defn}
There is a natural $\Q[\mu]$-linear bijection between $\Q[\mu]\langle Y\rangle$ and $\Q[\mu]\langle X^{\mu}\rangle^1$:
\begin{align}\label{equ:bijection-varphi}
\varphi_{\mu}:\Q[\mu]\langle Y\rangle&\longrightarrow\Q[\mu]\langle X^{\mu}\rangle^1,\\
y_{\bfk;\bfm}=y_{k_1,m_1}\cdots y_{k_r,m_r}&\longmapsto
(y,x)^\mu_{\bfk;\bfm}=y^{\mu}_{m_1}x^{k_1-1}y^{\mu}_{m_2-m_1}x^{k_2-1}\cdots y^{\mu}_{m_r-m_{r-1}}x^{k_r-1}.
\end{align}
\end{defn}
Notice that, this map does not transform the $\mu$-stuffle product on $\Q[\mu]\langle Y\rangle$ into the shuffle
product on $\Q[\mu]\langle X^{\mu}\rangle^1$. To see $\varphi_{\mu}$ is a bijection we only need to check that
it has an inverse given by
\begin{align*}
\varphi_{\mu}^{-1}:\Q[\mu]\langle X^{\mu}\rangle^1&\longrightarrow\Q[\mu]\langle Y\rangle,\\
y^{\mu}_{l_1}x^{t_1}y^{\mu}_{l_2}x^{t_2}\cdots y^{\mu}_{l_r}x^{t_r}&\longmapsto y_{l_1,t_1+1}y_{l_1+l_2,t_2+1}\cdots y_{l_1+\cdots+l_r,t_r+1}.
\end{align*}
Furthermore, if we restrict the map $\varphi_{\mu}$ to $\Q[\mu]\langle Y\rangle^0$, then it
induces a bijection between $\Q[\mu]\langle Y\rangle^0$ and $\Q[\mu]\langle X^{\mu}\rangle^0 $.

\begin{prop}
We have the following commutative diagram

\begin{center}
\begin{tikzpicture}[scale=0.9]
\node (A) at (-3.8,2) {$\Q[\mu]\langle Y\rangle_{\rm pre}^0$};
\node (B) at (0,2) {$(\Q[\mu]\langle Y\rangle^0,\must)$};
\node (C) at (3.6,2) {$(\Q[\mu]\langle Y\rangle,\must)$};
\node (D) at (-3.8,0) {$(\Q[\mu]\langle X^{\mu}\rangle_{\rm pre}^0,\shuffle)$};
\node (E) at (0,0) {$(\Q[\mu]\langle X^{\mu}\rangle^0 ,\shuffle)$};
\node (F) at (3.6,0) {$(\Q[\mu]\langle X^{\mu}\rangle^1,\shuffle)$};
\draw[->] (A) to node[pos=.5,left] {$\varphi_{\mu}$}(D);
\draw[->] (B) to node[pos=.5,left] {$\varphi_{\mu}$}(E);
\draw[->] (C) to node[pos=.5,left] {$\varphi_{\mu}$}(F);
\node at (-1.8,1.9) {$\hookrightarrow$};
\node at (1.8,1.9) {$\hookrightarrow$};
\node at (-1.8,0) {$\hookrightarrow$};
\node at (1.8,0) {$\hookrightarrow$};
\end{tikzpicture}
\end{center}
\end{prop}

\begin{proof}
 One can check this result by the definition of $\varphi_{\mu}$.
\end{proof}

\begin{thm}
Let $\zeta^{\mu}_{*}:(\Q[\mu]\langle Y\rangle^0,\must)\longrightarrow(\calZ^{\mu},\cdot)$ and $\zeta^{\mu}_{\shuffle}:(\Q[\mu]\langle X^{\mu}\rangle^0 ,\shuffle)\longrightarrow(\calZ^{\mu},\cdot)$ be as before. Then we have
$$
\zeta^{\mu}_{*}(w)=\zeta^{\mu}_{\shuffle}(\varphi_{\mu}(w))\qquad\forall w\in\Q[\mu]\langle Y\rangle^0.
$$
Namely, the following diagram commutes:
\begin{center}
\begin{tikzpicture}[scale=0.9]
\node (A) at (-2,2) {$(\Q[\mu]\langle Y\rangle^0,\must)$};
\node (B) at (2,2) {$(\Q[\mu]\langle X^{\mu}\rangle^0 ,\shuffle)$};
\node (C) at (0,0) {$(\calZ^{\mu},\cdot)[T]$};
\draw[->] (A) to node[pos=.2,below] {$\zeta_{*}^{\mu}$}(C);
\draw[->] (B) to node[pos=.2,below] {$\zeta_{\shuffle}^{\mu}$}(C);
\draw[->] (A) to node[pos=.5,above] {${\varphi_{\mu}}$} (B);
\end{tikzpicture}
\end{center}
\end{thm}

\begin{proof} This follows directly from the definition of $\zeta_{*}^{\mu}$ and $\zeta_{\shuffle}^{\mu}$.
\end{proof}

\section{Regularization homomorphisms and comparison theorems} \label{sec:RegHomo}
In this section, we discuss how to extend $\mu$-MHZVs onto non-admissible words and use this regularization to derive
relations among them. We will then prove the comparison theorem among the two different products, which is an
analogy of the Ihara-Kaneko-Zagier comparison theorem. As applications, we will find some interesting identities
among $\mu$-MHZVs.

\subsection{Regularization homomorphisms for $\mu$-stuffle product}
We now consider the regularization of $\zeta^{\mu}_{*}$ on non-admissible words. First we observe that
there is a bijection between the set of positive bi-indices and the set of words over $Y$:
$$(\bfk;\bfm) \leftrightarrow y_{\bfk;\bfm}.$$
We will use this bijection to identify both sets. The next lemma will be useful when we estimate
partial sums of $\mu$-MHZVs.

\begin{lem}\label{epvh}
Put $n_0=0$. Let
\begin{equation*}
U_{r,i}(M):=\left\{(n_1,\cdots,n_r)\in\Z^r\big|0\leqslant n_1\leqslant\cdots\leqslant n_i\leqslant M\leqslant n_{i+1}\leqslant\cdots\leqslant n_r\right\}.
\end{equation*}
Let $(\bfk;\bfm)=(k_1,\cdots,k_r;m_1,\cdots,m_r)$ be a positive bi-index. Then
\begin{equation}\label{equ:Urr}
\sum_{U_{r,r}(M)}\frac{\mu^r}{(n_1\mu+m_1)^{k_1}\cdots(n_r\mu+m_r)^{k_r}}=O(\ln^r(M))\quad\text{as}\quad M\to+\infty.
\end{equation}
If $(\bfk;\bfm)=(k_1,\cdots,k_r;m_1,\cdots,m_r)$ is admissible and $0\leqslant i\leqslant r-1$, then
\begin{equation}\label{equ:Uri}
\sum_{U_{r,i}(M)}\frac{\mu^r}{(n_1\mu+m_1)^{k_1}\cdots(n_r\mu+m_r)^{k_r}}=
 O\left(\frac{\ln^i(M)}{M}\right)\quad\text{as}\quad M\to+\infty.
\end{equation}
In both \eqref{equ:Urr} and \eqref{equ:Urr} the big-$O$ constants only depend on $(\bfk;\bfm)$ and $\mu$.
\end{lem}
\begin{proof}
Without loss of generality we may assume $\mu>0$ and by scaling we can even assume $m_j$'s are positive integers.
We prove \eqref{equ:Urr} by induction on $r$. When $r=1$ this is obvious. For $r>1$, we have
\begin{align*}
&\sum_{U_{r,r}(M)}\frac{\mu^r}{(n_1\mu+m_1)^{k_1}\cdots(n_r\mu+m_r)^{k_r}}\\
&\leqslant\sum_{U_{r,r}(M)}\frac{\mu^r}{(n_1\mu+1)\cdots(n_r\mu+1)} \\
&\leqslant \sum_{0\leqslant n_2\leqslant \cdots \leqslant n_r\leqslant M}\frac{\mu^r}{(n_2\mu+1)\cdots(n_r\mu+1)}
+\sum_{1\leqslant n_1\leqslant \cdots\leqslant n_r\leqslant M} \frac{1}{n_1\cdots n_r} \\
&\leqslant O(\ln^{r-1}(M)) +\sum_{n_1,\cdots,n_r=1}^{\lfloor M\rfloor}\frac{1}{n_1\cdots n_r}\quad(\text{by induction}) \\
&=O(\ln^r(M)).
\end{align*}

To prove \eqref{equ:Uri} we notice that
\begin{align*}
 &\sum_{U_{r,i}(M)}\frac{\mu^r}{(n_1\mu+m_1)^{k_1}\cdots(n_r\mu+m_r)^{k_r}}\\
 &\leqslant\sum_{U_{i,i}(M)}\frac{\mu^i}{(n_1\mu+1)\cdots(n_i\mu+1)}
\sum_{M\leqslant n_{i+1} \leqslant \cdots\leqslant n_r} \frac{\mu^{r-i}}{(n_{i+1}\mu+1)\cdots(n_{r-1}\mu+1)(n_r\mu+1)^2}.
\end{align*}
By \eqref{equ:Urr} we only need to show that for all $r\in\N$, as $M\to \infty$
\begin{align}\label{equ:1..12Induction}
 \sum_{M\leqslant n_1 \leqslant \cdots \leqslant n_r} \frac{1}{n_1\cdots n_{r-1} n_r^2}=O(1/M).
\end{align}
This is clear if $r=1$. For general $r>1$, we have, for all $M>2$
\begin{align*}
\sum_{n_{r-1}\leqslant n_r} \frac{1}{n_r^2} <\sum_{n_{r-1}\leqslant n_r} \frac{2}{n_r(n_r+1)}=\frac{2}{n_{r-1}}
\end{align*}
by telescoping series. Thus \eqref{equ:1..12Induction} follows immediately by induction.
\end{proof}

\begin{lem}
Let $M$ be a positive real number. Then

$(1).$ If $(\bfk;\bfm)=(k_1,\cdots,k_r;m_1,\cdots,m_r)$ is an admissible bi-index (i.e., $k_r>1$), then
$$\zeta_M^{\mu}(\bfk;\bfm)=\zeta^{\mu}(\bfk;\bfm)+O\left(\frac{\ln^{r-1}(M)}{M}\right)\qquad M\to+\infty.$$

$(2).$ If $(\bfk;\bfm;\tbfm)$ is a special triple then
\begin{equation*}
\calD^{\mu}_M(\bfk;\bfm;\tbfm)=
\calD^{\mu}(\bfk;\bfm;\tbfm)+O\left(\frac{\ln^{r-1}(M)}{M}\right).
\end{equation*}

In (1) (resp.\ (2)) the big-$O$ constant only depends on $(\bfk;\bfm)$ (resp.\ $(\bfk;\bfm;\tbfm)$) and $\mu$.
\end{lem}

\begin{proof}
Without loss of generality we may assume $\mu>0$ and $m_j$'s are positive integers.

(1). If $(\bfk;\bfm)=(k_1,\cdots,k_r;m_1,\cdots,m_r)$ is an admissible bi-index, then
\begin{align*}
|\zeta_M^{\mu}(\bfk;\bfm)-\zeta^{\mu}(\bfk;\bfm)|&\leqslant\sum_{i=0}^{r-1}\sum_{U_{r,i}(M)}\frac{\mu^r}{(n_1\mu+m_1)^{k_1}\cdots(n_r\mu+m_r)^{k_r}}\\
&\leqslant\sum_{i=0}^{r-1}\sum_{U_{r,i}(M)}\frac{\mu^r}{(n_1\mu+1)\cdots(n_{r-1}\mu+1)(n_r\mu+1)^2}\\
&\overset{\eqref{equ:Uri}}{=}\sum_{i=0}^{r-1}O\left(\frac{\ln^i(M)}{M}\right)
=O\left(\frac{\ln^{r-1}(M)}{M}\right).
\end{align*}

(2). Notice that
\begin{align*}
&\calD^{\mu}_M(\bfk;\bfm;\tbfm)=\zeta^{\mu}_M(\bfk;\bfm)-\zeta^{\mu}_M(\bfk;\tbfm)\\
&=(\tm_r-m_r)\sum_{0\leqslant n_1\leqslant\cdots\leqslant n_r\leqslant M}\frac{\mu^r}{(n_1\mu+m_1)^{k_1}\cdots(n_{r-1}\mu+m_{r-1})^{k_{r-1}}(n_r\mu+m_r)(n_r\mu+\tm_r)}.
\end{align*}
Hence,
\begin{align*}
&|\calD^{\mu}_M(\bfk;\bfm;\tbfm)-\calD^{\mu}(\bfk;\bfm;\tbfm)|\\
&\leqslant\sum_{i=0}^{r-1}|\tm_r-m_r|\sum_{U_{r,i}(M)}\frac{\mu^r}{(n_1\mu+m_1)^{k_1}\cdots(n_{r-1}\mu+m_{r-1})^{k_{r-1}}(n_r\mu+m_r)(n_r\mu+\tm_r)}\\
&\leqslant\sum_{i=0}^{r-1}|\tm_r-m_r|\sum_{U_{r,i}(M)}\frac{\mu^r}{(n_1\mu+1)\cdots(n_{r-1}\mu+1)(n_r\mu+1)^2}\\
&\overset{\eqref{equ:Uri}}{=}\sum_{i=0}^{r-1}|\tm_r-m_r|O\left(\frac{\ln^i(M)}{M}\right)
=O\left(\frac{\ln^{r-1}(M)}{M}\right).
\end{align*}

\end{proof}

Next, we will prove a crucial theorem which allows us to consider non-admissible words and then non-admissible $\mu$-MHZVs.

\begin{thm}\label{thm:stRegPsi}
The map
\begin{align*}
\psi_{*,m}:(\Q[\mu]\langle Y\rangle^0,\must)[T]&\longrightarrow(\Q[\mu]\langle Y\rangle,\must),\\
T&\longmapsto y_{1,m}
\end{align*}
is an isomorphism of $(\Q[\mu]\langle Y\rangle^0,\must)$-algebra.
\end{thm}

\begin{proof} We will prove this result in two steps.

\noindent
Step 1. We first show that the map $\psi_{*,m}$ is surjective, which amounts to saying that any element $w\in\Q[\mu]\langle Y\rangle$ can be written as a polynomial in $y_{1,m}$ with coefficients in $\Q[\mu]\langle Y\rangle^0$.
To this end, we observe that there exists a filtration on $\Q[\mu]\langle Y\rangle$ induced by the length $\ell$ as follows
$$0\subset F_{0}\Q[\mu]\langle Y\rangle\subset F_{1}\Q[\mu]\langle Y\rangle\subset F_{2}\Q[\mu]\langle Y\rangle\subset\cdots$$
where
$$F_{\ell}\Q[\mu]\langle Y\rangle:=\Big\langle y_{\bfk;\bfm}| (\bfk;\bfm): \text{ positive bi-index},\, \ell(\bfk)\leqslant \ell\Big\rangle_{\Q[\mu]}.$$
To show the surjectivity by induction on $\ell$, we only need to show that, for a fixed length $\ell$ and a word
$w\in F_{\ell}\Q[\mu]\langle Y\rangle$, there are elements $v_1\in F_{\ell}\Q[\mu]\langle Y\rangle^0$ and
$v_2,v_3\in F_{\ell-1}\Q[\mu]\langle Y\rangle$ such that
$$w=v_1+v_2\must y_{1,m}+v_3.$$

Let $(\bfk;\bfm)=(k_1,\cdots,k_r,1,\cdots,1;m_1,\cdots,m_r,m_{r+1},\cdots,m_{r+s})$, $k_r>1$ and
$$y_{\bfk;\bfm}=y_{k_1,m_1}\cdots y_{k_r,m_r}y_{1,m_{r+1}}\cdots y_{1,m_{r+s}}.$$
Then,
\begin{align*}
y_{\bfk;\bfm}&=y_{k_1,m_1}\cdots y_{k_r,m_r}y_{1,m_{r+1}}\cdots y_{1,m_{r+s}}\\
&=y_{k_1,m_1}\cdots y_{k_r,m_r}y_{1,m}^s+y_{k_1,m_1}\cdots y_{k_r,m_r}(y_{1,m_{r+1}}-y_{1,m})y_{1,m}^{s-1}\\
&+y_{k_1,m_1}\cdots y_{k_r,m_r}y_{1,m_{r+1}}(y_{1,m_{r+2}}-y_{1,m})y_{1,m}^{s-2}+\cdots\\
&+y_{k_1,m_1}\cdots y_{k_r,m_r}y_{1,m_{r+1}}\cdots y_{1,m_{r+s-1}}(y_{1,m_{r+s}}-y_{1,m}).
\end{align*}
We now apply induction of the power of $y_{1,m}$ to
treat all the terms on the right-hand side by breaking them into two cases.

(i). For the word $y_{k_1,m_1}\cdots y_{k_r,m_r}y_{1,m}^s$, by recursive calculation, we have
\begin{align*}
&y_{k_1,m_1}\cdots y_{k_r,m_r}y_{1,m}^{s-1}\must y_{1,m}\\
&=s\cdot y_{k_1,m_1}\cdots y_{k_r,m_r}y_{1,m}^s+y_{1,m}y_{k_1,m_1}\cdots y_{k_r,m_r}y_{1,m}^{s-1}\\
&\quad+\sum_{i=1}^{r-1}y_{k_1,m_1}\cdots y_{k_i,m_i}y_{1,m}y_{k_{i+1},m_{i+1}}\cdots y_{k_r,m_r}y_{1,m}^{s-1}+v_3,
\end{align*}
with $v_3\in F_{r+s-1}\Q[\mu]\langle Y\rangle$. Applying the induction hypothesis with respect to $s$ we obtain the desired conclusion.

(ii). For the word $y_{k_1,m_1}\cdots y_{k_r,m_r}y_{1,m_{r+1}}\cdots y_{1,m_{r+i-1}}(y_{1,m_{r+i}}-y_{1,m})y_{1,m}^{s-i}, i=1,\cdots s$, by recursive calculation, we also have
\begin{align*}
&y_{k_1,m_1}\cdots y_{k_r,m_r}y_{1,m_{r+1}}\cdots y_{1,m_{r+i-1}}(y_{1,m_{r+i}}-y_{1,m})y_{1,m}^{s-i-1}\must y_{1,m}\\
&=(s-i)\cdot y_{k_1,m_1}\cdots y_{k_r,m_r}y_{1,m_{r+1}}\cdots y_{1,m_{r+i-1}}(y_{1,m_{r+i}}-y_{1,m})y_{1,m}^{s-i}\\
&\quad+y_{1,m}y_{k_1,m_1}\cdots y_{k_r,m_r}y_{1,m_{r+1}}\cdots y_{1,m_{r+i-1}}(y_{1,m_{r+i}}-y_{1,m})y_{1,m}^{s-i-1}\\
&\quad+\cdots\\
&\quad+y_{k_1,m_1}\cdots y_{k_r,m_r}y_{1,m_{r+1}}\cdots y_{1,m_{r+i-1}}y_{1,m}(y_{1,m_{r+i}}-y_{1,m})y_{1,m}^{s-i-1}+v_3
\end{align*}
with $v_3\in F_{r+s-1}\Q[\mu]\langle Y\rangle$. Applying the induction hypothesis with respect to $s-i$ we obtain the desired conclusion.

Moreover, by keeping track of the lengths of all the words involved we conclude that
$$y_{\bfk;\bfm}=w_s\must \overbrace{y_{1,m}\must \cdots\must y_{1,m}}^s+w_{s-1}\must \overbrace{y_{1,m}\must \cdots\must y_{1,m}}^{s-1}+\cdots+w_0$$
where $w_s,\cdots,w_0\in\Q[\mu]\langle Y\rangle^0$ such that $\ell(w_j)=r+s-j$ for all $r\leqslant j\leqslant r+s$. It follows that $\psi_{*,m}$ is surjective.

\noindent
Step 2. To prove the injectivity of $\psi_{*,m}$, we write each non-zero
$P\in(\Q[\mu]\langle Y\rangle^0,\must)[T]$ as
$$P=w_1T^s+w_2$$
where $0\neq w_1\in\Q[\mu]\langle Y\rangle^0$ and $w_2$ has $T$-degree less than $s$. Then
$$\psi_{*,m}(P)=s!w_1y_{1,m}^s+v_2$$
where all the words in $v_2$ have less than $s$ copies of $y_{1,m}$ at the end. Thus, $\psi_{*,m}(P)\neq0$ and $\psi_{*,m}$ is injective.

\end{proof}

Using the previous results, we can now prove one of the most important theorems in the paper.

\begin{thm-defn}\label{thm:stReg_zetaMuT}
Let $(\bfk;\bfm)=(k_1,\cdots,k_r,1_s;m_1,\cdots,m_{r+s})\in\N^{r+s}\times\Q^{r+s}$ with $k_r>1$.
For any fixed $m\in\Q$ we have
$$\zeta_M^{\mu}(\bfk;\bfm)=a_s(\zeta_M^{\mu}(1;m))^s+a_{s-1}(\zeta_M^{\mu}(1;m))^{s-1}
 +\cdots+a_0+O\left(\frac{\ln^{r+s-1}(M)}{M}\right),$$
where $a_0,\cdots,a_s\in\calZ^{\mu}$.  We define the \emph{$\mu$-stuffle regularization homomorphism} as follows
\begin{align*}
\zeta^{\mu,T}_{*,m}:(\Q[\mu]\langle Y\rangle,\must)&\longrightarrow(\calZ^{\mu},\cdot)[T],\\
y_{\bfk;\bfm}&\longmapsto\sum_{i=0}^sa_iT^i.
\end{align*}
\end{thm-defn}

\begin{proof} By Theorem~\ref{thm:stRegPsi},
$$y_{\bfk;\bfm}=w_s\must \overbrace{y_{1,m}\must \cdots\must y_{1,m}}^s+w_{s-1}\must
\overbrace{y_{1,m}\must \cdots\must y_{1,m}}^{s-1}+\cdots+w_0$$
where $w_s,\cdots,w_0\in\Q[\mu]\langle Y\rangle^0$ such that $\ell(w_j)=r+s-j$ for all $r\leqslant j\leqslant r+s$.
Hence,
\begin{align*}
\zeta^{\mu}_M(\bfk;\bfm)&=\zeta^{\mu}_{*,M}(y_{\bfk;\bfm})\\
&=\sum_{i=0}^s\zeta^{\mu}_{*,M}(w_i\must \overbrace{y_{1,m}\must \cdots\must y_{1,m}}^i)\\
&=\sum_{i=0}^s\zeta^{\mu}_{*,M}(w_i)\cdot(\zeta^{\mu}_{*,M}(y_{1,m}))^i\\
&\overset{\eqref{equ:Urr}}{=}\sum_{i=0}^s\left(a_i+O\left(\frac{\ln^{r+s-i-1}(M)}{M}\right)\right)\cdot(\zeta^{\mu}_{M}(1;m))^i \quad(\text{since } \ell(w_i)=r+s-i)\\
&=\sum_{i=0}^sa_i\cdot(\zeta^{\mu}_{M}(1;m))^i+O\left(\frac{\ln^{r+s-1}(M)}{M}\right)
\end{align*}
since $\zeta^{\mu}_{M}(1;m)=O(\ln(M))$ by \eqref{equ:Urr}, where $a_i=\zeta^{\mu}_{*}(w_i)\in\calZ^{\mu}$ for all $0\leqslant i\leqslant s$.
\end{proof}

\begin{eg}\label{nxbt}
Let $(\bfk;\bfm)=(1,1;1,2)$ be a positive bi-index. Then we have
\begin{align*}
y_{1,1}y_{1,2}&=y_{1,1}(y_{1,2}-y_{1,1})+y_{1,1}^2\\
&=\frac{y_{1,1}\must y_{1,1}+\mu y_{2,1}}{2}+y_{1,1}(y_{1,2}-y_{1,2}).
\end{align*}
Hence
\begin{align*}
\zeta_M^{\mu}(1,1;1,2)&=\zeta_{*,M}^{\mu}(y_{1,1}y_{1,2})\\
&=\frac{(\zeta^{\mu}_{*,M}(y_{1,1}))^2}{2}+\frac{\mu}{2}\zeta^{\mu}_{*,M}(y_{2,1})+\zeta^{\mu}_{*,M}(y_{1,1}(y_{1,2}-y_{1,1}))\\
&=\frac{(\zeta^{\mu}_{*,M}(y_{1,1}))^2}{2}+\frac{\mu}{2}\zeta^{\mu}_{*}(y_{2,1})+\zeta^{\mu}_{*}(y_{1,1}(y_{1,2}-y_{1,1}))+O\left(\frac{\ln^2(M)}{M}\right)\\
&=\frac{(\zeta^{\mu}_{*,M}(y_{1,1}))^2}{2}+\frac{\mu}{2}\zeta^{\mu}(2;1)+\calD^{\mu}(1,1;1,2;1,1)+O\left(\frac{\ln^2(M)}{M}\right).
\end{align*}
\end{eg}

\begin{thm}\label{rtxz}
The map $\zeta^{\mu,T}_{*,m}:(\Q[\mu]\langle Y\rangle,\must)\longrightarrow(\calZ^{\mu},\cdot)[T]$ is a $\Q[\mu]$-algebra homomorphism.
\end{thm}

\begin{proof} By the definition of $\zeta^{\mu}_{*,M}$, we know that
$$\zeta^{\mu}_{*,M}(w\must v)=\zeta^{\mu}_{*,M}(w)\cdot\zeta^{\mu}_{*,M}(v)\qquad\forall w,v\in\Q[\mu]\langle Y\rangle.$$
By the definition of $\zeta_{*,m}^{\mu,T}$, we have
$$\zeta^{\mu}_{*,M}(w)=\zeta_{*,m}^{\mu,T}(w)+O\left(\frac{\ln^r(M)}{M}\right).$$
Since $\zeta^{\mu}_{M}(1;m)=O(\ln(M))$ by \eqref{equ:Urr} we get
$$\zeta^{\mu,T}_{*,m}(w\must v)=\zeta^{\mu,T}_{*,m}(w)\cdot\zeta^{\mu,T}_{*,m}(v)\qquad\forall w,v\in\Q[\mu]\langle Y\rangle,$$
as desired.
\end{proof}

Furthermore, the map $\zeta^{\mu,T}_{*,m}$ is characterized by the following conditions.

\begin{thm}
For any given $m\in\Q_{>0}$, there exists a unique $\Q[\mu]$-linear map
$$\zeta_{*,m}^{\mu,T}:(\Q[\mu]\langle Y\rangle,\must)\longrightarrow(\calZ^{\mu},\cdot)[T]$$
that satisfies
$$\begin{cases}
\zeta_{*,m}^{\mu,T}(w)=\zeta_{*}^{\mu}(w) \qquad&\forall w\in\Q[\mu]\langle Y\rangle^0;\\
\zeta_{*,m}^{\mu,T}(y_{1,m})=T;\\
\zeta_{*,m}^{\mu,T}(v\must w)=\zeta_{*,m}^{\mu,T}(v)\cdot\zeta_{*,m}^{\mu,T}(w) &\forall v,w\in\Q[\mu]\langle Y\rangle.
\end{cases}$$
\end{thm}

\begin{proof}
The existence of $\zeta_{*,m}^{\mu,T}$ is guaranteed by Theorem~\ref{rtxz}.
We only need to prove the uniqueness. By Theorem~\ref{thm:stRegPsi}, we have
$$w=w_s\must \overbrace{y_{1,m}\must \cdots\must y_{1,m}}^s+w_{s-1}\must \overbrace{y_{1,m}\must \cdots\must y_{1,m}}^{s-1}+\cdots+w_0$$
where $w_0,\cdots,w_s\in\Q[\mu]\langle Y\rangle^0$. If $f$ is another map satisfying the above properties, then we have
$$
f(w)=\sum_{i=0}^sf(w_i)\cdot(f(y_{1,m}))^i=\sum_{i=0}^sf(w_i)\cdot T^i.
$$
But $w_i$ is the $\Q[\mu]$-linear combination of (quasi-)admissible words and therefore
$$f(w)=\zeta_{*,m}^{\mu,T}(w).$$
This proves the uniqueness $\zeta_{*,m}^{\mu,T}$.
\end{proof}

\begin{eg}
From Example~\ref{nxbt}, we have
\begin{align*}
\zeta^{\mu,T}_{*,1}(y_{1,1}y_{1,2})&=\frac{(\zeta^{\mu,T}_{*,1}(y_{1,1}))^2}{2}
+\frac{\mu}{2}\zeta^{\mu}_{*}(y_{2,1})+\zeta^{\mu}_{*}(y_{1,1}(y_{1,2}-y_{1,1}))\\
&=\frac{T^2}{2}+\frac{\mu}{2}\zeta^{\mu}(2;1)+\calD^{\mu}(1,1;1,2;1,1).
\end{align*}
\end{eg}

\subsection{Regularization homomorphisms for shuffle product}
In this section, we will consider the regularization of $\zeta_{\shuffle}^{\mu}$ on non-admissible words.
Recall that each bi-index $(\bfk;\bfm)$ defines a word $(y,x)^{\mu}_{\bfk;\bfm}.$
In fact, there is a bijection between the set of positive bi-indices and the set of words in $\Q[\mu]\langle X^{\mu}\rangle^1$.

For $0<t<1$, we defines the $\Q[\mu]$-linear map
\begin{align*}
\zeta^{\mu}_{\shuffle,t}:(\Q[\mu]\langle X^{\mu}\rangle^1,\shuffle)&\longrightarrow C^{\infty}(0,1),\\
(y,x)^{\mu}_{\bfk;\bfm} &\longmapsto\int_{\Delta^{|\bfk|}(t)}\omega^{\mu}_{(\bfk;\bfm)}.
\end{align*}

\begin{lem}\label{kdov}
For $0<t<1$, we have

$(1)$. If $(\bfk;\bfm)=(k_1,\cdots,k_r;m_1,\cdots,m_r)$ is a positive bi-index, then
$$\int_{\Delta^{|\bfk|}(t)}\omega^{\mu}_{(\bfk;\bfm)}=O\Big(\ln^r\Big(\frac1{1-t}\Big)\Big)\quad\text{as}\quad t\to 1^-.$$

$(2)$. If $(\bfk;\bfm)=(k_1,\cdots,k_r;m_1,\cdots,m_r)$ is an admissible bi-index, then
$$\int_{\Delta^{|\bfk|}(t)}\omega^{\mu}_{(\bfk;\bfm)}
=\int_{\Delta^{|\bfk|}}\omega^{\mu}_{(\bfk;\bfm)}
 +O\Big((1-t)\ln^{r}\Big(\frac1{1-t}\Big)\Big)\quad\text{as}\quad t\to 1^-.
$$
\end{lem}

\begin{proof} (1). For $0<t<1$ we have
$$\frac{\mu}{1-t^{\mu}}\leqslant
\left\{
 \begin{array}{ll}
 \displaystyle \frac{\mu}{1-t}, \quad& \hbox{if $\mu\geqslant1$;} \\
 \displaystyle \frac{1}{1-t}, \quad& \hbox{if $0<\mu\leqslant1$.}
 \end{array}
\right.
 \Longrightarrow \quad \frac{\mu}{1-t^{\mu}}\leqslant\frac{c}{1-t}\quad \forall t\in(0,1),
$$
where $c:=\max\{1,\mu\}$. By Theorem~\ref{pzci}, we have
\begin{align*}
\int_{\Delta^{|\bfk|}(t)}\omega^{\mu}_{(\bfk;\bfm)}=t^{m_r}\Li^{\mu}(\bfk;\bfm;t)
&=\sum_{0\leqslant n_1\leqslant\cdots\leqslant n_r}\frac{\mu^rt^{n_r\mu+m_r}}{(n_1\mu+m_1)^{k_1}\cdots(n_r\mu+m_r)^{k_r}}.
\end{align*}
It is clear that for $t$ sufficiently close to 1 we always have $t^{m_r}<2 t$. Further,
by scaling we now may assume all $m_j$'s are positive integers. Hence,
\begin{align*}
\int_{\Delta^{|\bfk|}(t)}\omega^{\mu}_{(\bfk;\bfm)}&\leqslant\sum_{0\leqslant n_1\leqslant\cdots\leqslant n_r}\frac{2\mu^rt^{n_r\mu+1}}{(n_1\mu+1)\cdots(n_r\mu+1)}\\
&=2\int\limits_{0<t_r<\cdots<t_1<t}\frac{\mu dt_r}{1-t_r^{\mu}}\frac{\mu dt_{r-1}}{t_{r-1}(1-t_{r-1}^{\mu})}\frac{\mu dt_{r-2}}{t_{r-2}(1-t_{r-2}^{\mu})}\cdots\frac{\mu dt_{1}}{t_{1}(1-t_1^{\mu})}\\
&\leqslant 2c^r\int\limits_{0<t_r<\cdots<t_1<t}\frac{dt_r}{1-t_r}\frac{dt_{r-1}}{t_{r-1}(1-t_{r-1})}\frac{dt_{r-2}}{t_{r-2}(1-t_{r-2})}\cdots\frac{dt_{1}}{t_{1}(1-t_1)}\\
&=2c^r\int\limits_{0<t_r<\cdots<t_1<t}\frac{dt_r}{1-t_r}\left(\frac{dt_{r-1}}{t_{r-1}}+\frac{dt_{r-1}}{1-t_{r-1}}\right)\left(\frac{dt_{r-2}}{t_{r-2}}+\frac{dt_{r-2}}{1-t_{r-1}}\right)\cdots\left(\frac{dt_{1}}{t_{1}}+\frac{dt_{1}}{1-t_{1}}\right)\\
&=2c^r \Li_{\underbrace{\scriptstyle 1\circ \cdots \circ 1}_{\text{$r$ times}}}(t) \qquad\text{(where $\circ=$``+'' or ``,'', see \eqref{equ:MPLrel})} \\
&=O\Big(\ln^{r}\Big(\frac1{1-t}\Big)\Big).
\end{align*}
by \cite[Lemma 3.3.20]{ZhaoBook}.

(2). For $0<t<1$, we have the inequalities
$$0<1-t^{n_r\mu+m_r}<1,\quad 0<1-t^{n_r\mu+m_r}<(n_r\mu+m_r)(1-t).$$
Therefore,
\begin{align*}
0&<\int_{\Delta^{|\bfk|}}\omega^{\mu}_{(\bfk;\bfm)}-\int_{\Delta^{|\bfk|}(t)}\omega^{\mu}_{(\bfk;\bfm)}\\
&=\mu^r\sum_{0\leqslant n_1\leqslant\cdots\leqslant n_r}\frac{1-t^{n_r\mu+m_r}}{(n_1\mu+m_1)^{k_1}\cdots(n_r\mu+m_r)^{k_r}}\\
&\leqslant\mu^r\sum_{i=0}^r\sum_{U_{r,i}(\frac{1}{1-t})}\frac{1-t^{n_r\mu+m_r}}{(n_1\mu+m_1)^{k_1}\cdots(n_r\mu+m_r)^{k_r}}\\
&\leqslant\mu^r\sum_{U_{r,r}(\frac{1}{1-t})}\frac{(n_r\mu+m_r)(1-t)}{(n_1\mu+m_1)^{k_1}\cdots(n_r\mu+m_r)^{k_r}}
+\mu^r\sum_{i=0}^{r-1}\sum_{U_{r,i}(\frac{1}{1-t})}\frac{1-t^{n_r\mu+m_r}}{(n_1\mu+m_1)^{k_1}\cdots(n_r\mu+m_r)^{k_r}}\\
&\leqslant\mu^r\sum_{U_{r,r}\left(\frac{1}{1-t}\right)}\frac{1-t}{(n_1\mu+1)\cdots(n_r\mu+1)}
+\mu^r\sum_{i=0}^{r-1}\sum_{U_{r,i}\left(\frac{1}{1-t}\right)}\frac{1}{(n_1\mu+1)\cdots(x_{r-1}+1)(n_r\mu+1)^2}\\
&=\mu^r(1-t)O\Big(\ln^{r}\Big(\frac1{1-t}\Big)\Big)+
 \mu^r\sum_{i=0}^{r-1}O\Big((1-t)\ln^i\Big(\frac1{1-t}\Big)\Big)\qquad\text{(by\ lemma\ \ref{epvh})}\\
&=O\Big((1-t)\ln^{r}\Big(\frac1{1-t}\Big)\Big),
\end{align*}
as desired.
\end{proof}

\begin{lem}
Let $(\bfk;\bfm;\tbfm)$ be a special triple. Then we have
\begin{align*}
\int_{\Delta^{|\bfk|}(t)}(\omega^{\mu}_{(\bfk;\bfm)}-\omega^{\mu}_{(\bfk;\tbfm)})
=\int_{\Delta^{|\bfk|}}(\omega^{\mu}_{(\bfk;\bfm)}-\omega^{\mu}_{(\bfk;\tbfm)})
+O\Big((1-t)\ln^{r}\Big(\frac1{1-t}\Big)\Big).
\end{align*}
\end{lem}

\begin{proof}
We have
\begin{align*}
&\int_{\Delta^{|\bfk|}(t)}(\omega^{\mu}_{(\bfk;\bfm)}-\omega^{\mu}_{(\bfk;\tbfm)})\\
&=t^{m_r}\Li^{\mu}(\bfk;\bfm;t)-t^{\tm_r}\Li^{\mu}(\bfk;\tbfm;t)\\
&=t^{m_r}\Li^{\mu}(\bfk;\bfm;t)-\Li^{\mu}(\bfk;\bfm;t)+\Li^{\mu}(\bfk;\bfm;t)-\Li^{\mu}(\bfk;\tbfm;t)\\
&\quad+\Li^{\mu}(\bfk;\tbfm;t)-t^{\tm_r}\Li^{\mu}(\bfk;\tbfm;t)\\
&=(1-t^{-m_r})\cdot t^{m_r}\Li^{\mu}(\bfk;\bfm;t)+(t^{-\tm_r}-1)\cdot t^{\tm_r}\Li^{\mu}(\bfk;\tbfm;t)\\
&   \quad+\Li^{\mu}(\bfk;\bfm;t)-\Li^{\mu}(\bfk;\tbfm;t)\\
&=(1-t^{-m_r})\int_{\Delta^{|\bfk|}(t)}\omega^{\mu}_{(\bfk;\bfm)}
    +(t^{\tm_r}-1)\int_{\Delta^{|\bfk|}(t)}\omega^{\mu}_{(\bfk;\bfm)}
    +\Li^{\mu}(\bfk;\bfm;t)-\Li^{\mu}(\bfk;\tbfm;t)\\
&=O\Big((1-t)\ln^{r}\Big(\frac1{1-t}\Big)\Big)+\Li^{\mu}(\bfk;\bfm;t)-\Li^{\mu}(\bfk;\tbfm;t)\qquad\text{by\ lemma\ \ref{kdov}}\\
&=O\Big((1-t)\ln^{r}\Big(\frac1{1-t}\Big)\Big)+\Li^{\mu}(\bfk;\bfm;1)-\Li^{\mu}(\bfk;\tbfm;1)+O\Big((1-t)\ln^{r}\Big(\frac1{1-t}\Big)\Big)\\
&=\int_{\Delta^{|\bfk|}}(\omega^{\mu}_{(\bfk;\bfm)}-\omega^{\mu}_{(\bfk;\tbfm)})+O\Big((1-t)\ln^{r}\Big(\frac1{1-t}\Big)\Big),
\end{align*}
as desired.
\end{proof}

\begin{thm}\label{thm:shaRegPsi}
The map of $(\Q[\mu]\langle X^{\mu}\rangle^0 ,\shuffle)$-algebra
\begin{align*}
\psi_{\shuffle,m}:(\Q[\mu]\langle X^{\mu}\rangle^0 ,\shuffle)[T]&\longrightarrow(\Q[\mu]\langle X^{\mu}\rangle^1,\shuffle),\\
T&\longmapsto y^{\mu}_{m}
\end{align*}
is an isomorphism.
\end{thm}

\begin{proof}
Let $(\bfk;\bfm)=(k_1,\cdots,k_r,1,\cdots,1;m_1,\cdots,m_r,m_{r+1},\cdots,m_{r+s})$. Then by definition,
$$(y,x)^{\mu}_{\bfk;\bfm}=y^{\mu}_{m_1}x^{k_1-1}\cdots y^{\mu}_{m_r-m_{r-1}}x^{k_r-1}y^{\mu}_{m_{r+1}-m_r}\cdots y^{\mu}_{m_{r+s}-m_{r+s-1}}.
$$
We now prove the theorem by induction on $s$. If $s=0$, then $(\bfk;\bfm)$ is an admissible bi-index and we are done.
We assume now the theorem holds for $s-1$ and write
\begin{align*}
(y,x)^{\mu}_{\bfk;\bfm}&=y^{\mu}_{m_1}x^{k_1-1}\cdots y^{\mu}_{m_r-m_{r-1}}x^{k_r-1}y^{\mu}_{m_{r+1}-m_r}\cdots y^{\mu}_{m_{r+s}-m_{r+s-1}}\\
&=y^{\mu}_{m_1}x^{k_1-1}\cdots y^{\mu}_{m_r-m_{r-1}}x^{k_r-1}(y_m^{\mu})^s\\
&\quad+y^{\mu}_{m_1}x^{k_1-1}\cdots y^{\mu}_{m_r-m_{r-1}}x^{k_r-1}(y^{\mu}_{m_{r+1}-m_r}-y^{\mu}_m)(y^{\mu}_m)^{s-1}\\
&\quad+\cdots\\
&\quad+y^{\mu}_{m_1}x^{k_1-1}\cdots y^{\mu}_{m_r-m_{r-1}}x^{k_r-1}y^{\mu}_{m_{r+1}-m_r}\cdots y^{\mu}_{m_{r+s-1}-m_{r+s-2}}(y^{\mu}_{m_{r+s}-m_{r+s-1}}-y^{\mu}_m).
\end{align*}
By simple calculation while keeping track of the word lengths, we have
\begin{align*}
y_{m_1}x^{k_1-1}\cdots y_{m_r-m_{r-1}}x^{k_r-1}(y_m^{\mu})^s=w_s\shuffle\overbrace{y^{\mu}_m\shuffle\cdots\shuffle y^{\mu}_m}^s+w_{s-1}\shuffle\overbrace{y^{\mu}_m\shuffle\cdots\shuffle y^{\mu}_m}^{s-1}+\cdots+w_0,
\end{align*}
where $w_0,\cdots,w_0\in\Q[\mu]\langle X^{\mu}\rangle^0$ (in fact, $\Q[\mu]\langle X^{\mu}\rangle_{\rm pre}^0$)
such that $\ell(w_j)=r+s-j$ for $0\leqslant j\leqslant s$.
For $i=1,\cdots,s$, we have
\begin{align*}
&y^{\mu}_{m_1}x^{k_1-1}\cdots y^{\mu}_{m_r-m_{r-1}}x^{k_r-1}y^{\mu}_{m_{r+1}-m_r}\cdots y^{\mu}_{m_{r+i-1}-m_{r+i-2}}(y^{\mu}_{m_{r+i}-m_{r+i-1}}-y^{\mu}_m)(y_m^{\mu})^{s-i}\\
&=w^{(s-i)}_{s-i}\shuffle\overbrace{y^{\mu}_m\shuffle\cdots\shuffle y^{\mu}_m}^{s-i}+w^{(s-i)}_{s-i-1}\shuffle\overbrace{y^{\mu}_m\shuffle\cdots\shuffle y^{\mu}_m}^{s-i-1}+\cdots+w^{(s-i)}_0,
\end{align*}
where $w^{(s-i)}_0,\cdots,w^{(s-i)}_{s-i}\in\Q[\mu]\langle X^{\mu}\rangle^0$
such that $\ell^{(s-i)}(w_j)=r+s-j$ for $0\leqslant j\leqslant s-i$. Hence,
$$(y,x)^{\mu}_{\bfk;\bfm}=u_s\shuffle\overbrace{y^{\mu}_m\shuffle\cdots\shuffle y^{\mu}_m}^s+u_{s-1}\shuffle\overbrace{y^{\mu}_m\shuffle\cdots\shuffle y^{\mu}_m}^{s-1}+\cdots+u_0,$$
where $u_0,\cdots,u_s\in\Q[\mu]\langle X^{\mu}\rangle^0$ such that $\ell(u_j)=r+s-j$ for $0\leqslant j\leqslant s$.
This completes the proof of the theorem.
\end{proof}

\begin{thm}\label{thm:shaReg_zetaMuT}
Let $(\bfk;\bfm)=(k_1,\cdots,k_r,1,\cdots,1;m_1,\cdots,m_r,m_{r+1},\cdots,m_{r+s}),k_r>1$, and $m\in\Q_{>0}$.
Then
$$\int_{\Delta^{|\bfk|}(t)}\omega^{\mu}_{(\bfk;\bfm)}=
 a_s\left(\int_{\Delta^1(t)}\omega^{\mu}_{(1;m)}\right)^s
 +a_{s-1}\left(\int_{\Delta^1(t)}\omega^{\mu}_{(1;m)}\right)^{s-1}
 +\cdots+a_0+O\Big((1-t)\ln^{r+s}\Big(\frac1{1-t}\Big)\Big)$$
where $a_0,\cdots,a_s\in\calZ^{\mu}$. We define the \emph{shuffle regularization homomorphism} by
\begin{align*}
\zeta^{\mu,T}_{\shuffle,m}:(\Q[\mu]\langle X^{\mu}\rangle^1,\shuffle)&\longrightarrow(\calZ^{\mu},\cdot)[T],\\
(y,x)^{\mu}_{\bfk;\bfm}&\longmapsto\sum_{i=0}^sa_iT^i.
\end{align*}
\end{thm}

\begin{proof} By Theorem~\ref{thm:shaRegPsi}, we conclude that
$$(y,x)^{\mu}_{\bfk;\bfm}=u_s\shuffle\overbrace{y^{\mu}_m\shuffle\cdots\shuffle y^{\mu}_m}^s+u_{s-1}\shuffle\overbrace{y^{\mu}_m\shuffle\cdots\shuffle y^{\mu}_m}^{s-1}+\cdots+u_0$$
where $u_0,\cdots,u_s\in\Q[\mu]\langle X^{\mu}\rangle^0$ such that $\ell(u_i)=r+s-i$. Then we have
\begin{align*}
\int_{\Delta^{|\bfk|}(t)}\omega^{\mu}_{(\bfk;\bfm)}&=\zeta^{\mu}_{\shuffle,t}((y,x)^{\mu}_{\bfk;\bfm})\\
&=\sum_{i=0}^s\zeta^{\mu}_{\shuffle,t}(u_i\shuffle\overbrace{y^{\mu}_m\shuffle\cdots\shuffle y^{\mu}_m}^i)\\
&=\sum_{i=0}^s\zeta^{\mu}_{\shuffle,t}(u_i)\cdot(\zeta^{\mu}_{\shuffle,t}\big(y^{\mu}_m)\big)^i\\
&=\sum_{i=0}^s\left(a_i+O\Big((1-t)\ln^{r+s-i}\Big(\frac1{1-t}\Big)\Big)\right)\cdot\left(\int_{\Delta^1(t)}\omega_{(1;m)}\right)^i\\
&=\sum_{i=0}^sa_i\cdot\left(\int_{\Delta^1(t)}\omega^{\mu}_{(1;m)}\right)^i+O\Big((1-t)\ln^{r+s}\Big(\frac1{1-t}\Big)\Big),
\end{align*}
where $a_i=\zeta^{\mu}_{\shuffle}(u_i)\in\calZ^{\mu}$, as desired.
\end{proof}

\begin{eg}
Let $m\in\Q_{>0}$ and $(\bfk;\bfm)=(2,1;m_1,m_2)$ be a positive bi-index. Then we have
\begin{align*}
(y,x)^{\mu}_{\bfk;\bfm}&=v_{(2,1;m_1,m_2)}\\
&=y^{\mu}_{m_1}xy^{\mu}_{m_2-m_1}\\
&=y^{\mu}_{m_1}xy^{\mu}_m+y^{\mu}_{m_1}x(y^{\mu}_{m_2-m_1}-y^{\mu}_m)\\
&=y^{\mu}_{m_1}x\shuffle y^{\mu}_m+y_{m_1}x(y^{\mu}_{m_2-m_1}-y^{\mu}_m)-y^{\mu}_my^{\mu}_{m_1}x-y^{\mu}_{m_1}y^{\mu}_mx.
\end{align*}
Hence,
\begin{align*}
\zeta^{\mu}_{\shuffle,t}((y,x)^{\mu}_{\bfk;\bfm})&=\zeta^{\mu}_{\shuffle_2,t}(y^{\mu}_{m_1}xy^{\mu}_{m_2-m_1})\\
&=\zeta^{\mu}_{\shuffle,t}(y^{\mu}_{m_1}x\shuffle y^{\mu}_m)+\zeta^{\mu}_{\shuffle,t}(y^{\mu}_{m_1}x(y^{\mu}_{m_2-m_1}-y^{\mu}_m))
-\zeta^{\mu}_{\shuffle,t}(y^{\mu}_my^{\mu}_{m_1}x)-\zeta^{\mu}_{\shuffle,t}(y^{\mu}_{m_1}y^{\mu}_mx)\\
&=\zeta^{\mu}_{\shuffle,t}(y^{\mu}_{m_1}x)\cdot\zeta^{\mu}_{\shuffle,t}(y^{\mu}_m)
+\zeta^{\mu}_{\shuffle,t}(y^{\mu}_{m_1}x(y^{\mu}_{m_2-m_1}-y^{\mu}_m))
-\zeta^{\mu}_{\shuffle,t}(y^{\mu}_my^{\mu}_{m_1}x)-\zeta^{\mu}_{\shuffle,t}(y^{\mu}_{m_1}y^{\mu}_mx)\\
&=\zeta^{\mu}(2;m_1)\cdot\int_{\Delta^1(t)}\omega^{\mu}_{(1;m)}+\calD^{\mu}(2,1;m_1,m_2;m_1,m+m_1)\\
&-\zeta^{\mu}(1,2;m,m+m_1)-\zeta^{\mu}(1,2;m_1,m_1+m)+O\Big((1-t)\ln^2\Big(\frac1{1-t}\Big)\Big).
\end{align*}
\end{eg}

\begin{thm}\label{zsok}
For all $m\in\Q_{>0}$, the map $\zeta^{\mu,T}_{\shuffle,m}:(\Q[\mu]\langle X^{\mu}\rangle^1,\shuffle)\longrightarrow(\calZ^{\mu},\cdot)[T]$ is a $\Q[\mu]$-algebra homomorphism.
\end{thm}

\begin{proof} By the definition of $\zeta^{\mu}_{\shuffle,t}$, we know that
$$\zeta^{\mu}_{\shuffle,t}(w\shuffle v)=\zeta^{\mu}_{\shuffle,t}(w)\cdot\zeta^{\mu}_{\shuffle,t}(v).$$
By the definition of $\zeta_{\shuffle,m}^{\mu,T}$, we have
$$\zeta^{\mu}_{\shuffle,t}(w)=\zeta_{\shuffle,m}^{\mu,T}(w)+O\Big((1-t)\ln^j\Big(\frac1{1-t}\Big)\Big)$$
where $T=\int_{\Delta^1(t)}\omega^{\mu}_{(1;m)}=O\Big(\ln\Big(\frac1{1-t}\Big)\Big)$
by Lemma \ref{kdov}(1). Taking $t\to 1^-$ we see that
\begin{align*}
\zeta^{\mu,T}_{\shuffle,m}(w\shuffle v)=\zeta^{\mu,T}_{\shuffle,m}(w)\cdot\zeta^{\mu,T}_{\shuffle,m}(v),
\end{align*}
which shows that $\zeta^{\mu,T}_{\shuffle,m}$ is a $\Q[\mu]$-algebra homomorphism.
\end{proof}

Furthermore, the map $\zeta^{\mu,T}_{\shuffle,m}$ is characterized by the following conditions.

\begin{thm}
For any given $m\in\Q_{>0}$, there exists a unique $\Q[\mu]$-linear map
$$\zeta_{\shuffle,m}^{\mu,T}:(\Q[\mu]\langle X^{\mu}\rangle^1,\shuffle)\longrightarrow(\calZ^{\mu},\cdot)[T]$$
satisfying
$$\begin{cases}
\zeta_{\shuffle,m}^{\mu,T}(w)=\zeta_{\shuffle}^{\mu}(w)\qquad&\forall w\in\Q[\mu]\langle X^{\mu}\rangle^0 \\
\zeta_{\shuffle,m}^{\mu,T}(y^{\mu}_m)=T;\\
\zeta_{\shuffle,m}^{\mu,T}(v\shuffle w)=\zeta_{\shuffle,m}^{\mu,T}(v)\cdot\zeta_{\shuffle,m}^{\mu,T}(w)&\forall v,w\in\Q[\mu]\langle X^{\mu}\rangle^1.
\end{cases}$$
\end{thm}

\begin{proof}
The existence of $\zeta_{\shuffle,m}^{\mu,T}$ is guaranteed by Theorem~\ref{zsok}.
Now, we only need to prove the uniqueness. By Theorem~\ref{thm:shaRegPsi}, we have
$$(y,x)^{\mu}_{\bfk;\bfm}=v_s\shuffle\overbrace{y^{\mu}_m\shuffle\cdots\shuffle y^{\mu}_m}^s+v_{s-1}\shuffle\overbrace{y_m\shuffle\cdots\shuffle y^{\mu}_m}^{s-1}+\cdots+v_0$$
where $v_0,\cdots,v_s\in\Q[\mu]\langle X\rangle^0$. If $f$ is another map satisfying the above properties, then we have
$$f((y,x)^{\mu}_{\bfk;\bfm})=\sum_{i=0}^sf(v_i)\cdot(f(y^{\mu}_m))^i=\sum_{i=0}^sf(v_i)\cdot T^i.$$
But $v_i$ is the $\Q$-linear combination of (quasi-)admissible words and therefore
$$f(v_i)=\zeta_{\shuffle,m}^{\mu,T}(v_i),$$
yielding the uniqueness of $\zeta_{\shuffle,m}^{\mu,T}$.
\end{proof}

\begin{eg} We have
\begin{align*}
\zeta^{\mu,T}_{\shuffle,m}(y_{m_1}xy_{m_2-m_1})
&=\zeta^{\mu}(2;m_1)\cdot T+\calD^{\mu}(2,1;m_1,m_2;m_1,m+m_1)\\
& -\zeta^{\mu}(1,2;m,m+m_1)-\zeta^{\mu}(1,2;m_1,m_1+m).
\end{align*}
\end{eg}

\subsection{The comparison theorem}
In the previous sections, we have put a lot of effort into constructing the $\mu$-stuffle regularization homomorphism $\zeta_{*,m}^{\mu,T}$ and the shuffle regularization homomorphism $\zeta_{\shuffle,m}^{\mu,T}$. As we just saw in the previous example, the regularization homomorphisms
$\zeta_{*,m}^{\mu,T}$ and $\zeta_{\shuffle,m}^{\mu,T}$ are in general different from each other.
In this section, we will discuss how to relate them to each other.

\begin{defn}
Let $x>-1$. Define the \emph{harmonic number function} by
$$
H(x):=\int_0^1\frac{1-t^x}{1-t}dt.
$$
\end{defn}

\begin{thm}\label{thm:stComparisonMap}
Fix $\mu\in\CC$ with $\Ree(\mu)>0$. For any given $m,m'\in\Q_{>0}$, there exists a unique $\R$-linear map
\begin{align*}
\myrho^{\mu,*}_{m,m'}:\R[T]&\longrightarrow\R[T],\\
T^s&\longmapsto(T+H(m'/\mu-1)-H(m/\mu-1))^s \quad \forall s\in\Z_{\geqslant0},
\end{align*}
called the \emph{$\mu$-stuffle comparison map}, that satisfies
\begin{equation}\label{equ:zetaSt}
\zeta_{*,m'}^{\mu,T}(w)=\myrho^{\mu,*}_{m,m'}(\zeta_{*,m}^{\mu,T}(w))\qquad\forall w\in \Q[\mu]\langle Y\rangle.
\end{equation}
That is, we have the following commutative diagram:
\begin{center}
\begin{tikzpicture}[scale=0.9]
\node (A) at (0,2) {$(\Q[\mu]\langle Y\rangle,\must)$};
\node (B) at (-2,0) {$(\calZ^{\mu},\cdot)[T]$};
\node (C) at (2,0) {$(\calZ^{\mu},\cdot)[T]$};
\draw[->] (A) to node[pos=.8,above] {$\zeta_{*,m}^{\mu,T}\ \ $}(B);
\draw[->] (A) to node[pos=.8,above] {$\ \ \zeta_{*,m'}^{\mu,T}$}(C);
\draw[->] (B) to (C) node[pos=.5,below] {$\myrho^{\mu,*}_{m,m'}$};
\end{tikzpicture}
\end{center}

\end{thm}

\begin{proof} We only need to prove the uniqueness of $\myrho^{\mu,*}_{m,m'}$ and \eqref{equ:zetaSt}.
Let $w\in\Q[\mu]\langle Y\rangle$ be any word. By Theorem~\ref{thm:stRegPsi}, we may write
\begin{align*}
w=v_s\must \overbrace{y_{1,m}\must \cdots\must y_{1,m}}^s+v_{s-1}\must \overbrace{y_{1,m}\must \cdots\must y_{1,m}}^{s-1}+\cdots+v_0
\end{align*}
where $v_0,\cdots,v_s\in\Q[\mu]\langle Y\rangle^0$. Then we have
\begin{align*}
\myrho^{\mu,*}_{m,m'}(\zeta_{*,m}^{\mu,T}(w))&=\myrho^{\mu,*}_{m,m'}\left(\sum_{i=0}^s\zeta_{*,m}^{\mu,T}(v_i)\cdot(\zeta_{*,m}^{\mu,T}(y_{1,m}))^i\right)\\
&=\myrho^{\mu,*}_{m,m'}\left(\sum_{i=0}^s\zeta_{*,m}^{\mu,T}(v_i)\cdot T^i\right)\\
&=\sum_{i=0}^s\zeta_{*,m}^{\mu,T}(v_i)\cdot\myrho^{\mu,*}_{m,m'}(T^i)\\
&=\sum_{i=0}^s\zeta_{*,m'}^{\mu,T}(v_i)\cdot(T+H(m'/\mu-1)-H(m/\mu-1))^i\\
&=\sum_{i=0}^s\zeta_{*,m'}^{\mu,T}(v_i)\cdot(\zeta_{*,m'}^{\mu,T}(y_{1,m'})+\zeta_{*,m'}^{\mu,T}(y_{1,m}-y_{1,m'}))^i\\
&=\sum_{i=0}^s\zeta_{*,m'}^{\mu,T}(v_i)\cdot(\zeta_{*,m'}^{\mu,T}(y_{1,m}))^i\\
&=\sum_{i=0}^s\zeta_{*,m'}^{\mu,T}(v_i)\cdot\zeta_{*,m'}^{\mu,T}(\overbrace{y_{1,m}\must \cdots\must y_{1,m}}^i)\\
&=\sum_{i=0}^s\zeta_{*,m'}^{\mu,T}(v_i\must \overbrace{y_{1,m}\must \cdots\must y_{1,m}}^i)\\
&=\zeta_{*,m'}^{\mu,T}\left(\sum_{i=0}^sv_i\must \overbrace{y_{1,m}\must \cdots\must y_{1,m}}^i\right)\\
&=\zeta_{*,m'}^{\mu,T}(w).
\end{align*}
Next, we prove the uniqueness of $\myrho^{\mu,*}_{m,m'}$. Assume there exists another $\R$-linear map $f:\R[T]\longrightarrow\R[T]$ satisfying
$$
\zeta_{*,m'}^{\mu,T}(w)=f(\zeta_{*,m}^{\mu,T}(w))\qquad\forall w\in \Q[\mu]\langle Y\rangle.
$$
By taking
$$w=\overbrace{y_{1,m}\must \cdots\must y_{1,m}}^s$$
we find that
\begin{align*}
f(T^s)&=f((\zeta_{*,m}^{\mu,T}(y_{1,m}))^s)\\
&=f(\zeta_{*,m}^{\mu,T}(\overbrace{y_{1,m}\must \cdots\must y_{1,m}}^s))\\
&=f(\zeta_{*,m}^{\mu,T}(w))\\
&=\zeta_{*,m'}^{\mu,T}(w)\\
&=\zeta_{*,m'}^{\mu,T}(\overbrace{y_{1,m}\must \cdots\must y_{1,m}}^s)\\
&=(\zeta_{*,m'}^{\mu,T}(y_{1,m}))^s\\
&=(\zeta_{*,m'}^{\mu,T}((y_{1,m}-y_{1,m'})+y_{1,m'}))^s\\
&=(\zeta_{*,m'}^{\mu,T}(y_{1,m}-y_{1,m'})+\zeta_{*,m'}^{\mu,T}(y_{1,m'}))^s\\
&=(T+H(m'/\mu-1)-H(m/\mu-1))^s\\
&=\myrho^{\mu,*}_{m,m'}(T^s).
\end{align*}
This proves the uniqueness of $\myrho^{\mu,*}_{m,m'}$.
\end{proof}

Next, we will compare the maps $\zeta_{\shuffle,m}^{\mu,T}$ and $\zeta_{\shuffle,m'}^{\mu,T}$.

\begin{thm}\label{thm:shaComparisonMap}
Fix $\mu\in\CC$ with $\Ree(\mu)>0$. For any given $m,m'\in\Q_{>0}$, there exists a unique $\R$-linear map
\begin{align*}
\myrho^{\mu,\shuffle}_{m,m'}:\R[T]&\longrightarrow\R[T],\\
T^s&\longmapsto(T+H(m'/\mu-1)-H(m/\mu-1))^s \quad \forall s\in\Z_{\geqslant0},
\end{align*}
called \emph{the shuffle comparison map}, that satisfies
\begin{equation}\label{equ:zetaSh}
\zeta_{\shuffle,m'}^{\mu,T}(w)=\myrho^{\mu,\shuffle}_{m,m'}(\zeta_{\shuffle,m}^{\mu,T}(w))\qquad\forall w\in \Q[\mu]\langle X^{\mu}\rangle^1.
\end{equation}
That is, we have the following commutative diagram:
\begin{center}
\begin{tikzpicture}[scale=0.9]
\node (A) at (0,2) {$(\Q[\mu]\langle X^{\mu}\rangle^1,\shuffle)$};
\node (B) at (-2,0) {$(\calZ^{\mu},\cdot)[T]$};
\node (C) at (2,0) {$(\calZ^{\mu},\cdot)[T]$};
\draw[->] (A) to node[pos=.8,above] {$\zeta_{\shuffle,m}^{\mu,T}\ \ $}(B);
\draw[->] (A) to node[pos=.8,above] {$\ \ \zeta_{\shuffle,m'}^{\mu,T}$}(C);
\draw[->] (B) to (C) node[pos=.5,below] {$\myrho^{\mu,\shuffle}_{m,m'}$};
\end{tikzpicture}
\end{center}
\end{thm}

\begin{proof}
We only need to prove the uniqueness of $\myrho^{\mu,\shuffle}_{m,m'}$ and \eqref{equ:zetaSh}.
Let $w\in\Q[\mu]\langle X^{\mu}\rangle^1$ be any word. By Theorem~\ref{thm:shaRegPsi}, we may write
\begin{align*}
w=v_s\shuffle\overbrace{y^{\mu}_{m}\shuffle\cdots\shuffle y^{\mu}_{m}}^s+v_{s-1}\shuffle\overbrace{y^{\mu}_{m}\shuffle\cdots\shuffle y^{\mu}_{m}}^{s-1}+\cdots+v_0
\end{align*}
where $v_0,\cdots,v_s\in\Q[\mu]\langle X^{\mu}\rangle^0 $. Then we have
\begin{align*}
\myrho^{\mu,\shuffle}_{m,m'}(\zeta_{\shuffle,m}^{\mu,T}(w))&=\myrho^{\mu,\shuffle}_{m,m'}\left(\sum_{i=0}^s\zeta_{\shuffle,m}^{\mu,T}(v_i)\cdot(\zeta_{\shuffle,m}^{\mu,T}(y^{\mu}_m))^i\right)\\
&=\myrho^{\mu,\shuffle}_{m,m'}\left(\sum_{i=0}^s\zeta_{\shuffle,m}^{\mu,T}(v_i)\cdot T^i\right)\\
&=\sum_{i=0}^s\zeta_{\shuffle,m}^{\mu,T}(v_i)\cdot\myrho^{\mu,\shuffle}_{m,m'}(T^i)\\
&=\sum_{i=0}^s\zeta_{\shuffle,m'}^{\mu,T}(v_i)\cdot(T+H(m'/\mu-1)-H(m/\mu-1))^i\\
&=\sum_{i=0}^s\zeta_{\shuffle,m'}^{\mu,T}(v_i)\cdot(\zeta_{\shuffle,m'}^{\mu,T}(y^{\mu}_{m'})+\zeta_{\shuffle,m'}^{\mu,T}(y^{\mu}_m-y^{\mu}_{m'}))^i\\
&=\sum_{i=0}^s\zeta_{\shuffle,m'}^{\mu,T}(v_i)\cdot(\zeta_{\shuffle,m'}^{\mu,T}(y_{m}))^i\\
&=\sum_{i=0}^s\zeta_{\shuffle,m'}^{\mu,T}(v_i)\cdot\zeta_{\shuffle,m'}^{\mu,T}(\overbrace{y^{\mu}_{m}\shuffle\cdots\shuffle y^{\mu}_{m}}^i)\\
&=\sum_{i=0}^s\zeta_{\shuffle,m'}^{\mu,T}(v_i\shuffle\overbrace{y^{\mu}_{m}\shuffle\cdots\shuffle y^{\mu}_{m}}^i)
\end{align*}
\begin{align*}
&=\zeta_{\shuffle,m'}^{\mu,T}\left(\sum_{i=0}^sv_i\shuffle\overbrace{y^{\mu}_{m}\shuffle\cdots\shuffle y^{\mu}_{m}}^i\right)\\
&=\zeta_{\shuffle,m'}^{\mu,T}(w).
\end{align*}
Next, we prove the uniqueness of $\myrho^{\mu,\shuffle}_{m,m'}$. Assume there exists another $\R$-linear map $f:\R[T]\longrightarrow\R[T]$ satisfying
$$\zeta_{\shuffle,m'}^{\mu,T}(w)=f(\zeta_{\shuffle,m}^{\mu,T}(w))\qquad\forall w\in \Q[\mu]\langle X^{\mu}\rangle^1.$$
By taking
$$w=\overbrace{y^{\mu}_{m}\shuffle\cdots\shuffle y^{\mu}_{m}}^s$$
we see that
\begin{align*}
f(T^s)&=f((\zeta_{\shuffle,m}^{\mu,T}(y^{\mu}_m))^s)\\
&=f(\zeta_{\shuffle,m}^{\mu,T}(\overbrace{y^{\mu}_{m}\shuffle\cdots\shuffle y^{\mu}_{m}}^s))\\
&=f(\zeta_{\shuffle,m}^{\mu,T}(w))\\
&=\zeta_{\shuffle,m'}^{\mu,T}(w)\\
&=\zeta_{\shuffle,m'}^{\mu,T}(\overbrace{y^{\mu}_{m}\shuffle\cdots\shuffle y^{\mu}_{m}}^s)\\
&=(\zeta_{\shuffle,m'}^{\mu,T}(y^{\mu}_{m}))^s\\
&=(\zeta_{\shuffle,m'}^{\mu,T}((y^{\mu}_m-y^{\mu}_{m'})+y_{m'}))^s\\
&=(\zeta_{\shuffle,m'}^{\mu,T}(y^{\mu}_m-y^{\mu}_{m'})+\zeta_{\shuffle,m'}^{\mu,T}(y^{\mu}_{m'}))^s\\
&=(T+H(m'/\mu-1)-H(m/\mu-1))^s\\
&=\myrho^{\mu,\shuffle}_{m,m'}(T^s).
\end{align*}
This verifies the uniqueness of $\myrho^{\mu,\shuffle}_{m,m'}$.
\end{proof}

Finally, we will compare the maps $\zeta_{*,m}^{\mu,T}$ and $\zeta_{\shuffle,m'}^{\mu,T}$.
This is the most important and interesting case, of which we will give a detailed analysis.

\begin{lem}\label{oijx}
As $t\to 1^-$, we have
$$
\sum_{m=0}^{\infty}\frac{\ln^j(m+1)}{m+1}t^m=O\left(\ln^{j+1}\left(\frac{1}{1-t}\right)\right).
$$
\end{lem}

\begin{proof} See \cite[Lemma 2.3.9]{ZhaoBook}.
\end{proof}

\begin{thm}\label{thm:mixedComparisonMap}
Fix $\mu\in\CC$ with $\Ree(\mu)>0$. For any given $m,m'\in\Q_{>0}$, there exists a unique $\R$-linear map
$$\rho_{m,m'}^{\mu}:\R[T]\longrightarrow\R[T],$$
called the \emph{mixed comparison map}, that satisfies
$$
\zeta_{\shuffle,m'}^{\mu,T}(\varphi_{\mu}(w))=\rho_{m,m'}^{\mu}(\zeta_{*,m}^{\mu,T}(w))\qquad\forall w\in \Q[\mu]\langle Y\rangle,
$$
where $\varphi_{\mu}$ is defined by \eqref{equ:bijection-varphi}.
Namely, the following diagram commutes:
$$\begin{CD}
(\Q[\mu]\langle Y\rangle,\must)@>{\varphi_{\mu}}>>(\Q[\mu]\langle X^{\mu}\rangle^1,\shuffle)\\
@V\zeta_{*,m}^{\mu,T}VV@VV\zeta_{\shuffle,m'}^{\mu,T}V\\
(\calZ^{\mu},\cdot)[T]@>{\rho_{m,m'}^{\mu}}>>(\calZ^{\mu},\cdot)[T]\end{CD}$$
\end{thm}

\begin{proof}
We construct $\rho_{m,m'}^{\mu}$ first. Consider the word
\begin{align*}
&\varphi_{\mu}(\overbrace{y_{1,m}\must \cdots\must y_{1,m}}^s)\\
&=w_{s}^{(s)}\shuffle\overbrace{y^{\mu}_{m'}\shuffle\cdots\shuffle y^{\mu}_{m'}}^s
+w_{s-1}^{(s)}\shuffle\overbrace{y^{\mu}_{m'}\shuffle\cdots\shuffle y^{\mu}_{m'}}^{s-1}
+\cdots+w_1^{(s)}\shuffle y^{\mu}_{m'}+w_0^{(s)}
\end{align*}
where $w_0^{(s)},\cdots,w_{s}^{(s)}\in\Q[\mu]\langle X^{\mu}\rangle^0$. Let
$$a_i^{(s)}:=\zeta_{*,m}^{\mu,T}(\varphi_{\mu}^{-1}(w_i^{(s)}))=\zeta_{\shuffle,m'}^{\mu,T}(w_i^{(s)})\in\calZ^{\mu}.$$
We define $\rho_{m,m'}^{\mu}$ by setting
$$\rho_{m,m'}^{\mu}(T^s):=a_s^{(s)}T^s+a_{s-1}^{(s)}T^{s-1}+\cdots+a_0^{(s)}\in(\calZ^{\mu},\cdot)[T] \quad \forall s\geqslant 0.$$
Noticing that $\rho_{m,m'}^{\mu}$ is a $\R$-linear map, we obtain a map
$$
\rho_{m,m'}^{\mu}:(\calZ^{\mu},\cdot)[T]\longrightarrow(\calZ^{\mu},\cdot)[T]
$$
by linear expansion.

Next, we will prove that the map $\rho_{m,m'}^{\mu}$ constructed above satisfies
$$
\zeta_{\shuffle,m'}^{\mu,T}(\varphi_{\mu}(w))=\rho_{m,m'}^{\mu}(\zeta_{*,m}^{\mu,T}(w))\qquad\forall w\in \Q[\mu]\langle Y\rangle.
$$
Let $w=y_{\bfk;\bfm}\in\Q[\mu]\langle Y\rangle$, where $(\bfk;\bfm)$ is a positive bi-index. By Theorem-Definition~\ref{thm:shaReg_zetaMuT}, we have
$$\int_{\Delta^{|\bfk|}(t)}\omega^{\mu}_{(\bfk;\bfm)}=
b_s\left(\int_{\Delta^1(t)}\omega^{\mu}_{(1;m')}\right)^s
+b_{s-1}\left(\int_{\Delta^1(t)}\omega^{\mu}_{(1;m')}\right)^{s-1}
+\cdots+b_0+O\Big((1-t)\ln^{r}\Big(\frac1{1-t}\Big)\Big).$$
Hence,
$$\zeta_{\shuffle,m'}^{\mu,T}(\varphi_{\mu}(w))=b_sT^s+b_{s-1}T^{s-1}+\cdots+b_0$$
and
$$\int_{\Delta^{|\bfk|}(t)}\omega^{\mu}_{(\bfk;\bfm)}=
\zeta_{\shuffle,m'}^{\mu,T}(\varphi_{\mu}(w))+O\Big((1-t)\ln^{r}\Big(\frac1{1-t}\Big)\Big)$$
where $T=\int_{\Delta^1(t)}\omega^{\mu}_{(1;m')}$. By explicit calculations, we find
\begin{align*}
\int_{\Delta^{|\bfk|}(t)}\omega^{
\mu}_{(\bfk;\bfm)}&=t^{m_r}\Li^{\mu}(\bfk;\bfm;t)\\
&=\sum_{0\leqslant n_1\leqslant\cdots\leqslant n_r}\frac{\mu^rt^{n_r\mu+m_r}}{(n_1\mu+m_1)^{k_1}\cdots(n_r\mu+m_r)^{k_r}}\\
&=\sum_{n=0}^{\infty}\left(\sum_{0\leqslant n_1\leqslant\cdots\leqslant n_{r-1}\leqslant n}\frac{\mu^r}{(n_1\mu+m_1)^{k_1}\cdots(n_{r-1}\mu+m_{r-1})^{k_{r-1}}(n\mu+m_r)^{k_r}}\right)t^{n\mu+m_r}\\
&=\sum_{n=0}^{\infty}(\zeta_n^{\mu}(\bfk;\bfm)-\zeta_{n-1}^{\mu}(\bfk;\bfm))t^{n\mu+m_r}\\
&=(1-t^{\mu})\sum_{n=0}^{\infty}\zeta_n^{\mu}(\bfk;\bfm)t^{n\mu+m_r}.
\end{align*}
By Theorem-Definition~\ref{thm:stReg_zetaMuT}, $$\zeta_n^{\mu}(\bfk;\bfm)=a_s(\zeta_n^{\mu}(1;m))^s+a_{s-1}(\zeta_n^{\mu}(1;m))^{s-1}+\cdots+a_0+O\left(\frac{\ln^{r-1}(n+1)}{n+1}\right)$$
where $a_0,\cdots,a_s\in\calZ^{\mu}$. Then we see that
\begin{align*}
&(1-t^{\mu})\sum_{n=0}^{\infty}\zeta_n^{\mu}(\bfk;\bfm)t^{n\mu+m_r}\\
&=(1-t^{\mu})\sum_{n=0}^{\infty}\left(\sum_{i=0}^sa_i(\zeta^{\mu}_n(1;m))^i+O\left(\frac{\ln^{r-1}(n+1)}{n+1}\right)\right)\cdot t^{n\mu+m_r}\\
&=(1-t^{\mu})t^{m_r-m}\sum_{i=0}^sa_i\sum_{n=0}^{\infty}(\zeta^{\mu}_n(1;m))^i\cdot t^{n\mu+m}\\
&\quad+(1-t^{\mu})t^{m_r}\sum_{n=0}^{\infty}O\left(\frac{\ln^{r-1}(n+1)}{n+1}\right)\cdot t^{n\mu}\\
&=(1-t^{\mu})t^{m_r-m}\sum_{i=0}^sa_i\sum_{n=0}^{\infty}(\zeta^{\mu}_n(1;m))^i\cdot t^{n\mu+m}
+O\Big((1-t)\ln^{r}\Big(\frac1{1-t}\Big)\Big)
\end{align*}
by Lemma \ref{oijx}. Observe that
\begin{align*}
&(1-t^{\mu})\sum_{n=0}^{\infty}(\zeta^{\mu}_n(1;m))^i\cdot t^{n\mu+m}\\
&=(1-t^{\mu})\sum_{n=0}^{\infty}\zeta^{\mu}_{*,n}(\overbrace{y_{1,m}\must \cdots\must y_{1,m}}^i)\cdot t^{n\mu+m}\\
&=\sum_{n=0}^{\infty}(\zeta^{\mu}_{*,n}(\overbrace{y_{1,m}\must \cdots\must y_{1,m}}^i)-\zeta^{\mu}_{*,n-1}(\overbrace{y_{1,m}\must \cdots\must y_{1,m}}^i))\cdot t^{n\mu+m}\\
&=\zeta^{\mu}_{\shuffle,t}(\varphi_{\mu}(\overbrace{y_{1,m}\must \cdots\must y_{1,m}}^i))\\
&=a_i^{(i)}\left(\int_{\Delta^1(t)}\omega^{\mu}_{(1;m')}\right)^i+a_{i-1}^{(i)}\left(\int_{\Delta^1(t)}\omega^{\mu}_{(1;m')}\right)^{i-1}+\cdots+a_{0}^{(i)}+O((1-t)\ln^i(1/(1-t)))\\
&=\rho_{m,m'}^{\mu}\left(\left(\int_{\Delta^1(t)}\omega^{\mu}_{(1;m')}\right)^i\right)
 +O\Big((1-t)\ln^{i}\Big(\frac1{1-t}\Big)\Big).
\end{align*}
We thus conclude that
\begin{align*}
\int_{\Delta^{|\bfk|}(t)}\omega^{\mu}_{(\bfk;\bfm)}&=\rho_{m,m'}^{\mu}\left(\sum_{i=0}^sa_i\cdot\left(\int_{\Delta^1(t)}\omega^{\mu}_{(1;m')}\right)^i\right)+O\Big((1-t)\ln^{r}\Big(\frac1{1-t}\Big)\Big)\\
&=\rho_{m,m'}^{\mu}(\zeta_{*,m}^{\mu,T}(w))+O\Big((1-t)\ln^{r}\Big(\frac1{1-t}\Big)\Big)
\end{align*}
where $T=\int_{\Delta^1(t)}\omega^{\mu}_{(1;m')}$. This yields readily
$$\zeta_{\shuffle,m'}^{\mu,T}(\varphi_{\mu}(w))
=\rho_{m,m'}^{\mu}(\zeta_{*,m}^{\mu,T}(w))+O\Big((1-t)\ln^{r}\Big(\frac1{1-t}\Big)\Big).$$
Hence,
$$
\zeta_{\shuffle,m'}^{\mu,T}(\varphi_{\mu}(w))=\rho_{m,m'}^{\mu}(\zeta_{*,m}^{\mu,T}(w)).
$$

Next, we prove the uniqueness of $\rho_{m,m'}^{\mu}$. Assume there exists another $\R$-linear map $f:\R[T]\longrightarrow\R[T]$ satisfying
$$
\zeta_{\shuffle,m'}^{\mu,T}(\varphi_{\mu}(w))=f(\zeta_{*,m}^{\mu,T}(w))\qquad\forall w\in \Q[\mu]\langle Y\rangle.
$$
Then taking
$$w=\overbrace{y_{1,m}\must \cdots\must y_{1,m}}^s$$
we get
$$f(T^s)=f(\zeta_{*,m}^{\mu,T}(w))=\zeta_{\shuffle,m'}^{\mu,T}(\varphi(w))
=\rho_{m,m'}^{\mu}(\zeta^{\mu,T}_{*,m}(w))=\rho_{m,m'}^{\mu}(T^s).$$
This shows the uniqueness of $\rho_{m,m'}^{\mu}$.
\end{proof}

\begin{thm}\label{pozc}
We have the following commutative diagram
$$\begin{CD}
(\calZ^{\mu},\cdot)[T]@>\myrho^{\mu,*}_{m_1,m_1'}>>(\calZ^{\mu},\cdot)[T]\\
@V{\rho_{m_1,m_2}^{\mu}}VV@VV{\rho_{m_1',m_2'}^{\mu}}V\\
(\calZ^{\mu},\cdot)[T]@>\myrho^{\mu,\shuffle}_{m_2,m_2'} >>(\calZ^{\mu},\cdot)[T]\end{CD}$$
Hence, we have a network composed of comparative mappings as shown by the following diagram
$$\begin{tikzpicture}
\draw (-6,3) -- (6,3);
\draw (-6,-3) -- (6,-3);
\draw (-6,1) -- (6,1);
\draw (-6,-1) -- (6,-1);
\draw (4,-3) -- (4,3);
\draw (2,-3) -- (2,3);
\draw (0,-3) -- (0,3);
\draw (-2,-3) -- (-2,3);
\draw (-4,-3) -- (-4,3);
\draw (4,-1) -- (2,-3);
\draw (4,1) -- (0,-3);
\draw (4,3) -- (-2,-3);
\draw (2,3) -- (-4,-3);
\draw (0,3) -- (-4,-1);
\draw (-2,3) -- (-4,1);
\draw (2,3) -- (4,1);
\draw (0,3) -- (4,-1);
\draw (-2,3) -- (4,-3);
\draw (-4,3) -- (2,-3);
\draw (-4,1) -- (0,-3);
\draw (-4,-1) -- (-2,-3);
\draw (-4,3) -- (0,1);
\draw (-4,3) -- (2,1);
\draw (-4,3) -- (4,1);
\draw (-2,3) -- (2,1);
\draw (-2,3) -- (4,1);
\draw (0,3) -- (-4,1);
\draw (0,3) -- (4,1);
\draw (2,3) -- (-4,1);
\draw (2,3) -- (-2,1);
\draw (4,3) -- (-4,1);
\draw (4,3) -- (-2,1);
\draw (4,3) -- (0,1);
\draw (4,-1) -- (4,1);
\draw (4,-1) -- (2,1);
\draw (4,-1) -- (0,1);
\draw (4,-1) -- (-2,1);
\draw (4,-1) -- (-4,1);
\draw (2,-1) -- (4,1);
\draw (2,-1) -- (2,1);
\draw (2,-1) -- (0,1);
\draw (2,-1) -- (-2,1);
\draw (2,-1) -- (-4,1);
\draw (0,-1) -- (4,1);
\draw (0,-1) -- (2,1);
\draw (0,-1) -- (0,1);
\draw (0,-1) -- (-2,1);
\draw (0,-1) -- (-4,1);
\draw (-2,-1) -- (4,1);
\draw (-2,-1) -- (2,1);
\draw (-2,-1) -- (0,1);
\draw (-2,-1) -- (-2,1);
\draw (-2,-1) -- (-4,1);
\draw (-4,-1) -- (4,1);
\draw (-4,-1) -- (2,1);
\draw (-4,-1) -- (0,1);
\draw (-4,-1) -- (-2,1);
\draw (-4,-1) -- (-4,1);
\draw (-4,-3) -- (0,-1);
\draw (-4,-3) -- (2,-1);
\draw (-4,-3) -- (4,-1);
\draw (-2,-3) -- (2,-1);
\draw (-2,-3) -- (4,-1);
\draw (0,-3) -- (-4,-1);
\draw (0,-3) -- (4,-1);
\draw (2,-3) -- (-4,-1);
\draw (2,-3) -- (-2,-1);
\draw (4,-3) -- (-4,-1);
\draw (4,-3) -- (-2,-1);
\draw (4,-3) -- (0,-1);

\node [above] at (-4,3) {$1$};
\node [above] at (-2,3) {$2$};
\node [above] at (-0,3) {$3$};
\node [above] at (2,3) {$4$};
\node [above] at (4,3) {$5$};
\node [below] at (-4,-3) {$1$};
\node [below] at (-2,-3) {$2$};
\node [below] at (-0,-3) {$3$};
\node [below] at (2,-3) {$4$};
\node [below] at (4,-3) {$5$};
\node [above] at (5,3) {$\cdots$};
\node at (5,0) {$\cdots$};
\node at (5,2) {$\cdots$};
\node at (5,-2) {$\cdots$};
\node [below] at (5,-3) {$\cdots$};
\node [left] at (-6,1) {$y_{1,m_1}$};
\node [left] at (-6,3) {$y_{1,m_1'}$};
\node [left] at (-6,-1) {$y^{\mu}_{m_2}$};
\node [left] at (-6,-3) {$y^{\mu}_{m_2'}$};
\node at (-5,2) {$\myrho^{\mu,*}_{m_1',m_1}$};
\node at (-5,0) {$\rho_{m_1,m_2}^{\mu}$};
\node at (-5,-2) {$\myrho^{\mu,\shuffle}_{m_2,m_2'}$};
\draw [->] (-4.3,2.75) -- (-4.3,1.25);
\draw [->] (-4.3,0.75) -- (-4.3,-0.75);
\draw [->] (-4.3,-1.25) -- (-4.3,-2.75);
\draw [->] (-6.8,2.7) arc [radius=8, start angle=160, end angle= 199];
\node [left] at (-7.3,0) {$\rho_{m_1',m_2'}^{\mu}$};
\end{tikzpicture}$$
\end{thm}

\begin{proof} This follows from Theorem~\ref{thm:stComparisonMap}, Theorem~\ref{thm:shaComparisonMap} and Theorem~\ref{thm:mixedComparisonMap}.
\end{proof}

We can now give a detailed description of the map $\rho_{m,m'}^{\mu}$. By Theorem~\ref{pozc}, we only need to consider the map $\rho_{1,1}^{\mu}$. We first recall a classical result.

\begin{thm}\label{loaz}
\emph{(Abel-Plana formula)}
Let $f:\left\{z\in\CC\big|\Ree(z)\geqslant0\right\}\longrightarrow\CC$ be a holomorphic function, and assume that $|f(z)|$ is bounded by $C/|z|^{1+\varepsilon}$ in $\left\{z\in\CC\big|\Ree(z)\geqslant0\right\}$ for some constants $C,\varepsilon>0$. Then we have
$$
\sum_{n=0}^{\infty}f(n+a)=\int_a^{\infty}f(x)dx+\frac{f(a)}{2}+\int_0^{\infty}\frac{f(a-ix)-f(a+ix)}{i(e^{2\pi x}-1)}dx.
$$
For the case $a=0$, we have
\begin{equation}\label{equ:Abel-Plana}
\sum_{n=0}^{\infty}f(n)=\int_0^{\infty}f(x)dx+\frac{f(0)}{2}+i\int_0^{\infty}\frac{f(ix)-f(-ix)}{e^{2\pi x}-1}dx.
\end{equation}
\end{thm}

\begin{proof}
For the original proof by Abel, see \cite{Abel1823}.
For the proof of a generalized version of the theorem, see \cite{Saharian2000}.
\end{proof}

\begin{lem}\label{zbxd}
As $t\to1^-$, we have
$$\int_0^t\frac{\mu}{1-\tau^{\mu}}d\tau=\ln\left(\frac{1}{1-t}\right)+c_0(\mu)+c_1(\mu)(1-t)+c_2(\mu)(1-t)^2+O((1-t)^3),$$
where
\begin{align*}
c_0(\mu)&=\ln\left(\frac{1}{\mu}\right)-H\left(\frac{1}{\mu}-1\right),\
c_1(\mu)=\frac{1-\mu}{2},\
c_2(\mu)=\frac{1-\mu^2}{24},\
c_3(\mu)=\frac{1-\mu^2}{72},\\
c_4(\mu)&=\frac{(1-\mu^2)(19-\mu^2)}{2880},\
c_5(\mu)=\frac{(1-\mu^2)(9-\mu^2)}{2400},
\ldots
\end{align*}
Generally, we have $c_n(\mu)\in\Q[\mu]$ for all $n\in\N$.
\end{lem}

\begin{proof}
To compute $c_0(\mu)$ we make the change of variable $\tau\to \tau^{1/\mu}$ on the left and then combine it with
$\ln(1-t)=-\int_0^t 1/(1-\tau) d\tau$. The other coefficients can be computed by repeatedly applying
L'H\^opital's rule. For example, setting
$$
\psi_{\mu}(t):=\int_0^t\frac{\mu}{1-\tau^{\mu}}d\tau-\ln\left(\frac{1}{1-t}\right)-c_0(\mu),\qquad0<t<1,
$$
we then get
$$
\lim_{t\to 1^-}\frac{\psi_{\mu}(t)}{1-t}=\frac{1-\mu}{2}
$$
by L'H\^opital's rule. We leave the details to the interested reader.
\end{proof}

\begin{lem}\label{oonb}
If $0<t<1$, then we have
$$\sum_{m=0}^{\infty}\ln^j(m+1)t^m=\int_0^{\infty}\ln^j(x+1)t^xdx+O(1)$$
and
$$\sum_{m=0}^{\infty}\frac{\ln^j(m+1)}{m+1}t^m=\int_0^{\infty}\frac{\ln^j(x+1)}{x+1}t^x dx+O(1).$$
\end{lem}

\begin{proof}
Taking $f(x)=\ln^j(x+1)t^x$ in the Abel-Plana formula \eqref{equ:Abel-Plana}, we have
$$\sum_{m=0}^{\infty}\ln^j(m+1)t^m=\int_0^{\infty}\ln^j(x+1)t^xdx
 +i\int_0^{\infty}\frac{\ln^j(1+ix)t^{ix}-\ln^j(1-ix)t^{-ix}}{e^{2\pi x}-1}dx.
$$
For the second term above, we have
\begin{align*}
&\left|i\int_0^{\infty}\frac{\ln^j(1+ix)t^{ix}-\ln^j(1-ix)t^{-ix}}{e^{2\pi x}-1}dx\right|\\
&\leqslant\int_0^{\infty}\frac{|\ln(1+ix)|^j+|\ln(1-ix)|^j}{e^{2\pi x}-1}dx\\
&=\int_0^{1}\frac{|\ln(1+ix)|^j+|\ln(1-ix)|^j}{e^{2\pi x}-1}dx+\int_1^{\infty}\frac{|\ln(1+ix)|^j+|\ln(1-ix)|^j}{e^{2\pi x}-1}dx\\
&\leqslant\int_0^{1}\frac{|\ln(1+ix)|^j+|\ln(1-ix)|^j}{e^{2\pi x}-1}dx+2\int_1^{\infty}\left(\left(\frac{\ln(1+x^2)}{2}\right)^2+\left(\frac{\pi}{2}\right)^2\right)^\frac{j}{2}\frac{dx}{e^{2\pi x}-1}\\
&<+\infty.
\end{align*}
Moreover, if we let $f(x)=\frac{\ln^j(x+1)}{x+1}t^x$, then we may use the same procedure as above to get
$$
\sum_{m=0}^{\infty}\frac{\ln^j(m+1)}{m+1}t^m=\int_0^{\infty}\frac{\ln^j(x+1)}{x+1}t^xdx+O(1).
$$
This completes the proof of the lemma.
\end{proof}

\begin{lem}\label{sfiz}
Let $j$ be a positive integer. Then we have
$$\sum_{m=0}^{\infty}\frac{\Gamma^{(j)}(m+1)}{\Gamma(m+1)}t^m=
\frac{1}{1-t}O\Big(\ln^j\Big(\frac{1}{1-t}\Big)\Big)\quad \text{as} \quad t\to 1^-.$$
\end{lem}

\begin{proof}
For all $j\geqslant 0$, we have the estimate
$$
\frac{\Gamma^{(j)}(m+1)}{\Gamma(m+1)}=\ln^j(m+1)+O\left(\frac{\ln^{j-1}(m+1)}{m+1}\right) \quad \text{as} \quad m\to\infty.
$$
Hence
\begin{align*}
\sum_{m=0}^{\infty}\frac{\Gamma^{(j)}(m+1)}{\Gamma(m+1)}t^m
&=\sum_{m=0}^{\infty}\ln^j(m+1)t^m+\sum_{m=0}^{\infty}O\left(\frac{\ln^{j-1}(m+1)}{m+1}\right)t^m.
\end{align*}
By Lemma~\ref{oijx} we have
$$\sum_{m=0}^{\infty}O\left(\frac{\ln^{j-1}(m+1)}{m+1}\right)t^m=O\Big(\ln^j\Big(\frac{1}{1-t}\Big)\Big)\quad \text{as} \quad t\to1^-.$$
By Lemma~\ref{oonb} we get
$$\sum_{m=0}^{\infty}\ln^j(m+1)t^m=\int_0^{\infty}\ln^j(x+1)t^xdx+O(1)\quad \text{as} \quad t\to1^-.$$
Finally,
\begin{align*}
\int_0^{\infty}\ln^j(x+1)t^xdx&=\ln^j(x+1)\frac{t^x}{\ln(t)}\bigg|_0^{\infty}-\frac{1}{\ln(t)}\int_0^{\infty}\frac{\ln^{j-1}(x+1)}{x+1}t^xdx\\
&=-\frac{1}{\ln(t)}\int_0^{\infty}\frac{\ln^{j-1}(x+1)}{x+1}t^xdx\\
&=\frac{1}{1-t}\int_0^{\infty}\frac{\ln^{j-1}(x+1)}{x+1}t^xdx\\
(\text{by\ lemme\ \ref{oonb}})&=\frac{1}{1-t}\left(\sum_{m=0}^{\infty}\frac{\ln^{j-1}(m+1)}{m+1}t^m+O(1)\right)\\
(\text{by\ lemme\ \ref{oijx}})&=\frac{1}{1-t}O\Big(\ln^j\Big(\frac1{1-t}\Big)\Big)
\end{align*}
as desired.
\end{proof}

\begin{thm}\label{mnzs}
The map $\rho_{1,1}^{\mu}$ is characterized by
\begin{align*}
\rho_{1,1}^{\mu}:\R[T]&\longrightarrow\R[T],\\
T^n&\longmapsto n!\sum_{k=0}^n\frac{\gamma_k}{(n-k)!}\cdot T^{n-k},
\end{align*}
where
$$
\sum_{k=0}^{\infty}\gamma_ku^k=e^{\gamma u}\Gamma(1+u)
=\exp\left(\sum_{n=2}^{\infty}\frac{(-1)^n}{n}\zeta(n)u^n\right).
$$
\end{thm}

\begin{rem}
We point out that the map $\rho_{1,1}^{\mu}$ is exactly the same one used to compare the two
regularization schemes (i.e., stuffle and shuffle) of MZVs discovered by Ihara, Kaneko and Zagier in \cite{ikz}.
In particular, it is independent of $\mu$.
\end{rem}

\begin{proof}
Define
$$Q(T):=\frac{d^n}{du^n}\left(e^{\gamma u}\Gamma(1+u)\cdot e^{Tu}\right)\bigg|_{u=0}=\frac{d^n}{du^n}\left(\Gamma(1+u) e^{(T+\gamma)u}\right)\bigg|_{u=0},$$
and
$$
T=\int_{\Delta^1(t)}\omega^{\mu}_{(1;1)}=\int_0^t\frac{\mu}{1-\tau^{\mu}}d\tau.
$$
Then
\begin{align*}
&\frac{1}{1-t}Q\left(\int_0^t\frac{\mu}{1-\tau^{\mu}}d\tau-\gamma-c_0(\mu)\right)\\
&=\frac{1}{1-t}Q\left(\ln\left(\frac{1}{1-t}\right)-\gamma+\psi_{\mu}(t)\right)\qquad(\text{see\ lemma\ \ref{zbxd}})\\
&=\frac{d^n}{du^n}\left(\frac{\Gamma(1+u)}{(1-t)^{1+u}}\cdot e^{\psi_{\mu}(t)u}\right)\bigg|_{u=0}\\
&=\sum_{k=0}^n\binom{n}{k}\left(\frac{d^{n-k}}{du^{n-k}}\frac{\Gamma(1+u)}{(1-t)^{1+u}}\right)\cdot(\psi_{\mu}(t))^ke^{\psi_{\mu}(t)u}\bigg|_{u=0}\\
&=\sum_{k=0}^n\binom{n}{k}\left(\frac{d^{n-k}}{du^{n-k}}\sum_{m=0}^{\infty}\frac{\Gamma(m+1+u)}{\Gamma(m+1)}t^m\right)\cdot(\psi_{\mu}(t))^ke^{\psi_{\mu}(t)u}\bigg|_{u=0}\\
&=\sum_{k=0}^n\binom{n}{k}\left(\sum_{m=0}^{\infty}\frac{\Gamma^{(n-k)}(m+1+u)}{\Gamma(m+1)}t^m\right)\cdot(\psi_{\mu}(t))^ke^{\psi_{\mu}(t)u}\bigg|_{u=0}\\
&=\sum_{k=0}^n\binom{n}{k}\left(\sum_{m=0}^{\infty}\frac{\Gamma^{(n-k)}(m+1)}{\Gamma(m+1)}t^m\right)\cdot(\psi_{\mu}(t))^k\\
&=\sum_{m=0}^{\infty}\frac{\Gamma^{(n)}(m+1)}{\Gamma(m+1)}t^m+\sum_{k=1}^n\binom{n}{k}\left(\sum_{m=0}^{\infty}\frac{\Gamma^{(n-k)}(m+1)}{\Gamma(m+1)}t^m\right)\cdot(\psi_{\mu}(t))^k\\
(\text{by\ Lemma\ \ref{sfiz}})&=\sum_{m=0}^{\infty}\frac{\Gamma^{(n)}(m+1)}{\Gamma(m+1)}t^m+\sum_{k=1}^n\binom{n}{k}
 \left(\frac{1}{1-t}O\Big(\ln^{n-k}\Big(\frac1{1-t}\Big)\Big) \right)\cdot(\psi_{\mu}(t))^k\\
(\text{by\ Lemma\ \ref{zbxd}})&=\sum_{m=0}^{\infty}\frac{\Gamma^{(n)}(m+1)}{\Gamma(m+1)}t^m+O\Big(\ln^{n-1}\Big(\frac1{1-t}\Big)\Big)\\
&=\sum_{m=0}^{\infty}\ln^n(m+1)t^m+\sum_{m=0}^{\infty}O\left(\frac{\ln^{n-1}(m+1)}{m+1}\right)t^m+O\Big(\ln^{n-1}\Big(\frac1{1-t}\Big)\Big)\\
(\text{by\ Lemma\ \ref{oonb}\ and\ \ref{oijx}})&=\int_0^{\infty}\ln^n(x+1)t^x\,dx+O(1)+O\Big(\ln^{n}\Big(\frac1{1-t}\Big)\Big)\\
&=\int_0^{\infty}\ln^n(x+1)t^x\, dx+O\Big(\ln^{n}\Big(\frac1{1-t}\Big)\Big).
\end{align*}
On the other hand, we have
\begin{align*}
&\zeta^{\mu}_{\shuffle,t}\big(\varphi_{\mu}\big(\overbrace{(y_{1,1}-\gamma-c_0(\mu))\must \cdots\must (y_{1,1}-\gamma-c_0(\mu))}^n\big)\big)\\
&=\zeta^{\mu,T}_{\shuffle,1}\big(\varphi_{\mu}\big(\overbrace{(y_{1,1}-\gamma-c_0(\mu))\must \cdots\must (y_{1,1}-\gamma-c_0(\mu))}^n\big)\big)+O\Big((1-t)\ln^{n}\Big(\frac1{1-t}\Big)\Big)\\
&=\rho_{1,1}^{\mu}\big(\zeta^{\mu,T}_{*,1}\big(\overbrace{(y_{1,1}-\gamma-c_0(\mu))\must \cdots\must (y_{1,1}-\gamma-c_0(\mu))}^n\big)\big)+O\Big((1-t)\ln^{n}\Big(\frac1{1-t}\Big)\Big)\\
&=\rho_{1,1}^{\mu}\big((T-\gamma-c_0(\mu))^n\big)+O\Big((1-t)\ln^{n}\Big(\frac1{1-t}\Big)\Big),
\end{align*}
where
$$
T=\int_0^t\frac{\mu}{1-\tau^{\mu}}d\tau.
$$
Notice that
\begin{align*}
&\zeta^{\mu}_{\shuffle,t}\big(\varphi_{\mu}\big(\overbrace{(y_{1,1}-\gamma-c_0(\mu))\must \cdots\must (y_{1,1}-\gamma-c_0(\mu))}^n\big)\big)\\
&=(1-t^{\mu})\sum_{m=0}^{\infty}\zeta^{\mu}_{*,m}\big(\overbrace{(y_{1,1}-\gamma-c_0(\mu))\must \cdots\must (y_{1,1}-\gamma-c_0(\mu))}^n\big)t^{m\mu+1}\\
&=(1-t^{\mu})\sum_{m=0}^{\infty}\big(\zeta^{\mu}_{*,m}(y_{1,1})-\gamma-c_0(\mu)\big)^nt^{m\mu+1}\\
&=(1-t^{\mu})\sum_{m=0}^{\infty}\left(\sum_{k=0}^m\frac{\mu}{k\mu+1}-\gamma-c_0(\mu)\right)^nt^{m\mu+1}\\
&=(1-t^{\mu})\sum_{m=0}^{\infty}\left(H\left(m+\frac{1}{\mu}\right)-H\left(\frac{1}{\mu}-1\right)-\gamma-c_0(\mu)\right)^nt^{m\mu+1}\\
&=(1-t^{\mu})\int_0^{\infty}\left(H\left(x+\frac{1}{\mu}\right)-H\left(\frac{1}{\mu}-1\right)-\gamma-c_0(\mu)\right)^nt^{x\mu+1}dx\\
&\quad+(1-t^{\mu})\frac{f(0)}{2}+i(1-t^{\mu})\int_0^{\infty}\frac{f(ix)-f(-ix)}{e^{2\pi x}-1}dx\qquad(\text{by\ theorem\ \ref{loaz}})\\
&=\frac{1-t^{\mu}}{\mu}\int_0^{\infty}\left(H\left(\frac{x+1}{\mu}\right)-H\left(\frac{1}{\mu}-1\right)-\gamma-c_0(\mu)\right)^nt^{x+1}dx\\
&\quad+(1-t^{\mu})\frac{f(0)}{2}+i(1-t^{\mu})\int_0^{\infty}\frac{f(ix)-f(-ix)}{e^{2\pi x}-1}dx\\
&=\frac{1-t^{\mu}}{\mu}\int_0^{\infty}\left(\ln\left(\frac{x+1}{\mu}\right)+\gamma+O\left(\frac{1}{x+1}\right)-H\left(\frac{1}{\mu}-1\right)-\gamma-c_0(\mu)\right)^nt^{x+1}dx\\
&\quad+(1-t^{\mu})\frac{f(0)}{2}+i(1-t^{\mu})\int_0^{\infty}\frac{f(ix)-f(-ix)}{e^{2\pi x}-1}dx\\
&=\frac{1-t^{\mu}}{\mu}\int_0^{\infty}\left(\ln\left(x+1\right)+O\left(\frac{1}{x+1}\right)\right)^nt^{x+1}dx+O(1)\\
&=\frac{1-t^{\mu}}{\mu}\int_0^{\infty}\ln^n(x+1)t^{x+1}dx+\frac{1-t^{\mu}}{\mu}\int_0^{\infty}O\left(\frac{\ln^{n-1}(x+1)}{x+1}\right)t^{x+1}dx+O(1)\\
&=(1-t)\int_0^{\infty}\ln^n(x+1)t^x dx+O\Big((1-t)\ln^{n}\Big(\frac1{1-t}\Big)\Big)\qquad(\text{by\ lemma\ \ref{oijx}\ and\ \ref{oonb}})
\end{align*}
where
$$
f(x):=\left(H\left(x+\frac{1}{\mu}\right)-H\left(\frac{1}{\mu}-1\right)-\gamma-c_0(\mu)\right)^nt^{x\mu+1}.
$$
We conclude that
\begin{align*}
Q(T-\gamma-c_0(\mu))&=(1-t)\int_0^{\infty}\ln^n(x+1)t^xdx+O\Big((1-t)\ln^{n}\Big(\frac1{1-t}\Big)\Big)\\
&=\rho_{1,1}^{\mu}((T-\gamma-c_0(\mu))^n)+O\Big((1-t)\ln^{n}\Big(\frac1{1-t}\Big)\Big) \quad \text{as} \quad t\to 1^-,
\end{align*}
where $T=\int_0^t\frac{\mu}{1-\tau^{\mu}}d\tau$. Hence
$$\rho_{1,1}^{\mu}(T^n)=Q(T)=n!\sum_{k=0}^n\frac{\gamma_k}{(n-k)!}T^{n-k}.$$
This completes the proof of Theorem~\ref{mnzs}.
\end{proof}

\begin{table}[!h]
{
\begin{center}
\begin{tabular}{  |c|c|c|c|c|c|c } \hline
     \   $k$\  &\  \ 0 \ \ & \ \  1\ \   &  2  &  3  &  4  &  $\cdots$  \\ \hline
$\gamma_k$ & 1 &  0  &  $\frac{\zeta(2)}{2}$  &  $-\frac{\zeta(3)}{3}$  &  $\frac{9\zeta(4)}{16}$  &   $\cdots$ \\ \hline
\end{tabular}
\end{center}
} \caption{Some specific values for $\gamma_k$.}
\label{Table:dimMMV}
\end{table}
Hence we have
\begin{align*}
\rho_{1,1}^{\mu}(1)=& 1,\quad  \rho_{1,1}^{\mu}(T)=T,\quad \rho_{1,1}^{\mu}(T^2)= T^2+\zeta(2),\\
\rho_{1,1}^{\mu}(T^3)=& T^3+3\zeta(2)T-2\zeta(3),\quad
\rho_{1,1}^{\mu}(T^4)=T^4+6\zeta(2)T^2-8\zeta(3)T+\frac{27}{2}\zeta(4).
\end{align*}

\begin{cor}\label{kxol}
The map $\rho_{m,m'}^{\mu}$ is characterized by
\begin{align*}
\rho_{m,m'}^{\mu}:\R[T]&\longrightarrow\R[T]\\
T^n&\longmapsto\sum_{k=0}^n\binom{n}{k}c_{m,1}^{n-k}(\mu)\left(k!\sum_{l=0}^k\frac{\gamma_l}{(k-l)!}(T+c_{1,m'}(\mu))^{k-l}\right),
\end{align*}
where $c_{m,m'}(\mu)=H(m'/\mu-1)-H(m/\mu-1)$.
\end{cor}

\begin{proof} Notice that
$$\rho_{m,m'}^{\mu}=\myrho^{\mu,\shuffle}_{1,m'}\circ\rho_{1,1}^{\mu}\circ\myrho^{\mu,*}_{m,1}.$$
Hence,
\begin{align*}
\rho_{m,m'}^{\mu}(T^n)
&=(\myrho^{\mu,\shuffle}_{1,m'}\circ\rho_{1,1}^{\mu}\circ\myrho^{\mu,*}_{m,1})(T^n)\\
&=(\myrho^{\mu,\shuffle_2}_{1,m'}\circ\rho_{1,1}^{\mu})((T+c_{m,1}(\mu))^n)\\
&=(\myrho^{\mu,\shuffle_2}_{1,m'}\circ\rho_{1,1}^{\mu})\left(\sum_{k=0}^n\binom{n}{k}c_{m,1}^{n-k}(\mu)T^k\right)\\
&=\myrho^{\mu,\shuffle_2}_{1,m'}\left(\sum_{k=0}^n\binom{n}{k}c_{m,1}^{n-k}(\mu)\left(k!\sum_{l=0}^k\frac{\gamma_l}{(k-l)!}T^{k-l}\right)\right)\\
&=\sum_{k=0}^n\binom{n}{k}c_{m,1}^{n-k}(\mu)\left(k!\sum_{l=0}^k\frac{\gamma_l}{(k-l)!}(T+c_{1,m'}(\mu))^{k-l}\right).
\end{align*}
\end{proof}

We are now ready to give some applications of the map $\rho_{m,m'}^{\mu}$.

\begin{thm}
We have
\begin{equation*}
\zeta(k)\in\bigcap_{\mu>0}\calZ^{\mu}
\end{equation*}
\end{thm}

\begin{proof} From Theorem~\ref{mnzs}, we know that
$$\gamma_k\in\bigcap_{\mu>0}\calZ^{\mu}\Longrightarrow\zeta(k)\in\bigcap_{\mu>0}\calZ^{\mu}$$
for all $k\geqslant 2$.
\end{proof}

\textbf{Question}. Provide an accurate description of $\bigcap\limits_{\mu>0}\calZ^{\mu}$.
It is equivalent to find out all identities about $\mu$.

\begin{eg}
In this example, we consider the word $y_{1,1}\must y_{1,1}\must y_{1,1}$. Notice that
\begin{align*}
y_{1,1}\must y_{1,1}\must y_{1,1}&=(2y_{1,1}^2-\mu y_{2,1})\must y_{1,1}\\
&=2(y_{1,1}^3+(y_{1,1}\must y_{1,1})y_{1,1}-\mu y_{1,1}y_{2,1})-\mu(y_{1,1}y_{2,1}+y_{2,1}y_{1,1}-\mu y_{3,1})\\
&=2(3y_{1,1}^3-\mu y_{2,1}y_{1,1}-\mu y_{1,1}y_{2,1})-\mu(y_{1,1}y_{2,1}+y_{2,1}y_{1,1}-\mu y_{3,1})\\
&=6y_{1,1}^3-3\mu y_{2,1}y_{1,1}-3\mu y_{1,1}y_{2,1}+\mu^2y_{3,1}.
\end{align*}
Hence, we have
\begin{align*}
\varphi_{\mu}(y_{1,1}\must y_{1,1}\must y_{1,1})&=6\varphi_{\mu}(y_{1,1}^3)-3\mu\varphi_{\mu}(y_{2,1}y_{1,1})-3\mu\varphi_{\mu}(y_{1,1}y_{2,1})+\mu^2\varphi_{\mu}(y_{3,1})\\
&=6y_1^{\mu}y^{\mu}_0y^{\mu}_0-3\mu y_1^{\mu}xy_0^{\mu}-3\mu y_1^{\mu}y_0^{\mu}x+\mu^2y_1^{\mu}x^2\\
&=6(y_1^{\mu}y^{\mu}_0(y^{\mu}_0-y^{\mu}_1)+y^{\mu}_1(y^{\mu}_0-y^{\mu}_1)\shuffle y^{\mu}_1-2y^{\mu}_1y^{\mu}_1(y^{\mu}_0-y^{\mu}_1)+y^{\mu}_1\shuffle y^{\mu}_1\shuffle y^{\mu}_1/6)\\
&\quad-3\mu(y^{\mu}_1x(y^{\mu}_0-y^{\mu}_1)+y^{\mu}_1x\shuffle y^{\mu}_1-2y^{\mu}_1y^{\mu}_1x)-3\mu y_1^{\mu}y_0^{\mu}x+\mu^2y_1^{\mu}x^2.
\end{align*}
We see that
\begin{align*}
&\zeta_{\shuffle,1}^{\mu,T}(\varphi_{\mu}(y_{1,1}\must y_{1,1}\must y_{1,1}))\\
&=T^3+(6\calD^{\mu}(1,1;1,1;1,2)-3\mu\zeta^{\mu}(2;1))T\\
&\quad+6\calD^{\mu}(1,1,1;1,1,1;1,1,2)-12\calD^{\mu}(1,1,1;1,2,2;1,2,3)-3\mu \calD^{\mu}(2,1;1,1;1,2)\\
&\quad+6\mu\zeta^{\mu}(1,2;1,2)-3\mu\zeta^{\mu}(1,2;1,1)+\mu^2\zeta^{\mu}(3;1).
\end{align*}
Since
\begin{align*}
\rho_{1,1}^{\mu}(\zeta^{\mu,T}_{*,1}(y_{1,1}\must y_{1,1}\must y_{1,1}))=\rho_{1,1}^{\mu}(T^3)=T^3+3\zeta(2)T-2\zeta(3)
\end{align*}
we get
\begin{align}
3\zeta(2)&=6\calD^{\mu}(1,1;1,1;1,2)-3\mu\zeta^{\mu}(2;1) \label{equ:zeta2}\\
-2\zeta(3)&=6\calD^{\mu}(1,1,1;1,1,1;1,1,2)-12\calD^{\mu}(1,1,1;1,2,2;1,2,3)-3\mu\calD^{\mu}(2,1;1,1;1,2)\notag\\
&\quad+6\mu\zeta^{\mu}(1,2;1,2) -3\mu\zeta^{\mu}(1,2;1,1)+\mu^2\zeta^{\mu}(3;1). \notag
\end{align}
\end{eg}

\begin{eg} \label{eg:ClauCata}
We have the following identities
\begin{align*}
\sum_{0\leqslant n\leqslant m}\frac{1}{(2n+1)(2m+1)(2m+2)}&=\frac{\pi^2}{12},\\
\sum_{0\leqslant n\leqslant m}\frac{1}{(3n+1)(3m+1)(3m+2)}&=\frac{5\pi^2}{108}+\frac{\sqrt{3}}{6}\mathrm{Cl}_2\left(\frac{2\pi}{3}\right),\\
\sum_{0\leqslant n\leqslant m}\frac{1}{(3n+2)(3m+2)(3m+4)}&=\frac{5\pi^2}{216}-\frac{\sqrt{3}}{12}\mathrm{Cl}_2\left(\frac{2\pi}{3}\right),\\
\sum_{0\leqslant n\leqslant m}\frac{1}{(4n+1)(4m+1)(4m+2)}&=\frac{7\pi^2}{192}+\frac{G}{4},\\
\sum_{0\leqslant n\leqslant m}\frac{1}{(4n+3)(4m+3)(4m+6)}&=\frac{7\pi^2}{576}-\frac{G}{12},
\end{align*}
where $\text{Cl}_2(x)=\sum\limits_{n=1}^{\infty}\frac{\sin(nx)}{n^2}$
is the Clausen function, and $G=\sum\limits_{n=0}^{\infty}\frac{(-1)^n}{(2n+1)^2}$ is
Catalan's constant.
\end{eg}

\begin{proof} From \eqref{equ:zeta2}, we have
\begin{align*}
\sum_{0\leqslant n\leqslant m}\frac{2\mu^2}{(n\mu+1)(m\mu+1)(m\mu+2)}
=\zeta(2)+\sum_{n=0}^{\infty}\frac{\mu^2}{(n\mu+1)^2}
=\zeta(2)+\psi_1\left(\frac{1}{\mu}\right)
\end{align*}
where
$$\psi_1(x):=\sum_{n=0}^{\infty}\frac{1}{(n+x)^2}$$
is called the \emph{polygamma function}. With Mathematica, we can find the
following special values:
\begin{align*}
\psi_1\left(1\right)&=\frac{\pi^2}{6},\quad
\psi_1\left(\frac{1}{3}\right)=\frac{2\pi^2}{3}+3\sqrt{3}\text{Cl}_2\left(\frac{2\pi}{3}\right),\quad
\psi_1\left(\frac{1}{4}\right)=\pi^2+8G,\\
\psi_1\left(\frac{1}{2}\right)&=\frac{\pi^2}{2},\quad
\psi_1\left(\frac{2}{3}\right)=\frac{2\pi^2}{3}-3\sqrt{3}\text{Cl}_2\left(\frac{2\pi}{3}\right),\quad
\psi_1\left(\frac{3}{4}\right)=\pi^2-8G.
\end{align*}
This proves the identities in Example~\ref{eg:ClauCata}.
\end{proof}

We remark that the identities in Example \ref{eg:ClauCata} can also be derived by using the theory of colored MZVs
(see \cite[section 2]{YuanZhao2015b}) by noticing, e.g., the partial fraction decomposition
\begin{equation*}
\frac{1}{(3m+1)(3m+2)}=\frac{1}{3m+1}-\frac{1}{3m+2}.
\end{equation*}

\begin{eg} Consider the word $y_{1,1}y_{1,2}^2$. We have
\begin{align*}
y_{1,1}y_{1,2}^2=(y_{1,1}-y_{1,2})y_{1,2}^2+y_{1,2}^3.
\end{align*}
A direct calculation yields that
\begin{align*}
y_{1,2}^3=\frac{y_{1,2}\must y_{1,2}\must y_{1,2}}{6}+\frac{\mu}{2}y_{2,2}\must y_{1,2}+\frac{\mu^2}{3}y_{3,2},
\end{align*}
and
\begin{align*}
(y_{1,1}-y_{1,2})y_{1,2}^2&=\frac{(y_{1,1}-y_{1,2})\must y_{1,2}\must y_{1,2}}{2}+(-y_{1,2}(y_{1,1}-y_{1,2})+\mu(y_{1,1}-y_{1,2})-\mu y_{2,2})\must y_{1,2}\\
&\quad+y_{1,2}^2(y_{1,1}-y_{1,2})-\mu y_{1,2}(y_{1,1}-y_{1,2})-\frac{\mu}{2}y_{2,2}(y_{1,1}-y_{1,2})+\frac{\mu}{2}(y_{1,2}y_{2,2}+y_{1,1}y_{2,2})\\
&\quad+\frac{\mu^2}{2}(y_{1,1}-y_{1,2})-\frac{\mu^2}{2}(y_{2,2}+y_{3,2}).
\end{align*}
Thus
\begin{align*}
y_{1,1}y_{1,2}^2&=\frac{y_{1,2}\must y_{1,2}\must y_{1,2}}{3!}+\frac{w_2}{2}\must y_{1,2}\must y_{1,2}+w_1\must y_{1,2}+w_0
\end{align*}
where
\begin{align*}
w_2&=y_{1,1}-y_{1,2},\\
w_1&=-\frac{\mu}{2}y_{2,2}-y_{1,2}(y_{1,1}-y_{1,2})+\mu(y_{1,1}-y_{1,2}),\\
w_0&=y_{1,2}^2(y_{1,1}-y_{1,2})-\mu y_{1,2}(y_{1,1}-y_{1,2})-\frac{\mu}{2}y_{2,2}(y_{1,1}-y_{1,2})+\frac{\mu}{2}(y_{1,2}y_{2,2}+y_{1,1}y_{2,2})\\
&\quad+\frac{\mu^2}{2}(y_{1,1}-y_{1,2})-\frac{\mu^2}{2}y_{2,2}-\frac{\mu^2}{6}y_{3,2}.
\end{align*}
Hence,
\begin{align*}
\zeta^{\mu,T}_{*,2}(y_{1,1}y_{1,2}^2)=\frac{T^3}{3!}+\frac{\zeta^{\mu}_*(w_2)}{2}T^2+\zeta^{\mu}_*(w_1)T+\zeta^{\mu}_*(w_0).
\end{align*}
We have
\begin{align*}
\varphi_{\mu}(y_{1,1}y_{1,2}y_{1,2})&=y_1^{\mu}y_1^{\mu}y_0^{\mu}\\
&=y_1^{\mu}y_1^{\mu}(y_0^{\mu}-y_1^{\mu})+y_1^{\mu}y_1^{\mu}y_1^{\mu}\\
&=\frac{y_1^{\mu}\shuffle y_1^{\mu}\shuffle y_1^{\mu}}{3!}+y_1^{\mu}y_1^{\mu}(y_0^{\mu}-y_1^{\mu}).
\end{align*}
Thus
$$\zeta^{\mu,T}_{\shuffle,1}(y_1^{\mu}y_1^{\mu}y_0^{\mu})=\frac{T^3}{3!}+
\zeta^{\mu}_{\shuffle}(y_1^{\mu}y_1^{\mu}(y_0^{\mu}-y_1^{\mu})).
$$
By Corollary \ref{kxol}, we have
\begin{align*}
\rho_{2,1}^{\mu}:\R[T]&\longrightarrow\R[T],\\
1&\longmapsto1,\\
T&\longmapsto T+c_{2,1}(\mu),\\
T^2&\longmapsto T^2+2c_{2,1}(\mu)T+c_{2,1}^2(\mu)+\zeta(2),\\
T^3&\longmapsto T^3+3c_{2,1}(\mu)T^2+3(\zeta(2)+c_{2,1}^2(\mu))T+c_{2,1}^3(\mu)+3\zeta(2)c_{2,1}(\mu)-2\zeta(3),\\
&\vdots
\end{align*}
By the identity
$$\rho_{2,1}^{\mu}(\zeta_{*,2}^{\mu,T}(y_{1,1}y_{1,2}y_{1,2}))=\zeta^{\mu,T}_{\shuffle,1}(\varphi_{\mu}(y_{1,1}y_{1,2}y_{1,2}))$$
we conclude that
\begin{align*}
\zeta^{\mu}_{\shuffle}(y_1^{\mu}y_1^{\mu}(y_0^{\mu}-y_1^{\mu})
)&=\frac{c_{2,1}^3(\mu)+3\zeta(2)c_{2,1}(\mu)-2\zeta(3)}{6}+\frac{\zeta_*^{\mu}(w_2)(c_{2,1}^2(\mu)+\zeta(2))}{2}\\
&\quad+\zeta_*^{\mu}(w_1)c_{2,1}(\mu)+\zeta_*^{\mu}(w_0).
\end{align*}

\end{eg}

\section{Sum and weighted sum formulas} \label{sec:sumFormula}
In this section, we will apply the regularized $\mu$-MHZVs to derive two sum formulas for $\mu$-double Hurwitz
zeta values which generalize the sum formulas for double zeta values and double zeta star values, respectively.
We also prove a weighted sum formula which generalizes the weighted sum formulas for both double zeta values and double $T$-values simultaneously. For convenience, we often do not distinguish the regularized $\mu$-MHZVs and their word representations.

Let $\gs$ and $\gt$ be two formal variables.
For $\sharp=\shuffle$ or $*$, define the generating functions
\begin{align*}
F^\mu_\sharp(m;\gs):=&\, \sum_{j=1}^\infty \zeta^\mu_\sharp(k;m) \gs^{k-1},\\
F^\mu_\sharp(m,n;\gs,\gt):=&\, \sum_{j=1}^\infty\sum_{k=1}^\infty \zeta^\mu_\sharp(j,k;m,n) \gs^{j-1}\gt^{k-1}.
\end{align*}

\begin{thm} \label{thm:DblGenFunc}
For any $m,n\in \Q_{>0}$, we have
\begin{equation}\label{equ:Fsh}
F^\mu_\shuffle(m;\gs)F^\mu_\shuffle(n;\gt)
=F^\mu_\shuffle(m,m+n;\gs,\gs+\gt) +F^\mu_\shuffle(n,m+n;\gt,\gs+\gt).
\end{equation}
If $m\ne n$ then
\begin{align*}
F^\mu_*(m;\gs)F^\mu_*(n;\gt)=&\, F^\mu_*(m,n;\gs,\gt)+F^\mu_*(n,m;\gt,\gs)+\mu({\rm II}+{\rm III}-{\rm I})
\end{align*}
where
\begin{align*}
{\rm I}=&\, \frac{\calD^\mu(1;m;n)}{n-m-\gt+\gs}, \
{\rm II} = \frac{F^\mu_*(n;\gt)-\zeta^\mu_*(1;n)}{(1+\gs)(n-m)-\gt} , \
{\rm III} = \frac{F^\mu_*(m;\gs)-\zeta^\mu_*(1;m)}{(1+\gt)(m-n)-\gs} . \\
\end{align*}
If $m=n$ then
\begin{equation}\label{equ:m=nFstWithDiffQuotient}
F^\mu_*(n;\gs)F^\mu_*(n;\gt)=F^\mu_*(n,n;\gs,\gt)+F^\mu_*(n,n;\gt,\gs)
-\mu \frac{F^\mu_*(n;\gs)-F^\mu_*(n;\gt)}{\gs-\gt}.
\end{equation}
\end{thm}

\begin{proof}
We have
\begin{align}
&\, F^\mu_\shuffle(m;\gs)F^\mu_\shuffle(n;\gt)
= \sum_{j=1}^\infty\sum_{k=1}^\infty \zeta^\mu_\shuffle(j;m)\zeta^\mu_\shuffle(k;n)\gs^{j-1}\gt^{k-1}
= \sum_{j=0}^\infty\sum_{k=0}^\infty \Big(y^\mu_m x^{j}\shuffle y^\mu_n x^{k}\Big) \gs^{j}\gt^{k} \notag\\
=&\, \sum_{j=0}^\infty\sum_{k=0}^\infty \left(\sum_{a=0}^{j} \binom{k+a}{a} y^\mu_m x^{j-a} y^\mu_n x^{k+a}
 +\sum_{a=0}^{k} \binom{j+a}{a} y^\mu_n x^{k-a} y^\mu_m x^{j+a} \right)\gs^{j}\gt^{k} \notag\\
=&\,\sum_{a=0}^{\infty}\sum_{j=0}^\infty \sum_{k=0}^\infty \binom{k+a}{a} y^\mu_m x^{j} y^\mu_n x^{k+a}\gs^{j+a}\gt^{k}
+\sum_{j=0}^\infty\sum_{a=0}^{\infty}\sum_{k=0}^\infty \binom{j+a}{a} y^\mu_n x^{k} y^\mu_m x^{j+a} \gs^{j}\gt^{k+a} \notag\\
=&\,\sum_{a=0}^{\infty}\sum_{j=0}^\infty \sum_{k=a}^\infty \binom{k}{a} y^\mu_m x^{j} y^\mu_n x^{k}\gs^{j+a}\gt^{k-a} +\sum_{a=0}^{\infty}\sum_{j=a}^\infty\sum_{k=0}^\infty \binom{j}{a} y^\mu_n x^{k} y^\mu_m x^{j}\gs^{j-a}\gt^{k+a} \notag\\
=&\,\sum_{j=0}^\infty \sum_{k=0}^\infty\sum_{a=0}^{k} \binom{k}{a} y^\mu_m x^{j} y^\mu_n x^{k}\gs^{j+a}\gt^{k-a}
+\sum_{j=0}^\infty\sum_{k=0}^\infty\sum_{a=0}^{j} \binom{j}{a} y^\mu_n x^{k} y^\mu_m x^{j}\gs^{j-a}\gt^{k+a} \notag\\
=&\,\sum_{j=0}^\infty \sum_{k=0}^\infty y^\mu_m x^{j} y^\mu_n x^{k}\gs^{j}(\gs+\gt)^{k}
 +\sum_{j=0}^\infty\sum_{k=0}^\infty y^\mu_n x^{k} y^\mu_m x^{j}(\gs+\gt)^{j}\gt^k \label{equ:FshUse}\\
=&\,F^\mu_\shuffle(m,m+n;\gs,\gs+\gt) +F^\mu_\shuffle(n,m+n;\gt,\gs+\gt).\notag
\end{align}
This proves \eqref{equ:Fsh}. On the other hand, for all $m,n\in\N$
\begin{equation}\label{equ:m=nFstStep1}
F^\mu_*(m;\gs)F^\mu_*(n;\gt)=\sum_{j=1}^\infty\sum_{k=1}^\infty
\left\{
\aligned
& \zeta^\mu_*(j,k;m,n)+\zeta^\mu_*(k,j;n,m) \\
&-\sum_{q=0}^\infty \frac{\mu^2}{(q\mu+m)^j (q\mu+n)^k}
\endaligned
\right\} \gs^{j-1}\gt^{k-1}.
\end{equation}
Suppose $m\ne n$. Then by Lemma~\ref{lem:parFrac},
\begin{align}\label{equ:FstUse}
F^\mu_*(m;\gs)F^\mu_*(n;\gt) =F^\mu_*(m,n;\gs,\gt)+F^\mu_*(n,m;\gt,\gs)-\mu({\rm I}+{\rm II}+{\rm III})
\end{align}
where
\begin{align}
{\rm I}=&\, \calD^\mu(1;m;n)\sum_{j=1}^\infty\sum_{k=1}^\infty
 \frac{(-1)^{j-1}(j+k-2)!}{(j-1)!(k-1)!(n-m)^{j+k-1}} \gs^{j-1}\gt^{k-1} \notag\\
=&\, \calD^\mu(1;m;n)\sum_{l=0}^\infty \frac{(\gt-\gs)^{l}}{(n-m)^{l+1}} \label{equ:IUse} \\
=&\, \frac{\calD^\mu(1;m;n)}{n-m} \cdot \frac{1}{1-\frac{\gt-\gs}{n-m}} \notag\\
=&\, \frac{\calD^\mu(1;m;n)}{n-m-\gt+\gs},\notag
\end{align}
and
\begin{align}
{\rm II}=&\, -\sum_{q=0}^\infty\sum_{j=1}^\infty\sum_{k=1}^\infty \sum_{i=0}^{k-2} \binom{i+j-1}{i} \frac{(-1)^{j}}{(n-m)^{i+j}}\frac{\mu}{(q\mu+n)^{k-i}} \gs^{j-1}\gt^{k-1} \notag\\
 =&\, \sum_{q=0}^\infty \sum_{k=1}^\infty \sum_{i=0}^{k-2} \frac{1}{(n-m)^{i+1}}\frac{\mu}{(q\mu+n)^{k-i}}\gt^{k-1} \sum_{j=1}^\infty \binom{i+j-1}{i} (-\gs)^{j-1} \notag\\
 =&\, \sum_{q=0}^\infty \sum_{i=0}^{\infty} \sum_{k=i+2}^\infty \frac{1}{(n-m)^{i+1}}\frac{\mu}{(q\mu+n)^{k-i}}\gt^{k-1} \sum_{j=0}^\infty \binom{i+j}{j} (-\gs)^{j}\notag\\
 =&\, \sum_{q=0}^\infty \sum_{i=1}^{\infty} \sum_{k=0}^\infty \frac{1}{(n-m)^{i}}\frac{\mu}{(q\mu+n)^{k+2}}\gt^{k+i} \sum_{j=0}^\infty \binom{-i}{j} \gs^{j} \label{equ:IIUse} \\
 =&\,\sum_{q=0}^\infty \sum_{i=1}^{\infty} \sum_{k=0}^\infty
 \frac{\gt^{i}}{(1+\gs)^{i}(n-m)^{i}}\frac{\mu}{(q\mu+n)^{k+2}}\gt^{k} \notag\\
 =&\, \frac{\gt}{(1+\gs)(n-m)-\gt} \sum_{k=0}^\infty \zeta^\mu(k+2;n) \gt^{k}. \notag
\end{align}
Similarly,
\begin{align}
{\rm III} =&\, \sum_{q=0}^\infty \sum_{i=1}^{\infty} \sum_{j=0}^\infty \sum_{k=0}^\infty \frac{1}{(m-n)^{i}}\frac{\mu}{(q\mu+m)^{j+2}}\gs^{j+i} \binom{-i}{k} \gt^{k} \label{equ:IIIUse} \\
=&\, \frac{\gs}{(1+\gt)(m-n)-\gs} \sum_{j=0}^\infty \zeta^\mu(j+2;m) \gs^{j}. \notag
\end{align}
Finally, if $m=n$ then by \eqref{equ:m=nFstStep1} we have
\begin{align}
F^\mu_*(n;\gs)F^\mu_*(n;\gt)=&\, F^\mu_*(n,n;\gs,\gt)+F^\mu_*(n,n;\gt,\gs)
-\mu \sum_{j,k\geqslant 1}\zeta^\mu(j+k;n) \gs^{j-1}\gt^{k-1}\notag \\
=&\, F^\mu_*(n,n;\gs,\gt)+F^\mu_*(n,n;\gt,\gs)
-\mu \sum_{l\geqslant 2}\zeta^\mu(l;n)\sum_{j+k=l,j,k\geqslant 1} \gs^{j-1}\gt^{k-1} \label{equ:m=nFstUse}\\
=&\, F^\mu_*(n,n;\gs,\gt)+F^\mu_*(n,n;\gt,\gs)
-\mu \sum_{l\geqslant 2}\zeta^\mu(l;n)\frac{\gs^{l-1}-\gt^{l-1}}{\gs-\gt} \notag
\end{align}
and \eqref{equ:m=nFstWithDiffQuotient} follows at once.
\end{proof}

\begin{cor} \label{cor:DblGenFunc}
For any $m,n\in \Q_{>0}$ and $N\in\N_{\geqslant 3}$, we have
\begin{multline}\label{equ:shCoeff}
{\rm Coeff}_{t^{N-2}}\Big(F^\mu_\shuffle(m;t\gs)F^\mu_\shuffle(n;t\gt)\Big)\\
=\sum_{j+k=N,j,k\geqslant 1} \Big(\zeta^\mu_\shuffle(j,k;m,m+n)\gs^{j-1}+\zeta^\mu_\shuffle(j,k;n,m+n)\gt^{j-1}\Big)(\gs+\gt)^{k-1}.
\end{multline}
If $m\ne n$ then
\begin{multline}\label{equ:stCoeff}
{\rm Coeff}_{t^{N-2}}\Big(F^\mu_*(m;\gs)F^\mu_*(n;\gt)\Big)
= \mu({\rm II}+{\rm III}-{\rm I})\\
+\sum_{j+k=N,j,k\geqslant 1}\Big(\zeta^\mu_*(j,k;m,n)\gs^{j-1}\gt^{k-1}+\zeta^\mu_*(j,k;n,m)\gt^{j-1}\gs^{k-1}\Big)
\end{multline}
where
\begin{align}
{\rm I}=&\, \calD^\mu(1;m;n) \frac{(\gt-\gs)^{N-2}}{(n-m)^{N-1}}, \label{equ:IUse1} \\
{\rm II} =&\, \sum_{q=0}^\infty \sum_{i+j+k=N-2,i,j,k\geqslant 0} \frac{1}{(n-m)^{i+}}\frac{\mu}{(q\mu+n)^{k+2}} \binom{-i}{j} \gt^{k+i} \gs^{j} , \label{equ:IIUse1} \\
{\rm III} =&\, \sum_{q=0}^\infty \sum_{i+j+k=N-2,i,j,k\geqslant 0}
\frac{1}{(m-n)^{i}}\frac{\mu}{(q\mu+m)^{j+2}} \binom{-i}{k}\gs^{j+i} \gt^{k}. \label{equ:IIIUse1}
\end{align}
If $m=n$ then
\begin{multline} \label{equ:m=nFstUse1}
{\rm Coeff}_{t^{N-2}}\Big(F^\mu_*(m;\gs)F^\mu_*(n;\gt)\Big)
= -\mu \zeta^\mu(N;n)\sum_{j+k=N,j,k\geqslant 1} \gs^{j-1}\gt^{k-1} \\
+\sum_{j+k=N,j,k\geqslant 1}\zeta^\mu_*(j,k;n,n)\Big(\gs^{j-1}\gt^{k-1}+\gt^{j-1}\gs^{k-1}\Big).
\end{multline}
\end{cor}

\begin{proof}
The corollary follows from the proof of Theorem~\ref{thm:DblGenFunc} after we apply the change of variables
$\gs\to t\gs$, $\gt\to t\gt$, and extract the coefficient of $t^{N-2}$ in \eqref{equ:FshUse}, \eqref{equ:FstUse}
\eqref{equ:IUse}--\eqref{equ:IIIUse}, and \eqref{equ:m=nFstUse}.
\end{proof}

\begin{thm} \emph{(Sum Formula 1)} \label{thm:m=nSumFormula}
For any $m,n\in\Q_{>0}$, $m\ne n$, and $N\in \N_{\geqslant 3}$, we have
\begin{multline}\label{equ:sumFormmneqn}
 \sum_{j+k=N,k\geqslant 2} \zeta^\mu(j,k;n,m+n)=
\zeta^\mu(1,N';m,n)-\zeta^\mu(1,N';m,m+n)-\mu\frac{\calD^\mu(1;m;n)}{(n-m)^{N-1}} \\
+\calD^\mu(N',1;n,m;n,m+n)
+\mu \sum_{\substack{i+k=N-2\\ i\geqslant 1,k\geqslant 0}} \frac{\zeta^\mu(k+2;n)}{(n-m)^{i}} .
\end{multline}
where $N'=N-1$. Moreover,
\begin{multline}\label{equ:sumFormm=n}
\sum_{j+k=N,k\geqslant 2} \zeta^\mu(j,k;n,2n)=
\zeta^\mu(1,N';n,n)-\zeta^\mu(1,N';n,2n) -\mu\zeta^\mu(N;n)\\
+\calD^\mu(N',1;n,n;n,2n).
\end{multline}
\end{thm}

\begin{proof}
Taking $\gs=0$ and $\gt=1$ in Corollary~\ref{cor:DblGenFunc}, we see that
 II contributes only when $j=0$ while III does not contribute at all. Thus,
\begin{align*}
&{\rm Coeff}_{t^{N-2}}\Big(F^\mu_\shuffle(m;t\gs)F^\mu_\shuffle(n;t\gt)\Big)\Big|_{\gs=0,\gt=1}
=\zeta^\mu_\shuffle(1,N';m,m+n)+ \sum_{j+k=N} \zeta^\mu_\shuffle(j,k;n,m+n),\\
&{\rm Coeff}_{t^{N-2}}\Big(F^\mu_*(m;t\gs)F^\mu_*(n;t\gt)\Big)\Big|_{\gs=0,\gt=1} = \zeta^\mu_*(1,N';m,n)+\zeta^\mu_*(N',1;n,m) \\
& \hskip2cm -\frac{\calD^\mu(1;m;n)}{(n-m)^{N'}}
+\mu\sum_{q=0}^\infty \sum_{\substack{i+k=N-2\\ i\geqslant 1,k\geqslant 0}} \frac{1}{(n-m)^{i}}\frac{\mu}{(q\mu+n)^{k+2}},
\end{align*}
Note that $\zeta^\mu_*(N',1;m,n)=\rho^\mu_{1,1}\zeta^\mu_*(N',1;m,n)=\zeta^\mu_\shuffle(N',1;m,n)$
since $\rho^\mu_{1,1}(T)=T$.
Hence, using the comparison map $\rho^\mu_{1,1}$ we get \eqref{equ:sumFormmneqn} easily.

Further, taking $m=n$, $\gs=0$ and $\gt=1$ in Corollary~\ref{cor:DblGenFunc} and
applying the comparison map $\rho^\mu_{1,1}$, we have
\begin{align*}
\zeta^\mu_\shuffle(1,N';n,2n)+ \sum_{j+k=N} \zeta^\mu_\shuffle(j,k;n,2n)
=&\, \zeta^\mu_*(1,N';n,n)+\zeta^\mu_*(N',1;n,n)-\mu \zeta^\mu(N;n)
\end{align*}
which readily yields \eqref{equ:sumFormm=n}.
\end{proof}

\begin{rem} When $\mu=n=1$ \eqref{equ:sumFormm=n} reduces to the classical
sum formula for double zeta values:
\begin{align*}
 \sum_{j+k=N,k\geqslant 2} \zeta(j,k)=\zeta(N)
\end{align*}
since for all $j+k=N$ ($j,k>0$) we always have
\begin{equation*}
 \calD^{\mu=1}(j,k;1,1;1,2)
 =\sum_{0<n_1\leqslant n_2} \frac{1}{n_1^j}\left(\frac{1}{n_2^k}-\frac{1}{(n_2+1)^k} \right)=\zeta(N).
\end{equation*}
Further, when $\mu=2$ and $n=1$ \eqref{equ:sumFormm=n} reduces to the
sum formula for double $T$-values first proved as \cite[Corollary~4.3]{BCJXXZhao2020c}.
\end{rem}

\begin{thm} \emph{(Sum Formula 2)} \label{thm:m=0SumFormula}
For any $n\in\Q_{>0}$ and $N\in \N_{\geqslant 3}$, we have
\begin{align*}
 \sum_{j+k=N,k\geqslant 2} \zeta^\mu(j,k;n,n)=&\, \mu (N-1)\zeta^\mu(N;n)+\calD^\mu(N-1,1;n,\mu;n,n).
\end{align*}
\end{thm}

\begin{proof}
The key idea is to treat $m$ as a variable and let $m\to 0^+$ in Corollary~\ref{thm:m=nSumFormula}.
Set $N'=N-1$. Observe that
\begin{align*}
&\lim_{m\to 0^+}\zeta^\mu(1,N';m,n)-\zeta^\mu(1,N';m,m+n)\\
& \hskip1cm =\lim_{m\to 0^+}\sum_{1\leqslant q_1 \leqslant q_2} \left(\frac{\mu^2}{(q_1\mu+m)(q_1\mu+n)^{N'}}-\frac{\mu^2}{(q_1\mu+m)(q_1\mu+m+n)^{N'}}\right) \\
& \hskip1cm + \lim_{m\to 0^+} \frac{\mu}m\sum_{0\leqslant q} \left( \frac{\mu}{(q\mu+n)^{N'}}-\frac{\mu}{(q\mu+m+n)^{N'}} \right)
= \mu N' \zeta^\mu(N;n),\\
& \mu\frac{\calD^\mu(1;m;n)}{(n-m)^{N'}}=\frac{\mu^2}{m(n-m)^{N'}}-\frac{\mu^2}{n(n-m)^{N'}}
 +\mu\frac{\calD^\mu(1;m+\mu;n+\mu)}{(n-m)^{N'}}, \\
&\, \zeta^\mu_M(N',1;n,m) = \left(\sum_{0\leqslant q_1< q_2 \leqslant M}+\sum_{0\leqslant q_1=q_2 \leqslant M}\right) \frac{\mu^2}{(q_1\mu+n)^{N'}(q_2\mu+m)}\\
&\,\hskip1cm =\frac{\mu^2}{n^{N'}m}+\sum_{0\leqslant q_1\leqslant q_2 \leqslant M-1}\frac{\mu^2}{(q_1\mu+n)^{N'}(q_2\mu+\mu+m)} + \sum_{1\leqslant q \leqslant M} \frac{\mu^2}{(q\mu+n)^{N'}(q\mu+m)}.
\end{align*}
Hence,
\begin{align*}
\zeta^\mu_*(N',1;n,m)=\frac{\mu^2}{n^{N'}m} +\zeta^\mu_*(N',1;n,\mu+m) +f(m) \quad \text{as} \quad M\to\infty,
\end{align*}
where by Lemma \ref{lem:parFrac},
\begin{align*}
f(m) =&\sum_{1\leqslant q} \frac{\mu}{(n-m)^{N'}}\left(\frac{\mu}{q\mu+m}-\frac{\mu}{q\mu+n}\right)
-\sum_{1\leqslant q} \sum_{i=0}^{N-3} \frac{\mu^2}{(n-m)^{i+1}} \frac{1}{(q\mu+n)^{N'-i}} \\
=&\mu\frac{\calD^\mu(1;m+\mu;n+\mu)}{(n-m)^{N'}}
 -\mu \sum_{\substack{i+k=N-2\\ i\geqslant 1,k\geqslant 0}} \frac{\zeta^\mu(k+2;n)}{(n-m)^{i}}
 +\sum_{i=0}^{N-3} \frac{\mu^2}{(n-m)^{i+1}n^{N'-i}} .
\end{align*}
Thus
\begin{align*}
\lim_{m\to 0^+} f(m)= &\, \mu\frac{\calD^\mu(1;\mu;n+\mu)}{n^{N'}}
-\mu \sum_{\substack{i+k=N-2\\ i\geqslant 1,k\geqslant 0}} \frac{\zeta^\mu(k+2;n)}{n^{i}}
+\mu^2 \frac{N-2}{n^{N}}.
\end{align*}
By Theorem~\ref{thm:m=nSumFormula} we see that
\begin{align*}
\sum_{j+k=N,k\geqslant 2} \zeta^\mu(j,k;n,n)
=&\, \mu N' \zeta^\mu(N;n) +\calD^\mu(N',1;n,\mu;n,n) \\
&\,+ \lim_{m\to 0^+} \left(\frac{\mu^2}{n^{N'}m}-\frac{\mu^2}{m(n-m)^{N'}}+\frac{\mu^2}{n(n-m)^{N'}}\right) +\mu^2 \frac{N-2}{n^{N}}\\
=&\, \mu N' \zeta^\mu(N;n)+\calD^\mu(N',1;n,\mu;n,n),
\end{align*}
as desired.
\end{proof}

\begin{rem} When $\mu=n=1$ \eqref{thm:m=0SumFormula} reduces to the classical
sum formula for double zeta \emph{star} values:
\begin{align*}
 \sum_{j+k=N,k\geqslant 2} \zeta^\star(j,k)=N\zeta(N)
\end{align*}
since for all $j+k=N$ ($j>0,k>1$) we always have
\begin{equation*}
\zeta^\star(j,k) :=\sum_{0<n_1\leqslant n_2} \frac{1}{n_1^j n_2^k} =\zeta(N)+\zeta(j,k).
\end{equation*}
Further, when $\mu=2$ and $n=1$ \eqref{thm:m=0SumFormula} reduces to the
sum formula for double $t$-values
\begin{align*}
\sum_{j+k=N,k\geqslant 2} t(j,k)=&\, 2\Big(\zeta(N-1,\bar1)-\zeta(\overline{N-1},\bar1)\Big),
\end{align*}
which also follows from \cite[Theorem~4.2]{BCJXXZhao2020c}, where the terms on the
right-hand side are Euler sums defined by \eqref{equ:EulerSums}.
\end{rem}

\begin{thm}\emph{(Weighted Sum Formula)}\label{thm:weightedSumFormula}
For any $n\in\Q_{>0}$ and $N\in \N_{\geqslant 3}$, we have
\begin{align*}\label{equ:wtSum_m=n}
 \sum_{j+k=N, j\geqslant 1, k\geqslant 2} 2^k \zeta^\mu(j,k;n,2n)=&\, \mu (N-1)\zeta^\mu(N;n)
 +2\calD^\mu(N-1,1;n,\mu;n,2n).
\end{align*}
\end{thm}

\begin{proof}
Taking $\gs=\gt=1$ in Corollary~\ref{cor:DblGenFunc}, we see that
\begin{equation*}
\sum_{j+k=N,j,k\geqslant 1} 2^k\zeta^\mu_\shuffle(j,k;n,2n)
= -\mu N'\zeta^\mu(N;n) +2\sum_{j+k=N,j,k\geqslant 1} \zeta^\mu_*(j,k;n,n).
\end{equation*}
The theorem now follows from Theorem~\ref{thm:m=0SumFormula} easily.
\end{proof}

Taking $\mu=n=1$ in Theorem~\ref{thm:weightedSumFormula} we recover
the classical weighted sum formula
\begin{align*}\label{equ:wtSum_m=n}
 \sum_{j+k=N, j\geqslant 1, k\geqslant 2} 2^k \zeta^\mu(j,k)=(N+1)\zeta^\mu(N).
\end{align*}
Taking $\mu=2$ and $n=1$ in Theorem~\ref{thm:weightedSumFormula} we recover
the weighted sum formula for double $T$-values
\begin{align*}
 \sum_{j+k=N, j\geqslant 1, k\geqslant 2} 2^k T(j,k)=2(N-1)T(N)
\end{align*}
where
\begin{equation*}
T(N)=\sum_{0<n, 2\nmid n} \frac{2}{n^N}, \quad\text{and} \quad
T(j,k)=\sum_{0<n<m, 2\nmid n, 2|m } \frac{4}{n^j m^k}.
\end{equation*}
See \cite[Theorem~4.4]{BCJXXZhao2020c} or \cite[Theorem~3.2]{KanekoTs2020}.

\begin{rem}
By considering a two variable variant of multiple polylogarithms, Kaneko and Tsumura \cite{KanekoTs2014a}
discovered another
approach to generalizing the weighted sum formulas as above simultaneously.
\end{rem}

\section{The universal algebra and a conjecture}

We now introduce the regularized double shuffle relations. We first recall the two commutative diagrams
\begin{center}
\begin{tikzpicture}[scale=0.9]
\node (A) at (0,2) {$(\Q[\mu]\langle Y\rangle,\must)$};
\node (B) at (-2.5,0) {$(\Q[\mu]\langle Y\rangle^0,\must)[T]$};
\node (C) at (2.5,0) {$\R[T]$};
\draw[->] (A) to node[pos=.8,above] {$\psi_{*,m}^{-1}\ \ $}(B);
\draw[->] (A) to node[pos=.8,above] {$\ \ \zeta^{\mu,T}_{*,m}$}(C);
\draw[->] (B) to node[pos=.55,below] {$\zeta^{\mu}_{*}:T\mapsto T$}(C);
\end{tikzpicture}
\begin{tikzpicture}[scale=0.9]
\node (A) at (0,2) {$(\Q[\mu]\langle X^{\mu}\rangle^1,\shuffle)$};
\node (B) at (-2.5,0) {$(\Q[\mu]\langle X^{\mu}\rangle^0 ,\shuffle)[T]$};
\node (C) at (2.5,0) {$\R[T]$};
\draw[->] (A) to node[pos=.8,above] {$\psi_{\shuffle,m}^{-1}\ \ $}(B);
\draw[->] (A) to node[pos=.8,above] {$\ \ \zeta^{\mu,T}_{\shuffle,m}$}(C);
\draw[->] (B) to node[pos=.55,below] {$\zeta^{\mu}_{\shuffle}:T\mapsto T$}(C) ;
\end{tikzpicture}
\end{center}

\begin{defn}
Let $(R,\cdot)$ be a $\Q[\mu]$-algebra, $Z^R_{*}:\Q[\mu]\langle Y\rangle^0\longrightarrow R$ and $Z^R_{\shuffle}:\Q[\mu]\langle X^{\mu}\rangle^0 \longrightarrow R$ are two $\Q[\mu]$-linear maps. We say that $(R,Z^R_{*},Z^R_{\shuffle})$ satisfies the finite double shuffle relations if $Z^R_{*}$ is an algebra homomorphism $Z^R_{*}:(\Q[\mu]\langle Y\rangle^0,\must)\longrightarrow(R,\cdot)$, $Z^R_{\shuffle}$ is an algebra homomorphism $Z^R_{\shuffle}:(\Q[\mu]\langle X^{\mu}\rangle^0 ,\shuffle)\longrightarrow(R,\cdot)$, and the following diagram commutes
\begin{center}
\begin{tikzpicture}[scale=0.9]
\node (A) at (-2,2) {$(\Q[\mu]\langle Y\rangle^0,\must)$};
\node (B) at (2,2) {$(\Q[\mu]\langle X^{\mu}\rangle^0 ,\shuffle)$};
\node (C) at (0,0) {$(R,\cdot)$};
\draw[->] (A) to node[pos=.5,above] {${\varphi_{\mu}}$}(B);
\draw[->] (A) to node[pos=.2,below] {$Z^R_{*}\ $}(C);
\draw[->] (B) to node[pos=.2,below] {$Z^R_{\shuffle}$}(C);
\end{tikzpicture}
\end{center}
\end{defn}

Composing $Z^R_{*}$ (resp.\ $Z^R_{\shuffle}$) with the map $\psi^{-1}_{*,m}$ (resp.\ $\psi^{-1}_{\shuffle,m}$)
(see Theorem~\ref{thm:stRegPsi} and Theorem~\ref{thm:shaRegPsi}), we obtain two ways to regularize
\begin{align*}
Z^{R,T}_{*,m}&:(\Q[\mu]\langle Y\rangle,\must)\longrightarrow (R,\cdot)[T],\\
Z^{R,T}_{\shuffle,m}&:(\Q[\mu]\langle X^{\mu}\rangle^1,\shuffle)\longrightarrow (R,\cdot)[T].
\end{align*}
Next, we define the comparing maps $\myrho^{\mu,*,R}_{m,m'},\myrho^{\mu,\shuffle,R}_{m,m'}$ and $\rho_{m,m'}^{\mu,R}$ by
\begin{align*}
\myrho^{\mu,*,R}_{m,m'}:(R,\cdot)[T]&\longrightarrow (R,\cdot)[T]\\
T^s&\longmapsto(T+Z^R_{*}(y_{1,m}-y_{1,m'}))^s,\\
\myrho^{\mu,\shuffle,R}_{m,m'}:(R,\cdot)[T]&\longrightarrow(R,\cdot)[T]\\
T^s&\longmapsto(T+Z^R_{\shuffle}(y^{\mu}_{m}-y^{\mu}_{m'}))^s,
\end{align*}
and
\begin{align*}
\rho_{m,m'}^{\mu,R}:(R,\cdot)[T]&\longrightarrow(R,\cdot)[T]\\
T^s&\longmapsto\sum_{i=0}^sa^{(s)}_iT^i,
\end{align*}
where
$$\varphi_{\mu}(\overbrace{y_{1,m}\must \cdots\must y_{1,m}}^s)=w_{s}^{(s)}\shuffle\overbrace{y^{\mu}_{m'}\shuffle\cdots\shuffle y^{\mu}_{m'}}^s+w_{s-1}^{(s)}\shuffle\overbrace{y^{\mu}_{m'}\shuffle\cdots\shuffle y^{\mu}_{m'}}^{s-1}+\cdots+w_1^{(s)}\shuffle y^{\mu}_{m'}+w_0^{(s)}$$
and $a_i^{(s)}=Z_{\shuffle}^R(w_{i}^{(s)}),$ for all $0\leqslant i\leqslant s$.

\begin{defn}
Assume that $(R,Z_{*}^R,Z_{\shuffle}^R)$ satisfies the finite double shuffle relations. We say $(R,Z_{*}^R,Z_{\shuffle}^R)$ satisfies the regularized double shuffle relations if, in addition, for all $w\in\Q[\mu]\langle Y\rangle$, one has
\begin{align*}
Z^{R,T}_{*,m'}(w)&=\myrho^{\mu,*,R}_{m,m'}(Z^{R,T}_{*,m}(w)),\\
Z^{R,T}_{\shuffle,m'}(\varphi_{\mu}(w))&=\myrho^{\mu,\shuffle,R}_{m,m'}(Z^{R,T}_{\shuffle,m}(\varphi_{\mu}(w))),\\
Z^{R,T}_{\shuffle,m'}(\varphi_{\mu}(w))&=\rho_{m,m'}^{\mu,R}(Z^{R,T}_{*,m}(w)).
\end{align*}
\end{defn}

Combining Theorem~\ref{thm:stComparisonMap}, Theorem~\ref{thm:shaComparisonMap} and Theorem~\ref{thm:mixedComparisonMap}, we obtain the main result of this section.

\begin{thm}
The pair $(\R,\zeta^{\mu}_{*},\zeta^{\mu}_{\shuffle})$ satisfies the regularized double shuffle relations.
\end{thm}

Let $(R_{\text{RDS}},Z_{*}^{\text{RDS}},Z_{\shuffle}^{\text{RDS}})$ be the universal algebra satisfying the regularized double shuffle relations. This means that, for any $(R,Z_{*}^R,Z_{\shuffle}^R)$ satisfying the regularized double shuffle relations, there exists a pair of unique map $(\varphi^R_{*},\varphi^R_{\shuffle})$ such that the following diagram commutes
\begin{center}
\begin{tikzpicture}[scale=0.9]
\node (A) at (0,2) {$(\Q[\mu]\langle Y\rangle,\must)$};
\node (B) at (-2,0) {$(R_{\text{RDS}},\cdot)$};
\node (C) at (2,0) {$(R,\cdot)$};
\draw[->] (A) to node[pos=.8,above] {$Z^{\text{RDS}}_{*}\ \ $}(B);
\draw[->] (A) to node[pos=.8,above] {$\ \ Z^{R}_{*}$}(C);
\draw[->] (B) to node[pos=.55,below] {$\varphi^R_{*}$}(C);
\end{tikzpicture} $\qquad$
\begin{tikzpicture}[scale=0.9]
\node (A) at (0,2) {$(\Q[\mu]\langle X^{\mu}\rangle^1,\shuffle)$};
\node (B) at (-2,0) {$(R_{\text{RDS}},\cdot)$};
\node (C) at (2,0) {$(R,\cdot)$};
\draw[->] (A) to node[pos=.8,above] {$Z^{\text{RDS}}_{\shuffle}\ \ $}(B);
\draw[->] (A) to node[pos=.8,above] {$\ \ Z^{R}_{\shuffle}$}(C);
\draw[->] (B) to node[pos=.55,below] {$\varphi^R_{\shuffle}$}(C);
\end{tikzpicture}
\end{center}

The following conjecture describes the combinatorial structure of the algebra of $\mu$-MHZVs.

\begin{prob}
Is the map $(\varphi_{*}^{\R},\varphi_{\shuffle}^{\R})$ always injective? Equivalently, is the algebra $(\calZ^{\mu},\cdot)$ of $\mu$-MHZVs isomorphic to $(R_{\text{RDS}},\cdot)$?
\end{prob}

When $\mu=1$ and $m$'s are positive integers,
 we expect the answer to be affirmative since this is essentially the level one case
where we have the well-known conjecture by Ihara, Kaneko, and Zagier \cite{ikz}. But when
$\mu=2$ and $m$'s are positive integers, we are at level two and we know the double shuffle
relations do not generate all $\Q$-linear
relations among Euler sums (see the remark after \cite[Theorem~1.1]{Zhao2010a}. However, there is
some subtle difference between $\mu$-MHZVs and the Euler sums so we do not know the answer to the
problem in this case.

\end{document}